\documentclass[10pt]{amsart}

\usepackage{anysize}
\usepackage[utf8x]{inputenc}
\usepackage[english]{babel}
\usepackage{amssymb, amsmath, amsthm}
\usepackage{enumitem}
\usepackage{mathtools}
\usepackage{url}
\usepackage[all,cmtip]{xy}
\usepackage[hidelinks,bookmarksnumbered]{hyperref}
\usepackage{tikz-cd}

\usepackage{cleveref} 

\let\cal\mathcal
\def\AA{{\cal A}}
\def\BB{{\cal B}}
\def\CC{{\cal C}}
\def\DD{{\cal D}}
\def\EE{{\cal E}}
\def\FF{{\cal F}}
\def\GG{{\cal G}}

\def\II{{\cal I}}

\def\OO{{\cal O}}
\def\PP{{\cal P}}
\def\QQ{{\cal Q}}

\def\SS{{\cal S}}
\def\TT{{\cal T}}
\def\UU{{\cal U}}

\def\XX{{\cal X}}
\def\YY{{\cal Y}}

\let\blb\mathbb
\def\bC{{\blb C}}

\def\bX{{\blb X}}

\def\bP{{\blb P}}

\def\bZ{{\blb Z}}

\def\bZ{{\blb Z}}

\let\frak\mathfrak

\def\fe{\frak{e}}
\def\ff{\frak{f}}

\def\tt{\frak{t}}

\DeclareMathOperator{\Tr}{Tr}

\DeclareMathOperator{\gr}{gr}

\newcommand{\Yoneda}{\mathbb{Y}}

\DeclareMathOperator{\Coh}{Coh}
\DeclareMathOperator{\QCoh}{QCoh}
\DeclareMathOperator{\Gr}{Gr}
\DeclareMathOperator{\Tor}{Tor}
\DeclareMathOperator{\QGr}{QGr}

\title{Auslander's formula and correspondence for exact categories}
\author{Ruben Henrard}
\address{Ruben Henrard, Hasselt University, Campus Diepenbeek, Department WNI, 3590 Diepenbeek, Belgium}
\email{ruben.henrard@uhasselt.be}

\author{Sondre Kvamme}
\address{Sondre Kvamme, Department of Mathematical Sciences, Norwegian University of Science and Technology, 7491 Trondheim, Norway}
\email{sondre.kvamme@ntnu.no}

\author{Adam-Christiaan van Roosmalen}
\address{Adam-Christiaan van Roosmalen, Hasselt University, Campus Diepenbeek, Department WNI, 3590 Diepenbeek, Belgium}
\email{adamchristiaan.vanroosmalen@uhasselt.be}

\newtheorem{theorem}{Theorem}[section]
\newtheorem{proposition}[theorem]{Proposition}
\newtheorem{lemma}[theorem]{Lemma}
\newtheorem{corollary}[theorem]{Corollary}

\theoremstyle{definition}
\newtheorem{definition}[theorem]{Definition}
\newtheorem{remark}[theorem]{Remark}
\newtheorem{example}[theorem]{Example}
\newtheorem{construction}[theorem]{Construction}

\DeclareMathOperator{\inflation}{\rightarrowtail}
\DeclareMathOperator{\deflation}{\twoheadrightarrow}

\DeclareMathOperator{\cogen}{cogen}
\DeclareMathOperator{\gen}{gen}
\DeclareMathOperator{\gldim}{gl.dim}

\DeclareMathOperator{\im}{im}
\DeclareMathOperator{\coim}{coim}
\DeclareMathOperator{\coker}{coker}

\DeclareMathOperator{\Ob}{Ob}

\DeclareMathOperator{\Hom}{Hom}
\DeclareMathOperator{\End}{End}
\DeclareMathOperator{\Ext}{Ext}
\DeclareMathOperator{\Proj}{Proj}
\DeclareMathOperator{\Inj}{Inj}

\DeclareMathOperator{\Ex}{\boldsymbol{\mathsf{Ex}}}
\DeclareMathOperator{\AEx}{\boldsymbol{\mathsf{AEx}}}

\DeclareMathOperator{\IEx}{\boldsymbol{\mathsf{IEx}}}

\DeclareMathOperator{\Mor}{Mor}
\DeclareMathOperator{\Fun}{Fun}

\DeclareMathOperator{\Mod}{Mod}
\DeclareMathOperator{\smod}{mod}
\DeclareMathOperator{\smodad}{mod_{\mathsf{adm}}}
\DeclareMathOperator{\modad}{mod_{\mathsf{adm}}}
\DeclareMathOperator{\coh}{coh}
\DeclareMathOperator{\eff}{eff}
\DeclareMathOperator{\Weff}{Eff}
\DeclareMathOperator{\domdim}{dom.dim}
\DeclareMathOperator{\add}{add}
\DeclareMathOperator{\rep}{rep}

\DeclareMathOperator{\Abelian}{\boldsymbol{\mathsf{Ab}}}
\DeclareMathOperator{\AAb}{\boldsymbol{\mathsf{AAb}}}

\DeclareMathOperator{\Ab}{\mathsf{Ab}}

\DeclareMathOperator{\Acb}{\textbf{Ac}^{\mathsf{b}}}
\DeclareMathOperator{\Kb}{\mathsf{K}^{\mathsf{b}}}
\DeclareMathOperator{\Db}{\mathsf{D}^{\mathsf{b}}}
\DeclareMathOperator{\DAb}{\mathsf{D}^{\mathsf{b}}_{\mathcal{A}}}

\renewcommand{\Upsilon}{\mathbb{Y}}
\renewcommand{\mod}{\operatorname{mod}}

\newcommand{\ic}[1]{{#1}_{\text{ic}}}

\makeatletter
\newcommand{\myitem}[1]{%
\item[#1]\protected@edef\@currentlabel{#1}%
}
\makeatother


\begin{document}


\begin{abstract}
The Auslander correspondence is a fundamental result in Auslander-Reiten theory. In this paper we introduce the category $\smodad(\EE)$ of admissibly finitely presented functors and use it to give a version of Auslander correspondence for any exact category $\EE$. An important ingredient in the proof is the localization theory of exact categories. We also investigate how properties of $\EE$ are reflected in $\smodad(\EE)$, for example being (weakly) idempotent complete or having enough projectives or injectives. Furthermore, we describe $\smodad(\EE)$ as a subcategory of $\smod(\EE)$ when $\EE$ is a resolving subcategory of an abelian category. This includes the category of Gorenstein projective modules and the category of maximal Cohen-Macaulay modules as special cases. Finally, we use $\smodad(\EE)$ to give a bijection between exact structures on an idempotent complete additive category $\CC$ and certain resolving subcategories of $\smod(\CC)$.
\end{abstract}

\maketitle

\tableofcontents
\section{Introduction}

To study representations of finite-dimensional algebras, Auslander advocated a functorial point of view, which he developed in \cite{Auslander65,Auslander78,AuslanderBridger69,AuslanderReiten74,AuslanderReiten78}. This means that to study the category $\smod(\Lambda)$ of finite-dimensional modules over a finite-dimensional algebra $\Lambda$, one should study the category $\smod (\smod (\Lambda))$ of finitely presented functors $F\colon \smod(\Lambda)^{\operatorname{op}}\to \operatorname{Ab}$ to abelian groups, and translate the results back to $\smod (\Lambda)$.  This can be done using \emph{Auslander's formula} (see \cite{Auslander65}, the term was coined in \cite{Lenzing98}):
\[\frac{\smod(\smod(\Lambda))}{\eff(\mod(\Lambda))} \simeq \mod(\Lambda)\text{,}\]
where $\eff(\mod(\Lambda))$ is the full subcategory of effaceable functors, that is, objects of $\mod(\mod(\Lambda))$ which only admit zero morphisms to representable functors. Auslander's formula still holds when one replaces the category $\mod(\Lambda)$ by an arbitrary (essentially small) abelian category $\AA$. 

When $\Lambda$ is representation-finite, the category $\smod(\smod(\Lambda))$ is equivalent to $\smod(\Gamma)$ for some finite-dimensional algebra $\Gamma$ (here, $\Gamma$ is the endomorphism ring of an additive generator of $\smod(\Lambda)$).  The class of algebras $\Gamma$ that can occur in this way is equal to the class of Auslander algebras, i.e. the algebras satisfying 
\[
\domdim \Gamma \geq 2\geq \gldim \Gamma
\]
where $\domdim \Gamma$ is the dominant dimension (see \Cref{definition:DominantDimension}) and $\gldim \Gamma$ is the global dimension of $\Gamma$.  In fact, the association $\Lambda \mapsto \Gamma$ gives a bijection between Morita equivalences classes  of representation-finite algebras and Morita equivalence classes of Auslander algebras.  This bijection is called the \emph{Auslander correspondence}. 

In recent years, several variations of Auslander's formula \cite{AsadollahiAsadollahiHafeziRasool20,Fiorot20,Krause15,Ogawa19k,Ogawa19} and the Auslander correspondence \cite{AuslanderSolberg93a,Auslander71,Beligiannis00,Enomoto18,Hanihara20,Iyama07,IyamaSolberg18} have been discussed.  

In this paper, we generalize Auslander's formula and the Auslander correspondence to the setting of exact categories, in the sense of Quillen. An exact category is an additive category together with a chosen class of short exact sequences, called conflations, which satisfy certain axioms, see  \Cref{Definition:AxiomsR0-R2AndL0-L2}. Exact categories occur in many places in mathematics.  In fact, due to the Gabriel-Quillen embedding, we know that exact categories are precisely the extension-closed subcategories of abelian categories.  However, there is an abundance of exact categories without any obvious abelian categories around to resort to. Examples from functional analysis include the categories of Banach spaces, barrelled spaces, Schwartz spaces, Montel spaces, and Fr\'{e}chet spaces (see \cite{PerezCarrerasBonet87,RoelckeDierolf81}).  The category of locally free coherent sheaves on a scheme is exact.  The category of filtered modules over a filtered ring is an exact category \cite{SchapiraSchneiders16}. For quasi-hereditary algebras \cite{ClineParshallScott88,Scott87}, the modules filtered by standard modules is an exact subcategory which plays an important role. Quasi-hereditary algebras occur for example in Lie theory when considering blocks of the BGG category $\OO$, see \cite{Humphreys08}. We refer to \cite{Buhler10} for a comprehensive treatment of exact categories.

In addition, many categories of interest in the representation theory of artin algebras are exact, for example the Gorenstein projective modules \cite{AuslanderReiten91,Beligiannis11,Chen08,ChenShenZhou18,Ringel13}, or the modules of finite projective dimension \cite{AuslanderReiten91a,HappelUnger96,MarcosMerklenPlatzeck00,PlatzeckReiten01}. Furthermore, Cohen-Macaulay representations \cite{Auslander78,Auslander86,BuchweitzGreuelSchreyer87,Carlson83,IyamaWemyss14,Knorrer87,LeuschkeWiegand12,ReitenVandenBergh89,Yoshino90}, their graded counterparts \cite{AuslanderReiten87,AuslanderReiten89}, and their connection to McKay quivers \cite{ArtinVerdier85,Auslander86a} have also been actively studied from a representation theoretic point of view. 


When replacing $\mod(\Lambda)$ by an exact category $\EE$ in Auslander's formula, two obstructions become clear.  The first obstruction is that $\mod(\EE)$ is independent of the exact structure.  To remedy this, we introduce a different candidate to study instead of $\smod(\EE)$, which \emph{takes into account the exact structure of} $\EE$.

\begin{definition}[\Cref{Definition:smodadAndeff}]
Let $\EE$ be an exact category. We define $\smodad(\EE)$ to be the full subcategory of $\smod(\EE)$ consisting of those functors $F$ that admit a projective representation
		\[\Hom_{\EE}(-,X)\xrightarrow{\Hom_{\EE}(-,f)}\Hom_{\EE}(-,Y)\to F\to 0\]
		where $f\colon X\to Y$ is an admissible morphism in $\EE$, that is, $f = m \circ p$ where $m$ is an inflation and $p$ is a deflation.  A functor $F \in \modad(\EE)$ is called an \emph{admissibly presented} functor.
\end{definition}

The second obstruction is that $\modad(\EE)$ need not be abelian; consequently, the quotient construction that occurs in Auslander's formula cannot be the quotient described in \cite{GabrielZisman67,Gabriel62}.  In this paper, we use the quotient theory for exact categories described in \cite{Cardenas98,HenrardVanRoosmalen19b,HenrardVanRoosmalen19a}, which allows to take a quotient of exact categories by a so-called percolating subcategory (we refer to section \ref{Subsection:PercolatingSubcategories} for the precise definitions).  Given a percolating subcategory $\AA$ of an exact category $\EE'$, we want to construct an exact functor $Q\colon \EE' \to \EE' / \AA$ satisfying the usual 2-universal property of a quotient; in particular, $Q(\AA) = 0$.  Here, exact means that $Q$ maps conflation in $\EE'$ to conflations in $\EE' / \AA$.  To describe the quotient, consider a conflation $X \stackrel{f}{\inflation} Y \stackrel{g}{\deflation}Z$ in $\EE'$.  If $X \in \AA$, then $Q(X) = 0$.  As $Q(X) \stackrel{Q(f)}{\inflation} Q(Y) \stackrel{Q(g)}{\deflation}Q(Z)$ is a conflation, and hence a kernel-cokernel pair, we find that $Q(g)$ is invertible in $\EE' / \AA$.  Likewise, if $Z \in \AA$, then $Q(f)$ is invertible.  Let $\Sigma$ be the set of morphisms in $\EE'$ which can be written as a composition of a deflation with kernel in $\AA$ and an inflation with cokernel in $\AA$. It is shown in \cite{HenrardVanRoosmalen19a} that $\EE'[\Sigma^{-1}]$ is an exact category and that the localization $\EE' \to \EE'[\Sigma^{-1}]$ satisfies the 2-universal property of the quotient functor $Q\colon \EE' \to \EE' / \AA$.  

Having introduced the category $\modad(\EE)$, we can define the subcategory $\eff(\EE)$ of effaceable functors as the subcategory of functors in $\modad(\EE)$ which only admit zero morphisms to representable functors. We can now show Auslander's formula.

\begin{theorem}[\Cref{Proposition:EffAreTorsion} and \Cref{Theorem:AuslandersFormulaForExactCategories}]
The following holds:
\begin{enumerate}
\item $\eff(\EE)$ is a percolating subcategory of $\smodad(\EE)$.  Hence, the quotient $\smodad(\EE)/\eff(\EE)$ exists (in the category of exact categories).
\item $\smodad(\EE)/\eff(\EE)$ is equivalent to $\EE$ as an exact category.
\end{enumerate}
\end{theorem}

Next, we characterize the categories of the form $\smodad(\EE)$, analogous to the characterization of Auslander algebras. For this, we introduce the notion of an Auslander exact category. In the following, we write (for a subcategory $\PP \subseteq \EE$)
\begin{align*}
& {}^{\perp}\PP\coloneqq\{E\in \EE\mid \Hom_{\EE}(E,\PP)=0\} \quad \text{and} \quad \cogen(\PP)\coloneqq\{E\in \EE\mid \exists\text{ inflation }E\inflation P \text{ with }P\in \PP\}. 
\end{align*}

\begin{definition}[\Cref{AuslanderExactCategories}]
Let $\EE$ be an exact category with enough projectives $\PP$. We say that $\EE$ is an \emph{Auslander exact category} if it satisfies the following:
\begin{enumerate}
		\item[(i)] $({}^{\perp}\PP,\cogen(\PP))$ is a torsion pair in $\EE$.
		\item[(ii)] If $E'\in {}^{\perp}\PP$, then any morphism $f\colon E\to E'$ is admissible with image in ${}^{\perp}\PP$.
		\item[(iii)] $\Ext^1_{\EE}({}^{\perp}\PP,\PP)=0$.
		\item[(iv)] $\gldim(\EE)\leq 2$.
	\end{enumerate}
\end{definition}

\begin{theorem}[\Cref{FirstAuslanderCorrespondence}] The following hold:
\begin{enumerate}
\item If $\EE$ is an exact category, then $\smodad(\EE)$ is an Auslander exact category.
\item If $\EE'$ is an Auslander exact category, then there exists an exact equivalence $\EE'\cong \smodad(\EE)$ where $\EE$ is an exact category. Furthermore, $\EE$ is uniquely determined up to exact equivalence.
\end{enumerate}
\end{theorem}
In fact, we show that the association $\EE\mapsto \smodad(\EE)$ is an equivalence from the $2$-category of exact categories and left exact functors to the $2$-category of Auslander exact categories and exact functors preserving projective objects. It can also be interpreted as a left adjoint to an inclusion of $2$-categories, see \Cref{Corollary:Left2Adjoint}. 

Note that if $\QQ$ is the subcategory of projective objects in $\smodad(\EE)$, then we have the following identifications:
\begin{enumerate}
\item $F\in \QQ$ if and only if $F$ is isomorphic to a representable functor.
\item $F\in {}^{\perp}\QQ$ if and only if $F$ is an effaceable functor.  Effaceable functors can also be characterized as those functors $F$ having a projective representation
		\[\Hom_{\EE}(-,X)\xrightarrow{\Hom_{\EE}(-,f)}\Hom_{\EE}(-,Y)\to F\to 0\]
		where $f\colon X\to Y$ is a deflation.  Such functors have already been studied in the context of exact categories by several authors under the name \emph{defects}, see \cite{Buan01,Enomoto18,FangGorsky20}.
		\item $F\in \cogen(\QQ)$ if and only if there exist a projective representation
		\[\Hom_{\EE}(-,X)\xrightarrow{\Hom_{\EE}(-,f)}\Hom_{\EE}(-,Y)\to F\to 0\]
		where $f\colon X\to Y$ is an inflation.
\end{enumerate}
It follows that the subcategory of effaceable functors is the torsion part of a torsion pair in $\smodad(\EE)$.

If $\EE$ has enough injectives, then we obtain a characterization of $\smodad(\EE)$ which is similar to module categories of Auslander algebras.

\begin{theorem}[\Cref{Theorem:AuslanderSecond}]
Let $\EE'$ be an exact category with enough projectives $\PP'$.  Then there exists an exact equivalence $\EE'\cong \smodad(\EE)$ where $\EE$ is an exact category with enough injectives if and only if the following hold:
\begin{enumerate}
		\item $\domdim (\EE')\geq 2\geq \gldim (\EE')$.
		\item Any morphism $X\to E$ with $E\in{}^{\perp}\PP'$ is admissible.
		\item For any $E\in \EE'$ there exists an admissible left $\PP'$-approximation $E\to P$.
\end{enumerate}	
Furthermore, in this case $\EE'$ has enough injectives.	
\end{theorem}
We also analyse the properties of $\smodad(\EE)$ when $\EE$ has enough projectives, see Section \ref{subsection:EnoughProjectives}.  In this case, we find that Auslander's formula can be extended to a recollement connecting $\EE$ and $\modad(\EE)$, see \Cref{Proposition:TTFTripleYieldsRecollement}. 

One of our main examples of an exact category $\EE$ is the category of Cohen-Macaulay modules, and it turns out that $\smodad(\EE)$ has a simple description in this case. Let $R$ be a commutative Cohen-Macaulay local ring and let $\Lambda$ be a Cohen-Macaulay-finite $R$-order, i.e. a noetherian $R$-algebra which is maximal Cohen-Macaulay as an $R$ -module. Let $M$ be an additive generator of the subcategory $\operatorname{CM}(\Lambda)$ of maximal Cohen-Macaulay modules containing $\Lambda$ as a summand. Let $\Gamma=\End_{\Lambda}(M)$ and let $e$ be the idempotent corresponding to the direct summand $\Lambda$. Then we have 
\[
\smodad(\Gamma)=\{M\in \smod(\Gamma)\mid Me\in \operatorname{CM}(\Lambda)\}.
\]
Similar descriptions hold for any resolving subcategory of a module category, see \Cref{Corollary:SubcategoriesAbelian}.  Here, a \emph{resolving subcategory} $\EE$ of an exact category $\AA$ is a subcategory which is generating (in the sense that for every object $A \in \AA$, there is a deflation $E \deflation A$ with $E \in \EE$), closed under extensions, direct summands, and kernels of deflations. 

The study of exact structures on a fixed idempotent complete additive category $\CC$ has recently attracted a lot of interest \cite{BaillargeonBrustleGorskyHassoun20,BrustleHassounTattar20,BrustleHassounLangfordRoy20,Enomoto18,FangGorsky20,Rump11,Rump15,Rump20}. In particular, they form a lattice \cite{BrustleHassounLangfordRoy20,Rump11} and  the association $\EE\mapsto \eff(\EE)$ gives a bijection to certain Serre subcategories of $\smod(\EE)$ \cite{Buan01,Enomoto18,FangGorsky20}. We obtain a similar characterization by replacing $\eff(\EE)$ with $\smodad(\EE)$ and Serre subcategories by resolving subcategories. In the following, $\PP^2(\CC)$ denotes the subcategory of $\smod(\CC)$ consisting of functors $F$ admitting a projective resolution
\[
0\to \Hom_{\EE}(-,X)\to \Hom_{\EE}(-,Y)\to \Hom_{\EE}(-,Z)\to F\to 0.
\]
Also, $\underline{\smod}(\CC)$ denote the stable category of $\smod(\CC)$ modulo projectives, and $\operatorname{Tr}\colon \underline{\smod}(\CC)\to \underline{\smod}(\CC^{\operatorname{op}})^{\operatorname{op}}$ denotes the Auslander-Bridger transpose \cite{AuslanderBridger69}. For a subcategory $\XX$ of $\smod(\CC)$ we let $\operatorname{Tr}(\XX)$ denote the subcategory of $\smod(\CC^{\operatorname{op}})$ consisting of all functors $F$ which are isomorphic in $\underline{\smod}(\CC^{\operatorname{op}})$ to $\operatorname{Tr}(F)$ for some $F\in \XX$. For the notion of grade, see \Cref{Definition:grade}.

\begin{theorem}[\Cref{Theorem:ExactStructuresResolving}]\label{Theorem:IntroductionExactStructuresResolving}
Let $\CC$ be an idempotent complete additive category. The association $\EE\mapsto \smodad(\EE)$ gives a bijection between the following:
\begin{enumerate}
\item Exact structures $\SS$ on $\CC$, where $\EE=(\CC,\SS)$ is the corresponding exact category.
\item Subcategories $\XX$ of $\smod(\CC)$ satisfying the following:
\begin{enumerate}
\item $\XX$ is a resolving subcategory of $\PP^2(\CC)$ and $\operatorname{Tr(\XX)}$ is a resolving subcategory of $\PP^2(\CC^{\operatorname{op}})$, and
\item $\XX$ and $\operatorname{Tr}(\XX)$ have no objects of grade $1$.
\end{enumerate}
\end{enumerate}
\end{theorem}

It was shown in \cite{ButlerHorrocks61,DraxlerReitenSmaloSolberg99} that exact structures on an abelian category $\AA$ correspond to closed subfunctors of $\Ext^1_{\AA}(-,-)$. Such structures are one of the main objects of study in relative homological algebra \cite{EilenbergMoore65,EnochsJenda00}, see also \cite{AuslanderSolberg93a,AuslanderSolberg93d,AuslanderSolberg93c,AuslanderSolberg93b} for the case of Artin algebras. We obtain the following corollary for abelian categories. 

\begin{corollary}[\Cref{Corollary:ExactStructuresResolving}]
Let $\AA$ be an abelian category. Then there exists a bijection between the following:
\begin{enumerate}
\item Exact structures on $\AA$.
\item Resolving subcategories $\XX$ of $\smod (\AA)$ for which $\operatorname{Tr}(\XX)$ is a resolving subcategory of $\smod (\AA^{\operatorname{op}})$. 
\end{enumerate}
\end{corollary}

Explicit examples will be given in \Cref{example:MainTheorem1,example:MainTheorem2}.

\subsection*{Structure of the paper}

In Section \ref{section:Localizations} we recall the quotients of exact categories that occur in the formulation of Auslander's formula. In particular, we show that localization by a two-sided admissible percolating subcategory reflects admissible morphisms, see \Cref{Theorem:QReflectsAdmissibles}. We end by investigating torsion theory on exact categories where the torsion part is a two-sided admissible percolating subcategory.

In Section \ref{section:Formula} we study the category $\smodad(\EE)$. In particular, we show in \Cref{Theorem:Universal property} that it satisfies a universal property, in \Cref{Theorem:AuslandersFormulaForExactCategories} that Auslander's formula holds, and in \Cref{Corollary:Left2Adjoint} that the association $\EE\mapsto \smodad(\EE)$ can be made into a left adjoint $2$-functor. We finish the section by showing that $\smodad(\EE)$ has the properties of an Auslander exact category.

In Section \ref{section:Correspondence} we introduce Auslander exact categories, and we show that $\smodad(-)$ gives an equivalence onto the $2$-category of Auslander exact categories, see \Cref{FirstAuslanderCorrespondence}. We then investigate how properties of $\EE$ are reflected in $\smodad(\EE)$, for example having enough injectives, see \Cref{Theorem:AuslanderSecond}, and having enough projectives, see \Cref{Proposition:EnoughProjectivesInEEMakesTTTorsionFree}.  We end by characterizing $\smodad(\EE)$ for Gorenstein projective modules, see \Cref{Corollary:GorensteinProjective}, and for maximal Cohen-Macaulay modules, see \Cref{Corollary:CohenMacaulay}.

In Section \ref{section:ViaResolvingSubcategories} we investigate different exact structures on an idempotent complete additive category by studying $\smodad(\EE)$. In particular, we show that $\smodad(\EE)$ is the smallest resolving subcategory of $\PP^2(\EE)$ containing $\eff(\EE)$, see \Cref{Proposition:SmallestResolvingSubcategory}. In the last part we show \Cref{Theorem:IntroductionExactStructuresResolving}, which is the main result of this section.

\subsection*{Conventions}
All categories $\CC$ are assumed to be essentially small. We also assume that they are additive, i.e. that $\Hom_{\CC}(-,-)$ is enriched over abelian groups and that $\CC$ admit finite direct sums. All subcategories are assumed to be full and closed under isomorphisms.

An endomorphism $e\colon X \to X$ is called an \emph{idempotent} if $e^2=e$.  A morphism $r\colon X \to Y$ is called a \emph{retraction} if there is a morphism $s\colon Y \to X$ such that $r \circ s = 1_Y$.  A category for which every idempotent has a kernel is called \emph{idempotent complete}; a category for which every retraction has a kernel is called \emph{weakly idempotent complete}.  Both idempotent completeness and weakly idempotent completeness are self-dual concepts.

\subsection*{Acknowledgements}
The authors thank Luisa Fiorot, Mikhail Gorsky, and Chrysostomos Psaroudakis for useful discussions, as well as the anonymous referee for many and detailed comments improving the manuscript.  The 2nd author would like to thank the Hausdorff Institute for Mathematics in Bonn, since parts of the paper were written during his stay at the junior trimester program ``New Trends in Representation Theory''.  The third authors is currently a postdoctoral researcher at FWO (12.M33.16N).
\section{Localizations/Quotients of exact categories}\label{section:Localizations}

In this section, we recall the notion of an exact category as introduced by Quillen in \cite{Quillen73}, as well as the notion of a \emph{(two-sided admissibly) percolating subcategory} as in \cite{HenrardVanRoosmalen19a} (these types of subcategories are called \emph{subcategories that localize} in \cite{Cardenas98}). Given an exact category $\EE$ and a percolating subcategory $\AA\subseteq \EE$, the quotient $\EE/\AA$ can be realized as a localization $\Sigma_{\AA}^{-1}\EE$. 

%
%
%
%
%

\subsection{Exact categories}

An exact category in the sense of Quillen is an additive category equipped with extra structure, namely a class of chosen kernel-cokernel pairs, satisfying some additional conditions.

\begin{definition}
	\begin{enumerate}
		\item A \emph{conflation category} is an additive category $\CC$ together with a chosen class of kernel-cokernel pairs, called \emph{conflations}, closed under isomorphisms. The kernel part of a conflation is called an \emph{inflation} and the cokernel part of a conflation is called a \emph{deflation}. We depict inflations by the symbol $\inflation$ and deflations by $\deflation$.
		\item An additive functor $F\colon \CC\to \DD$ of conflation categories is called \emph{exact} or \emph{conflation-exact} if conflations are mapped to conflations.
		\item A map $f\colon X\to Y$ in a conflation category $\CC$ is called \emph{admissible} if $f$ admits a factorization $X\deflation I\inflation Y$. Admissible morphisms are depicted by $\xymatrix{X\ar[r]|{\circ}^f & Y}$.
		\item A cochain $\dots\to X^{n-1}\xrightarrow{d^{n-1}}X^n\xrightarrow{d^n}X^{n+1}\to\dots $ in a conflation category $\CC$ is called \emph{acyclic} or \emph{exact} if each $d^i$ is admissible and $\ker(d^{i+1})=\im(d^i)$.
	\end{enumerate}
\end{definition}

\begin{remark}
	Note that if $f\colon X\to Y$ is admissible, the factorization $X\deflation I\inflation Y$ is unique up to isomorphism. Moroever, $f$ admits a kernel and cokernel and $\im(f) \coloneqq \ker(\coker(f))\cong I\cong \coker(\ker(f))\eqqcolon\coim(f)$.
\end{remark}

\begin{definition}\label{Definition:AxiomsR0-R2AndL0-L2}
	A conflation category $\EE$ is called an \emph{exact category} if $\EE$ satisfies the following two dual lists of axioms:
	\begin{enumerate}[label=\textbf{R\arabic*},start=0]
		\item\label{R0} For any $X\in \EE$, $X\deflation 0$ is a deflation.
		\item\label{R1} Deflations are closed under composition.
		\item\label{R2} Pullbacks along deflations exist, moreover, deflations are stable under pullbacks.
	\end{enumerate}
	The dual list is given by:
	\begin{enumerate}[label=\textbf{L\arabic*},start=0]
		\item\label{L0} For any $X\in \EE$, $0\inflation X$ is an inflation.
		\item\label{L1} Inflations are closed under composition.
		\item\label{L2} Pushouts along inflations exist, moreover, inflations are stable under pushouts.
	\end{enumerate}
\end{definition}

Given an exact category $\EE$, we let $\Kb(\EE)$ denote the homotopy category of bounded complexes in $\EE$, and we let $\Acb(\EE)$ denote the subcategory of $\Kb(\EE)$ consisting of the acyclic complexes. Note that $\Acb(\EE)$ is not closed under isomorphism in $\Kb(\EE)$ unless $\EE$ is weakly idempotent complete, see \cite[Proposition 10.14]{Buhler10}. However, $\Acb(\EE)$ is still a triangulated subcategory of $\Kb(\EE)$, see \cite[Corollary 10.5]{Buhler10}. The Verdier localization $\Kb(\EE)/\Acb(\EE)$ yields the triangulated category $\Db(\EE)$, called the bounded derived category of $\EE$. The natural embedding $\EE\hookrightarrow \Db(\EE)$, mapping objects to stalk complexes in degree zero, is fully faithful. Moreover, conflations correspond to triangles under this embedding.  

\subsection{Localization of categories}

We recall briefly how to localize a (small) category at a set of morphisms.  We refer to \cite{GabrielZisman67, KashiwaraSchapira06} for more information.

\begin{definition}\label{Definition:LocalizationWithRespectToMorphisms}
Let $\CC$ be any category and let $\Sigma \subseteq \Mor \CC$ be any class of morphisms of $\CC$.  The \emph{localization of $\CC$ with respect to $\Sigma$} is a functor $Q\colon \CC \to \CC[\Sigma^{-1}]$ satisfying the following conditions:
\begin{enumerate}
	\item For each $s \in \Sigma$, the morphism $Q(s)$ is invertible in $\CC[\Sigma^{-1}]$.
	\item For each category $\DD$ and functor $F\colon \CC\to \DD$ such that $F(s)$ is an isomorphism for each $s\in \Sigma$, there exists a functor $F'\colon \CC[\Sigma^{-1}]\to \DD$ and a natural isomorphism $F\cong F'\circ Q$.
	\item For each category $\DD$, the functor $-\circ Q\colon \Fun(\CC[\Sigma^{-1}], \DD) \to \Fun(\CC, \DD)$ is fully faithful.
\end{enumerate}
\end{definition}

If $\CC$ is a small category, then localizations with respect to every class of morphisms of $\CC$ exist.

\begin{remark}
	Put differently, the localization $Q\colon \CC \to \CC[\Sigma^{-1}]$ is a functor such that, for each $s \in \Sigma$, the map $Q(s)$ is invertible, and is $2$-universal with respect to this property. Note that \cite{GabrielZisman67, KashiwaraSchapira06} define the localization of $\CC$ with respect to $\Sigma$ via a $1$-universal property.
\end{remark}

In this paper, we often consider localizations with respect to so-called right multiplicative systems.

\begin{definition}\label{definition:RMS}
Let $\CC$ be a category and let $\Sigma$ be a class of morphisms of $\CC$.  We say that $\Sigma$ is a \emph{right multiplicative system} if it has the following properties:
\begin{enumerate}[label=\textbf{RMS\arabic*},start=1]
	\item\label{RMS1} For every object $A$ of $\CC$ the identity $1_A$ is contained in $\Sigma$. Composition of composable arrows in $\Sigma$ is again in $\Sigma$.
	\item\label{RMS2} Every solid diagram
	\[\xymatrix{
		X \ar@{.>}[r]^{g} \ar@{.>}[d]_{t}^{\rotatebox{90}{$\sim$}}& Y\ar[d]_{s}^{\rotatebox{90}{$\sim$}}\\
		Z\ar[r]_{f} & W
	}\] with $s\in \Sigma$ can be completed to a commutative square with $t\in \Sigma$. 
	\item\label{RMS3} For every pair of morphisms $f,g\colon X\rightarrow Y$ and $s\in S$ with source $Y$ such that $s\circ f= s\circ g$ there exists a $t\in S$ with target $X$ such that $f\circ t =g\circ t$.
\end{enumerate}
\end{definition}

For localizations with respect to a right multiplicative system, we have the following description of the localization.

\begin{construction}\label{construction:Localization}
	Let $\CC$ be a category and $\Sigma$ a right multiplicative system in $\CC$.  We define a category $\Sigma^{-1}\CC$ as follows:
	\begin{enumerate}
		\item We set $\Ob(\Sigma^{-1}\CC)=\Ob(\CC)$.
		\item\label{construction:Localization:2} Let $f_1\colon X_1\rightarrow Y, s_1\colon X_1\rightarrow X, f_2\colon X_2\rightarrow Y, s_2\colon X_2\rightarrow X$ be morphisms in $\CC$ with $s_1,s_2\in \Sigma$. We call the pairs $(f_1,s_1), (f_2,s_2) \in (\Mor \CC) \times \Sigma$ equivalent (denoted by $(f_1,s_1) \sim (f_2,s_2)$) if there exists a third pair $(f_3\colon X_3\rightarrow Y,s_3\colon X_3\rightarrow X) \in (\Mor \CC) \times \Sigma$ and morphisms $u\colon X_3\rightarrow X_1, v\colon X_3\rightarrow X_2$ such that 
		\[\xymatrix@!{
			& X_1\ar[ld]_{s_1}^{}\ar[rd]^{f_1} & \\
			X &X_3\ar[d]^{v}\ar[u]_{u}\ar[l]_{s_3}^{}\ar[r]_{f_3} & Y\\
			& X_2 \ar[ul]^{s_2}_{}\ar[ur]_{f_2}&
		}\] is a commutative diagram.
		\item $\Hom_{\Sigma^{-1}\CC}(X,Y)=\left\{(f,s)\mid f\in \Hom_{\CC}(X',Y), s\in \Hom_{\CC}(X',X) \mbox{ with } s\in \Sigma \right\} / \sim$
		\item\label{construction:Localization:D} The composition of $(f\colon X'\rightarrow Y, s\colon X'\rightarrow X)$ and $(g\colon Y'\rightarrow Z, t\colon Y'\rightarrow Y)$ is given by $(g\circ h\colon X''\rightarrow Z,s\circ u\colon X''\rightarrow X)$ where $h$ and $u$ are chosen to fit in a commutative diagram 
		\[\xymatrix{
		X''\ar[r]^{h}\ar[d]_{u}^{\rotatebox{90}{$\sim$}} & Y'\ar[d]_{t}^{\rotatebox{90}{$\sim$}}\\
		X'\ar[r]^{f} & Y
		}\]which exists by \ref{RMS2}.
	\end{enumerate}
\end{construction}

\subsection{Localizations via percolating subcategories}\label{Subsection:PercolatingSubcategories}

Throughout this section, let $\EE$ be an exact category.

\begin{definition}\label{Definition:TwoSidedAdmissiblyPercolating}
	A full nonempty subcategory $\AA\subseteq \EE$ is called an \emph{deflation-percolating subcategory} if the following axioms are satisfied:
	\begin{enumerate}[align=left]
		\myitem{\textbf{A1}}\label{A1} the category $\AA$ is a \emph{Serre subcategory} of $\EE$, i.e.~given a conflation $X\inflation Y\deflation Z$ in $\EE$, we have that $Y\in \AA\Leftrightarrow X,Z\in \AA$.
		\myitem{\textbf{A2}}\label{A2} All morphisms $f\colon X\to A$ with $A\in \AA$ are admissible with image in $\AA$, i.e.~factors as \[X\deflation \im(f)\inflation A\] and $\im(f)\in \AA$.
	\end{enumerate}
	Dually, $\AA$ is called an \emph{inflation-percolating subcategory} if $\AA^{\mathsf{op}}$ is a deflation-percolating subcategory in $\EE^{\mathsf{op}}$. Finally, $\AA$ is a \emph{percolating subcategory} if it is both a deflation- and inflation-percolating subcategory.
\end{definition}

\begin{remark}\makeatletter
\hyper@anchor{\@currentHref}%
\makeatother\label{remark:Percolating}
\begin{enumerate}
	  \item The terminology used in \Cref{Definition:TwoSidedAdmissiblyPercolating} differs from the one used in \cite{HenrardVanRoosmalen19a}: a deflation- or inflation-percolating subcategory from \Cref{Definition:TwoSidedAdmissiblyPercolating} was called an admissibly deflation- or inflation-percolating subcategory; what we call a percolating subcategory was called a two-sided admissibly percolating subcategory in \cite{HenrardVanRoosmalen19a} and a subcategory that localizes in \cite{Cardenas98}.
		\item\label{remark:Percolating:2} Since an deflation-percolating subcategory $\AA\subseteq \EE$ is closed under extensions, it inherits an exact structure from $\EE$. Moreover, it follows from axiom \ref{A2} that $\AA$ is an abelian category.  The short exact sequences in $\AA$ are precisely those sequences which are conflations in $\EE$. 
	\end{enumerate}
\end{remark}

The following definition is based on \cite[Definition~1.12]{Schlichting04} and \cite[Definition~4.0.36]{Cardenas98}.

\begin{definition}\label{Definition:WeakIsomorphism}
	Let $\AA\subseteq \EE$ be a deflation-percolating subcategory. A morphism $f\colon X\to Y$ in $\EE$ is called a \emph{weak $\AA$-isomorphism} or simply a \emph{weak isomorphism} if $\AA$ is understood, if $f$ is admissible and $\ker(f), \coker(f)\in \AA$.  Morphisms endowed with $\sim$ are weak isomorphisms.  The set of weak isomorphisms is denoted by $\Sigma_{\AA}$.
\end{definition}

\begin{proposition}\label{Proposition:PropertiesOfWeakIsomorphisms}
	Let $\AA\subseteq \EE$ be a deflation-percolating subcategory.
	\begin{enumerate}
		\item\label{Proposition:PropertiesOfWeakIsomorphismsA} The set $\Sigma_{\AA}$ is a right multiplicative system.
		\item\label{Proposition:PropertiesOfWeakIsomorphismsB} The set $\Sigma_{\AA}$ is saturated, i.e.~$Q(f)$ is an isomorphism in $\Sigma_{\AA}^{-1}\EE$ if and only if $f\in\Sigma_{\AA}$ .
		\item\label{Proposition:PropertiesOfWeakIsomorphismsC} If $Q(M)\cong 0$ for an object $M\in \EE$, then $M\in \AA$.  If $Q(f)=0$ for a morphism $f$ in $\EE$, then $f$ factors through $\AA$.
		\item\label{Proposition:PropertiesOfWeakIsomorphismsD} Pullbacks along weak isomorphisms exist and weak isomorphisms are stable under pullbacks.
		\item\label{Proposition:PropertiesOfWeakIsomorphismsE} Admissible morphisms are stable under pullbacks along weak isomorphisms.
	\end{enumerate}
\end{proposition}

\begin{proof}
We refer the reader to \cite{HenrardVanRoosmalen19a} for the first four statements. The last statement follows from \cite[Corollary 6.10]{HenrardVanRoosmalen19a} and the fact that deflations are preserved under pullbacks.
\end{proof}

The following theorem yields that the localization $\Sigma_{\AA}^{-1}\EE$ satisfies the universal property of the quotient $\EE/\AA$.

\begin{theorem}\label{Theorem:TwoSidedLocalizationTheorem}
	Let $\AA\subseteq \EE$ be a percolating subcategory and write $Q\colon \EE\to \Sigma_{\AA}^{-1}\EE$ for the localization functor.
	\begin{enumerate}
		\item\label{Item:Theorem:TwoSidedLocalizationTheorem1} $Q$ sends conflations in $\EE$ to kernel-cokernel pairs in $\Sigma_{\AA}^{-1}\EE$. Furthermore, closing this class of kernel-cokernel pairs under isomorphisms makes $\Sigma_{\AA}^{-1}\EE$ into an exact category.
		\item\label{Item:Theorem:TwoSidedLocalizationTheorem2} The functor $Q$ satisfies the $2$-universal property of a quotient $\EE/\AA$ in the category of exact categories, i.e.~for any exact category $\FF$ and exact functor $F\colon \EE\to \FF$ such that $F(\AA)\cong 0$, there exists an exact functor $F'\colon \Sigma_{\AA}^{-1}\EE\to \FF$ such that $F'\circ Q=F$.
		\item\label{Item:Theorem:TwoSidedLocalizationTheorem3} The localization sequence $\AA\to \EE\xrightarrow{Q} \Sigma_{\AA}^{-1}\EE$ induces the Verdier localization sequence 
		\[\DAb(\EE)\to \Db(\EE)\to \Db(\EE/\AA)\]
	where $\DAb(\EE)$ is the thick full triangulated subcategory of $\Db(\EE)$ generated by $\AA$ under the canonical embedding $\EE\hookrightarrow\Db(\EE)$.
	\end{enumerate}
\end{theorem}

\begin{proof}
The first two statements are \cite[Theorem~8.1]{HenrardVanRoosmalen19a} (see also \cite[Theorem 4.0.57]{Cardenas98}).  The final statement is \cite[Theorem~1.4]{HenrardVanRoosmalen19b}.
\end{proof}

Explicitly, a sequence $X\xrightarrow{f} Y\xrightarrow{g} Z$ in $\Sigma_{\AA}^{-1}\EE$ is a conflation if and only if there exists a commutative diagram
\[\xymatrix{
			Q(X')\ar[r]^{Q(f')}\ar[d]^{\cong} & Q(Y')\ar[r]^{Q(g')}\ar[d]^{\cong} & Q(Z')\ar[d]^{\cong}\\
			X\ar[r]^{f} & Y\ar[r]^{g}& Z
		}\]
		where the vertical maps are isomorphisms in $\Sigma_{\AA}^{-1}\EE$, and where $X'\xrightarrow{f'} Y'\xrightarrow{g'} Z'$ is a conflation in $\EE$. 

\subsection{\texorpdfstring{Admissible morphisms in $\EE/\AA$}{Admissible morphisms in the quotient}}\label{Section:QReflectsAdmissibles}

In this section we show that the localization functor $Q\colon \EE\to \EE/\AA$ reflects admissible morphisms when $\AA$ is a percolating subcategory (see \Cref{Theorem:QReflectsAdmissibles} below).  We start with the following result.

\begin{proposition}\label{Proposition:weakisocompositionadmissible}
Let $\EE$ be an exact category, and let $\AA$ be a deflation-percolating subcategory of $\EE$.  Let $f\colon X\to Y$  and $g\colon Y\to Z$ be morphisms in $\EE$. The following holds:
\begin{enumerate}
\item\label{Proposition:weakisocompositionadmissible:1} If $f$ is an inflation with cokernel in $\AA$ and $g$ is a deflation, then $g\circ f=f'\circ g'$ where $g'$ is a deflation and $f'$ is an inflation with cokernel in $\AA$.
\item\label{Proposition:weakisocompositionadmissible:2} If $f$ is a weak isomorphism and $g$ is admissible, then $g\circ f$ is admissible.
\item\label{Proposition:weakisocompositionadmissible:3} If $g$ is a weak isomorphism and $g\circ f$ is admissible, then $f$ is admissible.
\end{enumerate} 
\end{proposition}

\begin{proof}
\begin{enumerate}
\item This follows from \cite[Corollary~6.10]{HenrardVanRoosmalen19a}.

 \item Write $f=f_2\circ f_1$ and $g=g_2\circ g_1$  where $f_1,g_1$ are deflations and $f_2,g_2$ are inflations. By  \eqref{Proposition:weakisocompositionadmissible:1} we have that $g_1\circ f_2$ can be written as a composite $f_2'\circ g_1'$ where $g_1'$ is a deflation and $f_2'$ is an inflation. In particular, $g_1\circ f_2$ is admissible. Hence the composite $g\circ f=g_2\circ (g_1\circ f_2)\circ f_1$ is admissible, which proves the claim.

\item Taking the pullback of $g$ along $g\circ f$, we get a commutative diagram
\[\xymatrix{
		  X'\ar[r]^{h}\ar[d]^{g'} & Y\ar[d]^{g}\\
		  X\ar[r]^{g\circ f} & Z.
	}\]
Since $g\circ f\colon X\to Z$ factors through $g\colon Y\to Z$ via $f$, there exists a map $i\colon X\to X'$ satisfying $g'\circ i = 1_X$ and $h\circ i=f$. Since weak isomorphisms are preserved by pullbacks, $g'$ must be a weak isomorphism, and hence a deflation since it has a right inverse. Therefore, $i$ is an inflation with cokernel in $\AA$. Also, since admissible morphisms are preserved under pullbacks by weak isomorphisms, $h$ must be admissible. Since $f=h\circ i$, we get that $f$ is admissible by the first part of the proof.\qedhere
\end{enumerate}
\end{proof}

Next we describe the inflations and deflations in $\EE/\AA$. To do this we introduce the following classes of maps 
\begin{align*}
& \SS_{\operatorname{def}}=\{f\in \operatorname{Mor}\CC\mid f=f_1\circ f_2 \mbox{ where }f_2 \mbox { is a deflation and }f_1\in \Sigma\}, \\
& \SS_{\operatorname{inf}}=\{f\in \operatorname{Mor}\CC\mid f=f_1\circ f_2 \mbox{ where }f_1 \mbox { is an inflation and }f_2\in \Sigma\}.
\end{align*}

\begin{lemma}\label{lemma:S_def}
Let $\EE$ be an exact category, and let $\AA$ be a percolating subcategory of $\EE$. The following hold:
\begin{enumerate}
\item\label{lemma:S_def:1} $\SS_{\operatorname{def}}$ is closed under pullbacks.
\item\label{lemma:S_def:2} $\SS_{\operatorname{def}}$ is closed under  pushouts along weak isomorphisms (which exist by the dual of \Cref{Proposition:PropertiesOfWeakIsomorphisms}.\eqref{Proposition:PropertiesOfWeakIsomorphismsD}).
\item\label{lemma:S_def:3} $\SS_{\operatorname{def}}$ is closed under compositions.
\end{enumerate}
\end{lemma}

\begin{proof}
\begin{enumerate}
\item Since pullbacks preserve weak isomorphisms by \Cref{Proposition:PropertiesOfWeakIsomorphisms}.\eqref{Proposition:PropertiesOfWeakIsomorphismsD} and deflations by definition, it follows that $\SS_{\operatorname{def}}$ is closed under pullbacks. 
\item Since pushouts preserve epimorphisms, and pushouts along weak isomorphisms preserve admissible morphisms by the dual of \Cref{Proposition:PropertiesOfWeakIsomorphisms}.\eqref{Proposition:PropertiesOfWeakIsomorphismsE}, it follows that pushouts along weak isomorphisms preserve deflations. Since pushouts of weak isomorphisms are weak isomorphisms by the dual of \Cref{Proposition:PropertiesOfWeakIsomorphisms}.\eqref{Proposition:PropertiesOfWeakIsomorphismsD}, it therefore follows that $\SS_{\operatorname{def}}$ is closed under pushouts along weak isomorphisms.  
\item Let $f=f_1\circ f_2$ and $g=g_1\circ g_2$ be two composable morphisms where $f_2$ and $g_2$ are deflations and $f_1$ and $g_1$ are weak isomorphisms. By \Cref{Proposition:weakisocompositionadmissible}.\eqref{Proposition:weakisocompositionadmissible:1} we have that $g_2\circ f_1=g_1'\circ f_2'$ where $f_2'$ is a deflation and $g_1'$ is a weak isomorphisms. Hence
\[
g\circ f=g_1\circ g_2\circ f_1\circ f_2=g_1\circ g_1'\circ f_2'\circ f_2
\]
Since deflations and weak isomorphisms are closed under compositions, it follows that $f_2'\circ f_2$ is a deflation and $g_1\circ g_1'$ is a weak isomorphism. This shows that $\SS_{\operatorname{def}}$ is closed under compositions. \qedhere
\end{enumerate}
\end{proof}

\begin{lemma}
\label{lemma:WhenDeflationInQuotient}
Let $\AA$ be a percolating subcategory of an exact category $\EE$. The following hold:
\begin{enumerate}
\item\label{lemma:WhenDeflationInQuotient:1} A morphism in $\EE / \AA$ is a deflation if and only if it can be represented by a roof $(f,s)$ where $f\in \SS_{\operatorname{def}}$.
\item\label{lemma:WhenDeflationInQuotient:2} A morphism in $\EE / \AA$ is an inflation if and only if it can be represented by a roof $(f,s)$ where $f\in \SS_{\operatorname{inf}}$.
\end{enumerate}
\end{lemma}

\begin{proof}
\begin{enumerate}
\item Consider the following class of morphisms in $\EE/\AA$
\[
\UU=\{g\in  \operatorname{Mor}(\EE/\AA)\mid g\mbox{ can be represented by a roof } (f,s) \mbox{ where }f\in \SS_{\operatorname{def}}.\}
\]  
Clearly, $\UU$ contains all the isomorphisms in $\EE/\AA$. Furthermore, $\UU$ is closed under compositions. Indeed, let $(f,s)$ and $(g,t)$ be two composable morphisms in $\EE/\AA$ where $f,g\in \SS_{\operatorname{def}}$. Following the notation in \Cref{construction:Localization}.\eqref{construction:Localization:D}, by taking the pullback of $f$ along $t$, we get a roof $(g\circ h, s\circ u)$ representing the composition $(g,t)\circ (f,s)$ where the commutative diagram in \Cref{construction:Localization}.\eqref{construction:Localization:D} is the pullback square. Since $h$ is a pullback of $f$ and  $\SS_{\operatorname{def}}$ is closed under pullbacks by \Cref{lemma:S_def}.\eqref{lemma:S_def:1}, it follows that $h\in\SS_{\operatorname{def}}$. Since $\SS_{\operatorname{def}}$ is closed under compositions by \Cref{lemma:S_def}.\eqref{lemma:S_def:3}, it follows that $g\circ h\in \SS_{\operatorname{def}}$. This shows that $\UU$ is closed under compositions.

Now assume $g\colon Y\to Z$ is a deflation in $\EE/\AA$. By the description of the conflations in $\EE/\AA$, we get that $g$ is equal to a composite $Y\xrightarrow{\cong}Q(Y')\xrightarrow{Q(g')}Q(Z')\xrightarrow{\cong}Z$ in $\EE/\AA$ for some deflation $g'\colon Y'\deflation Z'$ in $\EE$. Since $g$ is a composition of morphisms in $\UU$, it follows that $g$ is in $\UU$. This proves the claim.
\item This is proved similarly to \eqref{lemma:WhenDeflationInQuotient:1}, using that $\SS_{\operatorname{inf}}$ is closed under compositions and under pullbacks along weak isomorphisms by the dual of \Cref{lemma:S_def}.\eqref{lemma:S_def:3} and the dual of \Cref{lemma:S_def}.\eqref{lemma:S_def:2}, respectively. \qedhere
\end{enumerate}
\end{proof}

\begin{theorem}\label{Theorem:QReflectsAdmissibles}
Let $\EE$ be an exact category and let $\AA$ be a percolating subcategory of $\EE$.  A morphism $k\colon X\to Z$ is admissible in $\EE$ if and only if $Q(k)$ is admissible in $\EE/\AA$.  In other words, the localization functor $Q$ reflects admissible morphisms.
\end{theorem}

\begin{proof}
As $Q$ preserves inflations and deflations, it is easy to see that, if $k$ is admissible, then so is $Q(k)$. For the other direction, assume that $Q(k)$ is admissible in $\EE / \AA $.  By \Cref{lemma:WhenDeflationInQuotient} we can write $Q(k)$ as a composition $(g,t) \circ (f,s)$ of roofs in $\EE / \AA $, where $g\in \SS_{\operatorname{inf}}$ and $f\in \SS_{\operatorname{def}}$. Following the notation in \Cref{construction:Localization}.\eqref{construction:Localization:D}, by taking the pullback of $f$ along $t$, we get a roof $(g\circ h, s\circ u)$ which also represents the composition $(g,t)\circ (f,s)$, where the commutative diagram in \Cref{construction:Localization}.\eqref{construction:Localization:D} is the pullback square. Since $h$ is a pullback of $f$ and $\SS_{\operatorname{def}}$ is closed under pullbacks by \Cref{lemma:S_def}.\eqref{lemma:S_def:1}, it follows that $h\in\SS_{\operatorname{def}}$. 

Next, we claim that the composite $g\circ h$ is an admissible morphism in $\EE$. To this end, let $g=g_1\circ g_2$ where $g_1$ is an inflation and $g_2$ is a weak isomorphism, and let $h=h_1\circ h_2$ where $h_1$ is a weak isomorphism and $h_2$ is a deflation. Since weak isomorphisms are closed under compositions, we get that $g_2\circ h_1$ is a weak isomorphism. Hence, by \Cref{Proposition:weakisocompositionadmissible}.\eqref{Proposition:weakisocompositionadmissible:2} it follows that $g_1\circ g_2\circ h_1$ is an admissible morphism. Since $h_2$ is a deflation, we get that $(g_1\circ g_2\circ h_1)\circ h_2=g\circ h$ is an admissible morphism. 
 
Finally, note that the roofs $(k,1)$ and $(g\circ h, s\circ u)$ both represent the same morphism $Q(k)$ in $\EE/\AA$. Therefore, they must be equivalent as in \Cref{construction:Localization}.\eqref{construction:Localization:2}. Hence,
there exists a weak isomorphism $v$ such that the equality 
\[
g\circ h\circ v=k\circ s\circ u\circ v
\]
holds in $\EE$. Since $g\circ h$ is admissible, it follows from  \Cref{Proposition:weakisocompositionadmissible}.\eqref{Proposition:weakisocompositionadmissible:2} that $g\circ h\circ v$ is admissible. Hence $k\circ s\circ u \circ v$ is an admissible morphism. Since weak isomorphisms are closed under compositions, $s\circ u\circ v$ is a weak isomorphism. Therefore, by the dual of \Cref{Proposition:weakisocompositionadmissible}.\eqref{Proposition:weakisocompositionadmissible:3} we have that $k$ is an admissible morphism. This proves the claim. 
\end{proof}

\subsection{Torsion theory with percolating torsion class}\label{Section:TorsionTheory}

In this section, we show that for any torsion pair $(\TT,\FF)$ in an exact category $\EE$ such that $\TT$ satisfies axiom \ref{A2}, the torsion class $\TT$ is a percolating subcategory of $\EE$.

We start by recalling the notion of a torsion theory in an exact category.

\begin{definition}
	Let $\EE$ be an exact category. A \emph{torsion theory} or \emph{torsion pair} is a pair $(\TT,\FF)$ of full subcategories, closed under isomorphisms, such that the following conditions hold:
	\begin{enumerate}
		\item $\Hom_{\EE}(T,F)=0$ for all $T\in \TT$ and $F\in \FF$.
		\item For any object $M\in \EE$, there exists a conflation $T\inflation M\deflation F$ with $T\in \TT$ and $F\in \FF$.
	\end{enumerate}
\end{definition}

For a subcategory $\XX$ of $\EE$ we let ${}^{\perp}\XX$ and $\XX^{\perp}$ denote the subcategories consisting of all objects $E\in \EE$ satisfying $\Hom_{\EE}(E,X)=0$ or $\Hom_{\EE}(X,E)=0$ for all $X\in \XX$, respectively. The following proposition is standard.

\begin{proposition}\label{Proposition:BasicPropertiesTorsionPair}
	Let $(\TT,\FF)$ be a torsion pair in an exact category $\EE$. 
	\begin{enumerate}
	\item $\TT={}^{\perp}\FF$ and $\FF=\TT^{\perp}$.
		\item The subcategory $\TT\subseteq \EE$ is closed under extensions and epimorphic quotients. Dually, the subcategory $\FF\subseteq \EE$ is closed under extensions and subobjects. In particular, $\TT$ and $\FF$ inherit exact structures from $\EE$.
		\item For each $M\in \EE$, there is a unique (up to isomorphism) conflation $\tt M\inflation M\deflation \ff M$ with $\tt M\in \TT$ and $\ff M\in \FF$.
		\item The inclusion functor $\TT\to \EE$ has right adjoint $\tt$ such that $\tt(M)=\tt M$ as above. Dually, the inclusion functor $\FF\to\EE$ has a left adjoint $\ff$ such that $\ff(M)=\ff M$.
	\end{enumerate}
\end{proposition}

\begin{proof}
This statement combines \cite[Lemma~8.14, Proposition~8.15, and Corollary~8.16]{HenrardVanRoosmalen19a} and is based on \cite{BournBorceux06}.
\end{proof}

The following definition is adapted from \cite{PeschkeVanderLinden16}.

\begin{definition}
	An additive functor $F\colon \EE\to \FF$ between exact categories is called \emph{left exact} if a conflation $X\stackrel{f}{\inflation}Y\stackrel{g}{\deflation}Z$ is mapped to an exact sequence $\xymatrix{F(X)\ar@{>->}[r]^{F(f)} & F(Y)\ar[r]|{\circ}^{F(g)} & F(Z)\ar@{->>}[r] & \coker(F(g))}$. The notion of a \emph{right exact functor} is defined dually.
\end{definition}

\begin{remark}
If $F\colon \EE\to \FF$ is left exact and $h\colon Y\to Z$ is an admissible morphism in $\EE$ with kernel $f\colon X\inflation Y$, then $F(h)\colon F(Y)\to F(Z) $ is an admissible morphism in $\FF$ with kernel $F(f)\colon F(X)\inflation F(Y)$. Indeed, choosing a factorization $Y\stackrel{g'}{\deflation}\im(f)\stackrel{g''}{\inflation}Z$ of $g$ and applying $F$ to the conflation
$X\stackrel{f}{\inflation}Y\stackrel{g'}{\deflation}\im(f)$
we get an admissible morphism $F(g')\colon F(Y)\to F(\im(f))$ with kernel $F(f)\colon F(X)\inflation F(Y)$. The claim now follows since $F(g)=F(g'')\circ F(g')$ and $F(g'')$ is an inflation.
\end{remark}

\begin{remark}
Note that a left exact (resp right exact) functor between exact categories does not necessarily preserve all kernels (resp cokernels), but only kernels of admissible morphisms.
\end{remark}

\begin{proposition}\label{Proposition:Torsion+A2IsTwoSidedPercolating}
	Let $\EE$ be an exact category with a torsion pair $(\TT,\FF)$ such that $\TT\subseteq \EE$ satisfies axiom \ref{A2}.
	\begin{enumerate}	
		\item The subcategory $\TT\subseteq \EE$ is percolating.
		\item\label{Item:tIsSequentiallyLeftExactAndfPreservesInflationsAndDeflations} The functor $\tt$ is left exact and the functor $\ff$ preserves inflations and deflations. 
		\item\label{Item:fReflectsAdmissibles} The functor $\ff$ reflects admissible morphisms.
	\end{enumerate}
\end{proposition}

\begin{proof}
	\begin{enumerate}
		\item We first show axiom $\ref{A2}^{\mathsf{op}}$. Let $f\colon T\to X$ be a map with $X\in \TT$. As $(\TT,\FF)$ is a torsion pair, there is a conflation $\tt X\inflation X\deflation \ff X$. Note that $f$ factors through $\tt X\inflation X$ as $\Hom(\TT,\FF)=0$. By axiom \ref{A2}, the map $T\to \tt  X$ is admissible with image in $\TT$. It follows that the composition $T\to \tt X\inflation X$ is admissible with image in $\TT$. This shows axiom $\ref{A2}^{\mathsf{op}}$.
		
		We now show axiom \ref{A1}. Let $X\inflation Y\deflation Z$ be a conflation in $\EE$. Assume first that $Y\in \TT$. By axiom \ref{A2}, the map $X\inflation Y$ is admissible with image in $\TT$ and thus $X\in \TT$. Similarly, by axiom $\ref{A2}^{\mathsf{op}}$ yields that $Y\deflation Z$ is admissible with image in $\TT$ and thus $Z\in \TT$. Conversely, if $X,Z\in \TT$, then $Y\in \TT$ as $\TT\subseteq \EE$ is extension-closed.
		 
		\item Let $X\stackrel{a}{\inflation} Y \stackrel{b}{\deflation}Z$  be a conflation in $\EE$. By axiom \ref{A2} the maps $\tt(a)$ and $\tt(b)$ are admissible. As $\tt(a)$ is monic, it is an inflation. Since $\tt$ is a right adjoint, $\tt$ commutes with kernels, and thus the kernel of $\tt(b)$ in $\TT$ is given by $\tt(a)$. Since $\TT$ is abelian (as it satisfies axiom \ref{A2}, see \Cref{remark:Percolating}.\eqref{remark:Percolating:2}) and the inclusion $\TT\to \EE$ is exact, it follows that $\tt(a)$ is also the kernel of $\tt(b)$ in $\EE$. This shows that $\tt$ is left exact. In particular, it follows that $\im(\tt(b))\cong\coker(\tt(a))$. By axiom \ref{L1}, the composition $\im(\tt(b))\inflation \tt Z\inflation Z$ is an inflation and we write $L$ for its cokernel. Consider the following two commutative diagrams:
		\[\xymatrix{
			\tt X\ar@{>->}[r]\ar@{>->}[d] & \tt Y\ar@{->>}[r]\ar@{>->}[d] & \im(\tt (b))\ar@{>->}[d] && \im(\tt(b))\ar@{>->}[r]\ar@{=}[d] & \tt Z\ar@{->>}[r]\ar@{>->}[d] & \coker(\tt(b))\ar[d]\\
			X\ar@{>->}[r]\ar@{->>}[d] & Y\ar@{->>}[r]\ar@{->>}[d] & Z\ar@{->>}[d] && \im(\tt(b))\ar@{->>}[d]\ar@{>->}[r] & Z\ar@{->>}[r]\ar@{->>}[d] & L\ar[d]\\
			\ff X\ar[r] & \ff Y\ar[r] & L && 0\ar@{>->}[r] & \ff Z\ar@{=}[r] & \ff Z
		}\] Applying the Nine Lemma to both diagrams, we obtain conflations \[
		\ff X\inflation \ff Y\deflation L \quad \text{and} \quad  \coker(\tt (b))\inflation L\deflation \ff Z.
		\]
		 By axiom \ref{R1}, the composition $\ff Y\deflation L\deflation fZ$ is a deflation. The result follows.

		\item Let $f\colon X\to Y$ be a morphism in $\EE$. Consider the following commutative diagram:
		\[\xymatrix{
			\tt X\ar@{>->}[r]\ar[d]^{\tt(f)} & X\ar@{->>}[r]^{\sim}\ar[d]^{f} & \ff X\ar[d]^{\ff(f)}\\
			\tt Y\ar@{>->}[r] & Y\ar@{->>}[r]^{\sim} & \ff Y
		}\] If $f$ is admissible, then $\ff(f)$ is admissible as $\ff$ preserves inflations and deflations. Conversely, if $\ff(f)$ is admissible, then the composition $X\deflation \ff X \xrightarrow{\ff(f)}\ff Y$ is admissible. Therefore the composition $X\xrightarrow{f}Y\deflation \ff Y$ is admissible, and hence $f$ is admissible by \Cref{Proposition:weakisocompositionadmissible}.\eqref{Proposition:weakisocompositionadmissible:3}.\qedhere
	\end{enumerate}
\end{proof}

\begin{remark}
	Despite the functor $\ff$ preserving inflations and deflations, it need not be conflation-exact. Indeed, consider the standard torsion theory on the category abelian groups. The short exact sequence $\mathbb{Z} \stackrel{\cdot n}{\inflation}\mathbb{Z}\deflation \mathbb{Z}/n\mathbb{Z}$ is mapped to $\mathbb{Z} \stackrel{\cdot n}{\rightarrow}\mathbb{Z}\to 0$.
\end{remark}
\section{Auslander's formula for exact categories}\label{section:Formula}

Given an exact category $\EE$, we consider the category $\smodad(\EE)$ of admissibly presented functors and the subcategory $\eff(\EE)\subseteq\smodad(\EE)$ of effaceable functors (see \Cref{Definition:smodadAndeff} below). We show that Auslander's formula extends to exact categories, i.e.~the categories $\EE$ and $\smodad(\EE)/\eff(\EE)$ are equivalent as exact categories. Moreover, $\eff(\EE)\subseteq\smodad(\EE)$ is a percolating subcategory that arises as the torsion part of a torsion theory on $\smodad(\EE)$. 

\subsection{Basic definitions and results}
Write $\Upsilon\colon \EE\to \Mod(\EE)\colon A\mapsto \Hom_{\EE}(-,A)$ for the Yoneda embedding.
We start with the following definition. 

\begin{definition}\label{Definition:smodadAndeff}
	Let $\EE$ be a conflation category. 
	\begin{enumerate}
		\item Let $\smod(\EE)$ be the full subcategory of $\Mod(\EE)$ consisting of those contravariant functors $F\colon \EE\to \Ab$ that admit a projective presentation
		\[\Upsilon(X)\xrightarrow{\Upsilon(f)}\Upsilon(Y)\to F\to 0.\]
		\item Let $\smodad(\EE)$ be the full subcategory of $\Mod(\EE)$ consisting of those functors $F$ that admit a projective presentation
		\[\Upsilon(X)\xrightarrow{\Upsilon(f)}\Upsilon(Y)\to F\to 0\]
		where $f\colon X\to Y$ is an admissible morphism in $\EE$. We refer to $f$ as a \emph{presenting} morphism of $F$.
		\item Let $\eff(\EE)$ be the full subcategory of $\Mod(\EE)$ consisting of those functors $F$ that admit a projective presentation
		\[\Upsilon(X)\xrightarrow{\Upsilon(f)}\Upsilon(Y)\to F\to 0\]
		where $f\colon X\deflation Y$ is a deflation in $\EE$.  Objects of $\eff(\EE)$ are called \emph{effaceable}.
		\item Let $\FF$ be the full subcategory of $\Mod(\EE)$ consisting of those functors $F$ that admit a projective presentation
		\[\Upsilon(X)\xrightarrow{\Upsilon(f)}\Upsilon(Y)\to F\to 0\]
		where $f\colon X\inflation Y$ is an inflation in $\EE$.
	\end{enumerate}
\end{definition}

\begin{remark}
It will be verified in \Cref{Proposition:WeaklyEffaceableAndFinitelyPresentedIsEffaceable} that objects of $\eff(\EE)$ are weakly effaceable in the sense of \Cref{definition:WeaklyEffaceable} (following the original concept introduced in \cite{Grothendieck57}).  Furthermore, it is easy to verify that $\FF$ is the full subcategory of $\modad(\EE)$ consisting of objects of projective dimension at most one.
\end{remark}

\begin{lemma}\label{Lemma:YonedaEmbeddingIsLeftExact}
	Every $F\in \smodad(\EE)$ fits into an exact sequence
	\[0\to \Upsilon(\ker(f))\to \Upsilon(X)\xrightarrow{\Upsilon(f)}\Upsilon(Y)\to F\to 0.\]
	in $\Mod(\EE)$.
\end{lemma}

\begin{proof}
	The Yoneda embedding $\Upsilon\colon\EE\to \Mod(\EE)$ is a left exact covariant functor. The result follows as $f$ admits a kernel.
\end{proof}

The following proposition will be used multiple times throughout the text.

\begin{proposition}\label{Lemma:TheFamousDiagramChase}
	Let $\CC$ be an additive category and write $\Upsilon\colon \CC\to \Mod(\CC)$ for the Yoneda embedding. Consider a commutative diagram
	\begin{equation}\label{DefiningDiagram}
			\xymatrix{
				A\ar[r]^{\beta}\ar[d]^f & C\ar[d]^{g}\\
				B\ar[r]^{\alpha} & D
			}
	\end{equation}
	in $\CC$ such that $g$ admits a kernel $k\colon \ker(g)\to C$ and such that the cospan $B\stackrel{\alpha}{\rightarrow}D\stackrel{g}{\leftarrow}C$ admits a pullback $E$. Write $F=\coker(\Upsilon(f))$, $G=\coker(\Upsilon(g))$ and $\eta\colon F\to G$ for the induced map. Consider the commutative diagram
	\begin{gather}
		\begin{aligned}
			\xymatrix{
				&&\ker(g')\ar[d]^{k'}\ar@{=}[r] & \ker(g)\ar[d]^{k} &&\\
				\ker(g)\oplus A\ar[r]^-{\begin{psmallmatrix}0&1\end{psmallmatrix}}\ar[d]^{\begin{psmallmatrix}k'&\beta''\end{psmallmatrix}} & A\ar[r]^{\beta''}\ar[d]^f & E\ar[r]^{\beta'}\ar[d]^{g'} & C\ar[r]^-{\begin{psmallmatrix}1\\0\end{psmallmatrix}}\ar[d]^{g} & C\oplus B\ar[d]^{\begin{psmallmatrix}g&\alpha\end{psmallmatrix}}\\
				E\ar[r]^{g'} & B\ar@{=}[r] & B\ar[r]^{\alpha} & D\ar@{=}[r] & D
	}
		\end{aligned}\label{Thefamousdiagram}
	\end{gather}
	where $ECBD$ is a pullback square and $\beta=\beta'\beta''$. Applying $\Upsilon$ and taking the cokernel of the vertical maps induces the epi-mono factorization 
	\[\xymatrix{
		\ker(\eta)\ar@{>->}[r] & F\ar@{->>}[r] & \im(\eta)\ar@{>->}[r] & G\ar@{->>}[r] & \coker(\eta)
	}\] of $\eta$ in $\Mod(\CC)$.
\end{proposition}

\begin{proof}
	We follow the proof of \cite[Proposition~2.1.(b)]{Auslander65} closely. Let $P^{\bullet}$ be the complex $0\to 0\to \Upsilon(A)\xrightarrow{\Upsilon(f)} \Upsilon(B)$ and let $Q^{\bullet}$ be the complex $0\to \Upsilon(\ker(g))\stackrel{\Upsilon(k)}{\inflation} \Upsilon(C)\xrightarrow{\Upsilon(g)}\Upsilon(D)$. Clearly $H_0(P^{\bullet})=F$ and $H_0(Q^{\bullet})=G$. By \Cref{Lemma:YonedaEmbeddingIsLeftExact}, $H_1(Q^{\bullet})=0$. The map $\eta\colon F\to G$ lifts to the map $\eta^{\bullet}\colon P^{\bullet}\to Q^{\bullet}$ dictated by diagram \eqref{DefiningDiagram}. Let $M^{\bullet}=\operatorname{cone}(\eta^{\bullet})$ be the mapping cone. We obtain the long exact sequence
	\[\cdots \to H_1(P^{\bullet})\to H_1(Q^{\bullet})\to H_1(M^{\bullet})\to H_0(P^{\bullet})\to H_0(Q^{\bullet})\to H_0(M^{\bullet})\to 0.\] It follows that $H_0(M^{\bullet})\cong \coker(\eta)$ and that $H_1(M^{\bullet})\cong \ker(\eta)$. By definition, $H_0(M^{\bullet})$ is isomorphic to the cokernel of $\Upsilon\begin{psmallmatrix}g&\alpha\end{psmallmatrix}\colon \Upsilon(C\oplus B)\to \Upsilon(D)$. Hence $\coker(\eta)$ has the desired resolution. 
	
	As the square $ECBD$ is a pullback square, $\begin{psmallmatrix}\beta'\\-g'\end{psmallmatrix}\colon E\to C\oplus B$ is the kernel of $\begin{psmallmatrix}g&\alpha\end{psmallmatrix}$ and by \Cref{Lemma:YonedaEmbeddingIsLeftExact}, $\Upsilon\begin{psmallmatrix}\beta'\\-g'\end{psmallmatrix}$ is the kernel of $\Upsilon\begin{psmallmatrix}g&\alpha\end{psmallmatrix}$. It follows that $H_1(M)$ is given by the cokernel of the map $\Upsilon\begin{psmallmatrix}k'&\beta''\end{psmallmatrix} \colon \Upsilon(\ker(g)\oplus A)\to \Upsilon(E)$. This yields the desired resolution of $\ker(\eta)$.
	
	It remains to obtain the resolution of $\im(\eta)$. By definition, $\im(\eta)=\coker(\ker(\eta))$. Copying the technique of the first part of this proof, one sees that $\coker(\ker(\eta))$ is given by the zeroth homology of the complex
		\[\Upsilon(\ker(g)\oplus A)\xrightarrow{\Upsilon\begin{psmallmatrix}-k'&-\beta''\\ 0&1_A\end{psmallmatrix}} \Upsilon(E\oplus A)\xrightarrow{\Upsilon\begin{psmallmatrix}g'&f'\end{psmallmatrix}} \Upsilon(B)\to 0.\] Up to homotopy equivalence, the latter complex is given by the left exact complex 
	\[\Upsilon(\ker(g))\xrightarrow{\Upsilon(k')} \Upsilon(E)\xrightarrow{\Upsilon(g')} \Upsilon(B)\to 0.\] The result follows.
\end{proof}

\subsection{\texorpdfstring{The quotient $\smodad(\EE)/\eff(\EE)$}{The quotient}}

Throughout this section, $\EE$ denotes an exact category.

\begin{proposition}\label{Proposition:AdmissiblyPresentedLiesExtensionClosed}
	The category $\smodad(\EE)$ is an extension-closed subcategory of $\Mod(\EE)$. In particular, $\smodad(\EE)$ inherits an exact structure from $\Mod(\EE)$.
\end{proposition}

\begin{proof}
	Let $0\to F \to G \to H\to 0$ be a short exact sequence in the abelian category $\Mod(\EE)$ and assume that $F,H\in \smodad(\EE)$. Let $\Upsilon (\ker(f))\xrightarrow{\Upsilon(k_f)}\Upsilon (A)\xrightarrow{\Upsilon(f)}\Upsilon (B)\to F$ and $\Upsilon(\ker(h))\xrightarrow{\Upsilon(k_h)}\Upsilon (C)\xrightarrow{\Upsilon(h)}\Upsilon (D)\to H$ be projective resolutions of $F$ and $H$ respectively with $f,h$ admissible morphisms. By the Horseshoe Lemma and the fact that the Yoneda embedding is faithful, one finds a projective resolution $\Upsilon (\ker(h)\oplus \ker(f))\xrightarrow{\Upsilon\begin{psmallmatrix}k_h &0\\ l' & k_f\end{psmallmatrix}}\Upsilon(C\oplus A)\xrightarrow{\Upsilon\begin{psmallmatrix}h &0\\l&f\end{psmallmatrix}}\Upsilon(D\oplus B)\to G$ for some morphisms $l\colon C\to B$ and $l'\colon \ker(h)\to A$. Let $C\xrightarrow{h'}\im(h)\xrightarrow{h''} D$ and be a deflation-inflation factorization of $h$. We obtain the following morphism of acyclic complexes  in $\EE$:
	\[\xymatrix{
		\cdots \ar[r] & 0\ar[r]\ar[d]&0\ar@{>->}[r]\ar[d]&\ker(h)\ar@{>->}[r]^{-k_h}\ar[d]^{l'} & C\ar@{->>}[r]^{-h'}\ar[d]^l & \im(h)\ar[r]\ar@{.>}[d]^{l''} & 0\ar[r]\ar[d] & \cdots \\
		\cdots \ar[r] &0\ar@{>->}[r]&\ker(f)\ar@{>->}[r]^{k_f}&A\ar[r]^f & B\ar@{->>}[r] & \coker(f) \ar@{->>}[r]& 0\ar[r]&\cdots
	}\] By \cite[Lemma~1.1]{Neeman90}, the cone is an acyclic complex as well. It follows that the map $\begin{psmallmatrix}h'&0\\l&f\end{psmallmatrix}\colon C\oplus A\to \im(h)\oplus B$ is admissible. Also, the map  $\begin{psmallmatrix}h''&0\\0&1_B\end{psmallmatrix}\colon \im(h)\oplus B\to D\oplus B$ is an inflation as the direct sum of inflations are inflations. It follows that $\begin{psmallmatrix}h &0\\l&f\end{psmallmatrix}=\begin{psmallmatrix}h''&0\\0&1_B\end{psmallmatrix}\begin{psmallmatrix}h'&0\\l&f\end{psmallmatrix}$ is admissible. This shows that $G\in \smodad(\EE)$.
\end{proof}

\begin{proposition}\label{Proposition:EffAreTorsion}
The pair $(\eff(\EE),\FF)$ is a torsion pair in $\smodad(\EE)$ such that $\eff(\EE)\subseteq \smodad(\EE)$ satisfies axiom \ref{A2}. In particular, $\eff(\EE)\subseteq \smodad(\EE)$ is a percolating subcategory. 
\end{proposition}

\begin{proof}
We first show that $(\eff(\EE),\FF)$ yields a torsion pair in $\smodad(\EE)$. To that end, let $G\in \smodad(\EE)$ and choose a projective presentation $\Upsilon(X)\xrightarrow{\Upsilon(f)}\Upsilon(Y)\to G\to 0$ where $f$ is admissible. Let $X\stackrel{f'}{\deflation}X'\stackrel{f''}{\inflation}Y$ be the deflation-inflation factorization of $f$. The commutative diagram
	\[\xymatrix{
		X\ar@{=}[r]\ar@{->>}[d]^{f'} & X\ar@{->>}[r]^{f'}\ar[d]^{f} & X'\ar@{>->}[d]^{f''}\\
		X'\ar@{>->}[r]^{f''} & Y\ar@{=}[r] & Y\\
	}\] in $\EE$ induces a sequence $F\stackrel{\phi}{\to} G \stackrel{\psi}{\to} H$ in $\smodad(\EE)$ where $F=\coker(\Upsilon(f'))\in \eff(\EE)$ and $H=\coker(\Upsilon(f''))\in \FF$. A straightforward diagram chase shows that $F\stackrel{\phi}{\to} G \stackrel{\psi}{\to} H$ is a short exact sequence.
	
	We now show that $\Hom_{\smodad(\EE)}(\eff(\EE),\FF)=0$. Let $\eta\colon F\to H$ be a morphism in $\smodad(\EE)$ with $F\in \eff(\EE)$ and $H\in \FF$. Choose projective presentations $\Upsilon(X)\xrightarrow{\Upsilon(f)}\Upsilon(Y)\to F\to 0$ and $\Upsilon(X')\xrightarrow{\Upsilon(h)}\Upsilon(Y')\to H \to 0$ where $f$ is a deflation and $h$ is an inflation in $\EE$. Lifting the morphism $\eta$ to a map between the projective presentations and using that the Yoneda embedding is fully faithful, we find morphisms $\alpha\colon X\to X'$ and $\beta\colon Y\to Y'$ such that $h\circ \alpha=\beta \circ f$ and making the diagram
	\[\xymatrix{
		\Upsilon(X)\ar[r]^{\Upsilon(f)}\ar[d]^{\Upsilon(\alpha)} & \Upsilon(Y)\ar[r]^{}\ar[d]^{\Upsilon(\beta)}  & F\ar[d]^{\eta} \ar[r]&  0\\
		\Upsilon(X')\ar[r]^{\Upsilon(h)} & \Upsilon(Y')\ar[r]^{}  & F' \ar[r] & 0
	}\]
commutative. Moreover, using \Cref{Lemma:YonedaEmbeddingIsLeftExact}, we find that $\alpha\circ k_f=0$ (where $k_f\colon \ker(f)\inflation X$ is the kernel of $f$). Since $h\circ \alpha \circ k_f= \beta \circ f\circ k_f=0$ and $h$ is an inflation, it follows that $\alpha \circ k_f=0$. Hence $\alpha$ factors through $f$ via a morphism $\gamma\colon Y\to X'$. Since $h\circ \gamma\circ f=h\circ \alpha =\beta \circ f$ and $f$ is a deflation, we get that $h\circ \gamma = \beta$. Hence $\Upsilon(\beta)$ factors through $\Upsilon(h)$, so $\eta$ must be $0$.

	To prove that $\eff(\EE)\subseteq \smodad(\EE)$ satisfies axiom \ref{A2}, assume we are given a morphism $\eta\colon F\to G$ with $F\in \smodad(\EE)$ and $G\in \eff(\EE)$. Choose projective presentations $\Upsilon(A)\xrightarrow{\Upsilon(f)}\Upsilon(B)\to F\to 0$ and $\Upsilon(C)\xrightarrow{\Upsilon(g)}\Upsilon(D)\to G\to 0$ where $f$ is admissible and $g$ is a deflation. Lifting $\eta$ to these projective presentation gives maps $\alpha\colon B\to D$ and $\beta\colon A\to C$  making the square 
	\[\xymatrix{
		A\ar[r]^{\beta}\ar[d]^f & C\ar[d]^{g}\\
		B\ar[r]^{\alpha} & D
	}\]
	commutative. By taking the pullback of $\alpha$ along $g$ we get the diagram \eqref{Thefamousdiagram} in \Cref{Lemma:TheFamousDiagramChase}, which induces the factorization 
	\[\xymatrix{
		\ker(\eta)\ar@{>->}[r] & F\ar@{->>}[r] & \im(\eta)\ar@{>->}[r] & G\ar@{->>}[r] & \coker(\eta)
	}\] of $\eta$ in $\Mod(\EE)$. Since $g'$ and $\begin{psmallmatrix}g& \alpha\end{psmallmatrix}$ are deflations, it follows that $\im(\eta)$ and $\coker(\eta)$ are in $\eff(\EE)$. Now let $A\stackrel{f'}{\deflation}B'\stackrel{f''}{\inflation}B$ be a deflation-inflation factorization of $f$. Taking the pullback $E'$ of $f''$ along $g'$, we get a commutative diagram 
\[\xymatrix{
			\ker(f)\ar[d]\ar@{>->}[r] & A\ar[d]\ar@{->>}[r]^{f'} & B'\ar@{=}[d]\\
			\ker(g)\ar@{=}[d]\ar@{>->}[r] & E'\ar@{>->}[d] \ar@{->>}[r]& B'\ar@{>->}[d]^{f''}\\
			\ker(g)\ar@{>->}[r]^{k'} & E\ar@{->>}[r]^{g'} & B
		}\]
where the rows are exact and where the composite $A\to E'\inflation E$ is equal to $\beta''$. Since the pullback of an inflation along a deflation is an inflation \cite[Proposition 2.15]{Buhler10}, the map $E'\inflation E$ is an inflation. Also, by \cite[Proposition 2.12]{Buhler10} the top left square is pushout square and the induced map $\ker(g)\oplus A\deflation E'$ is a deflation. Since $\begin{psmallmatrix}k'& \beta''\end{psmallmatrix}\colon \ker(g)\oplus A\to E$ is equal to the composite $\ker(g)\oplus A\deflation E'\inflation E$, it must be admissible. Hence $\ker (\eta)\in \smodad(\EE)$, which proves axiom \ref{A2}. Finally, by \Cref{Proposition:Torsion+A2IsTwoSidedPercolating}, $\eff(\EE)\subseteq \smodad(\EE)$ is a percolating subcategory.
\end{proof}

\begin{remark}
It was already shown in \cite[Proposition 4.7]{Fiorot20} that $\eff \EE$ lies two-sided filtering in $\smod \EE$ and in \cite[Lemma 9]{Schlichting06} that $\eff \EE$ is an abelian subcategory of $\smod \EE$ (for idempotent complete categories $\EE$).
\end{remark}

\begin{corollary}
	The quotient $\smodad(\EE)/\eff(\EE)$ is an exact category.
\end{corollary}

\begin{proof}
This follows directly from \Cref{Theorem:TwoSidedLocalizationTheorem}.
\end{proof}

\subsection{Auslander's formula for exact categories}\label{Section:AuslandersFormulaForExactCategories}

In this section $\EE$ denotes an exact category.

\begin{theorem}[Universal property]\label{Theorem:Universal property}
	\begin{enumerate}
		\item The Yoneda embedding $\Upsilon\colon \EE\to \smodad(\EE)$ is left exact.
		\item If $\CC$ is an exact category and $\Phi\colon \EE\to \CC$ is a left exact functor, then there exists a functor $\overline{\Phi}\colon \smodad(\EE)\to \CC$, unique up to isomorphism, which is exact and satisfies $\overline{\Phi}\circ \Upsilon=\Phi$.
	\end{enumerate} 
\end{theorem}

\begin{proof}
	\begin{enumerate}
		\item This follows from \Cref{Lemma:YonedaEmbeddingIsLeftExact}.
		\item Uniqueness is clear since any exact functor $\smodad(\EE)\to \CC$ is uniquely determined (up to isomorphism) by its restriction to the representable functors.  We prove existence.  For each $F\in\smodad(\EE)$, choose a projective presentation  $\Upsilon(A)\xrightarrow{\Upsilon(f)}\Upsilon(B)\to F\to 0$ with $f$ admissible (where if $F=\Upsilon(E)$ we simply choose $0\to \Upsilon(E)\xrightarrow{1}\Upsilon(E)\to 0$). Since $\Phi$ is left exact, it preserves admissible morphisms, and hence $\Phi(f)$ is admissible in $\CC$. We define $\overline{\Phi}(F)=\coker (\Phi(f))$. Let $\eta\colon F\to F'$ be a morphism in $\smodad(\EE)$ between objects with fixed projective presentations $\Upsilon(A)\xrightarrow{\Upsilon(f)}\Upsilon(B)\to F\to 0$ and $\Upsilon(A')\xrightarrow{\Upsilon(f')}\Upsilon(B')\to F'\to 0$. Choose maps $\alpha\colon A\to A'$ and $\beta\colon B\to B'$ in $\EE$ making the diagram
		\[\xymatrix{
		\Upsilon(A)\ar[r]^{\Upsilon(f)}\ar[d]^{\Upsilon(\alpha)} & \Upsilon(B)\ar[r]^{}\ar[d]^{\Upsilon(\beta)}  & F\ar[d]^{\eta} \ar[r]&  0\\
		\Upsilon(A')\ar[r]^{\Upsilon(f')} & \Upsilon(B')\ar[r]^{}  & F' \ar[r] & 0
	}\]
commutative. We define $\overline{\Phi}(\eta)\colon \overline{\Phi}(F)\to \overline{\Phi}(F')$ to be the unique morphism in $\CC$ making the diagram \[\xymatrix{
		\Phi(A)\ar[r]^{\Phi(f)}\ar[d]^{\Phi(\alpha)} & \Phi(B)\ar[r]^{}\ar[d]^{\Phi(\beta)}  & \overline{\Phi}(F)\ar[d]^{\overline{\Phi}(\eta)} \ar[r]&  0\\
		\Phi(A')\ar[r]^{\Phi(f')} & \Phi(B')\ar[r]^{}  & \overline{\Phi}(F') \ar[r] & 0
	}\]
	commutative. It is straightforward to check that $\overline{\Phi}(\eta)$ is independent of choice of $\alpha$ and $\beta$, and that $\overline{\Phi}\colon \smodad(\EE)\to \CC$ is a well-defined functor. It only remains to show that $\overline{\Phi}$ is exact. Let $F\inflation G\deflation H$ be a conflation in $\smodad(\EE)$. Choose projective presentations $\Upsilon(A)\xrightarrow{\Upsilon(f)}\Upsilon(B)\to F\to 0$ and $\Upsilon(C)\xrightarrow{\Upsilon(h)}\Upsilon(D)\to H\to 0$ of $F$ and $H$, and use the Horseshoe Lemma as in \Cref{Proposition:AdmissiblyPresentedLiesExtensionClosed} to obtain the following commutative diagram in $\EE$:
		\[\xymatrix{
			\ker(f)\ar@{>->}[d]\ar@{>->}[r] & \ker(f)\oplus \ker(h)\ar@{>->}[d]\ar@{->>}[r] & \ker(h)\ar@{>->}[d]\\
			A\ar[d]^f\ar@{>->}[r] & A\oplus C\ar[d] \ar@{->>}[r]& C\ar[d]^h\\
			B\ar@{>->}[r] & B\oplus D\ar@{->>}[r] & D
		}\] where the rows are are split kernel-cokernels pairs and the columns are left exact sequences. Since $\Phi$ is left exact and preserves split kernel-cokernel pairs, by applying the $3\times 3$ lemma twice \cite[Corollary 3.6]{Buhler10} we get that $\overline{\Phi}(F)\inflation \overline{\Phi}(G)\deflation \overline{\Phi}(H)$ is a conflation. This shows that $\overline{\Phi}$ is exact. \qedhere 
	\end{enumerate}
\end{proof}

\begin{corollary}\label{corollary:YonedaLeftAdjoint}
		 The Yoneda embedding $\Upsilon\colon \EE\to \smodad(\EE)$ has an exact left adjoint $L\colon \smodad(\EE)\to\EE$ satisfying $L\circ \Upsilon=1_{\EE}$ and $\ker(L)=\eff(\EE)$.
\end{corollary}

\begin{proof}
	Applying \Cref{Theorem:Universal property} to the identity functor on $\EE$ gives an exact functor $L\colon \smodad(\EE)\to \EE$ satisfying $L\circ \Upsilon=1_{\EE}$. Explicitly, given $F\in \smodad(\EE)$ with projective presentation $\Upsilon (A)\xrightarrow{\Upsilon(f)}\Upsilon (B)\to F$ where $f$ admissible, we have $L(F)\cong \coker (f)$. Note that $\Upsilon (B)\to\Upsilon(\coker(f))$ factors as $\Upsilon (Y)\to F \xrightarrow{\alpha} \Upsilon(\coker(f))$ for a unique map $\alpha$. One readily verifies that any map $F\to \Upsilon(Z)$ must factor through $\alpha$ in a unique way, and hence we have natural isomorphisms \[
	\Hom_{\smodad(\EE)}(F,\Upsilon (Z))\cong \Hom_{\EE}(\coker (f),Z)\cong \Hom_{\EE}(L(F),Z)
	\]
which shows that $L$ is left adjoint to $\Upsilon$. Finally, note that $L(F)\cong 0$ if and only if $\coker(f)\cong 0$. As $f$ is admissible, $\coker(f)\cong 0$ implies that $f$ is a deflation. This shows that $\ker(L)=\eff(\EE)$.
\end{proof}

We now prove Auslander's formula for exact categories.

\begin{theorem}[Auslander's formula for exact categories]\label{Theorem:AuslandersFormulaForExactCategories}
	The functor $L\colon \smodad(\EE)\to \EE$ induces an equivalence
	\[
	\smodad(\EE)/\eff(\EE)\cong \EE
	\]
	of exact categories. 
\end{theorem}

\begin{proof}
	As $\ker(L)=\eff(\EE)$, \Cref{Theorem:TwoSidedLocalizationTheorem}.\eqref{Item:Theorem:TwoSidedLocalizationTheorem2} yields an exact functor $L'\colon \smodad(\EE)/\eff(\EE)\to \EE$ such that $L'\circ Q=L$ where $Q\colon \smodad(\EE)\to \smodad(\EE)/\eff(\EE)$ is the localization functor. By \cite[Proposition~1.3]{GabrielZisman67}, $L'$ is an equivalence.
\end{proof}

\begin{remark}\label{Remark:WeSeeLAsALocalizationFunctor}
	The functors $L$ and $Q$ satisfy the same universal property (see \Cref{Definition:LocalizationWithRespectToMorphisms}) of a localization functor as $L=L'\circ Q$ and $L'$ is an equivalence of exact categories.
\end{remark}
	
\begin{corollary}\label{Corollary:VerdierLocalizationAlternativeDefinitionViaAuslander}
\begin{enumerate}
\item The Yoneda embedding $\EE\to \smodad(\EE)$ lifts to a triangle equivalence $\Kb(\EE)\to\Db(\smodad(\EE))$.
		\item there is a natural commutative diagram
			\[\xymatrix{
				\Db_{\eff(\EE)}(\smodad(\EE))\ar[r]\ar[d]^{\simeq} & \Db(\smodad(\EE))\ar[r]\ar[d]^{\simeq} & \Db(\smodad(\EE)/\eff(\EE))\ar[d]^{\simeq}\\
				\left\langle\Acb(\EE)\right\rangle_{\text{thick}}\ar[r] & \Kb(\EE)\ar[r] & \Db(\EE)
			}\] whose rows are Verdier localization sequences.
	\end{enumerate}
\end{corollary}

\begin{proof}
	\begin{enumerate}
		\item	This follows immediately since any object in $\smodad(\EE)$ has finite projective dimension by \Cref{Lemma:YonedaEmbeddingIsLeftExact}.
		\item The upper Verdier localization sequence is induced by the localization sequence \[
		\eff(\EE)\to\smodad(\EE)\to \smodad(\EE)/\eff(\EE).\] The lower sequence is simply the definition of $\Db(\EE)$ for an exact category. The functor $L'$ induces the equivalence $\Db(\smodad(\EE)/\eff(\EE))\xrightarrow{\simeq} \Db(\EE)$ and the Yoneda embedding induces the equivalence $\Kb(\EE)\rightarrow \Db(\smodad(\EE))$. The result follows.\qedhere
	\end{enumerate}
\end{proof}

In the last part of this section we show that $\EE\mapsto \smodad(\EE)$ can be considered as a $2$-functor.

\begin{definition}
\begin{enumerate}
\item  Let $\Ex$ denote the $2$-category with objects exact categories, with $1$-morphisms exact functors, and with $2$-morphisms natural transformations.
\item Let $\Ex_L$ denote the $2$-category with objects exact categories, with $1$-morphisms left exact functors, and with $2$-morphisms natural transformations.
\end{enumerate}
\end{definition}

\begin{corollary}\label{Corollary:Left2Adjoint}
The association $\EE\mapsto \smodad(\EE)$ induces a $2$-functor
\[
\smodad(-)\colon \Ex_L\to \Ex
\]
which is left adjoint to the inclusion $\Ex\subseteq \Ex_L$, i.e. for which the restriction 
\[
\Hom_{\Ex}(\smodad(\EE),\EE')\to \Hom_{\Ex_L}(\EE,\EE')
\]
via $\Upsilon\colon \EE\to \smodad(\EE)$  is an equivalence of categories.

\end{corollary}

\begin{proof}
Using  \Cref{Theorem:Universal property}, for each left exact functor $\Phi\colon \EE\to \EE'$ we choose an exact functor \[
\smodad(\Phi)\colon \smodad(\EE)\to \smodad(\EE')
\]
 making the square  
\[\xymatrix@C=3em{
		\EE\ar[r]^{\Phi}\ar[d]^{\Upsilon} & \EE'\ar[d]^{\Upsilon}\\
		\smodad(\EE)\ar[r]^{\smodad(\Phi)} & \smodad(\EE')
	}\]
	commutative. Since each natural transformation $\Phi\to \Phi'$ can be extended uniquely to a natural transformation $\smodad(\Phi)\to \smodad(\Phi')$, we get a $2$-functor $\smodad(-)\colon \Ex_L\to \Ex$. Finally, the restriction $\Hom_{\Ex}(\smodad(\EE),\EE')\to \Hom_{\Ex_L}(\EE,\EE')$ is an equivalence by \Cref{Theorem:Universal property} and the fact that any natural transformation $\Psi\circ \Upsilon\to \Psi'\circ \Upsilon$ where $\Psi,\Psi'\in \Hom_{\Ex}(\smodad(\EE),\FF)$ can be extended uniquely to a natural transformation $\Psi\to \Psi'$.
\end{proof}

\subsection{\texorpdfstring{The torsion pair $(\eff(\EE),\cogen(\QQ))$ in $\smodad(\EE)$}{Torsion theory in admissibly presented functors}}

Throughout this section, fix an exact category $\EE$ and set $\QQ=\Proj(\smodad(\EE))$. In this section, we show that $(\eff(\EE),\FF)$ as in \Cref{Proposition:EffAreTorsion} satisfies several useful properties. In the next section, we show that these properties can be used to characterize the image of the $2$-functor $\smodad(-)\colon \Ex_L\to \Ex$, which leads to a generalization of Auslander correspondence, cf.~\cite{Auslander71}.

\begin{definition}
	Let $\mathcal{X}\subseteq \EE$ be a subcategory and $n\geq 0$ an integer.
	\begin{enumerate}
		\item We write $\gen_k(\mathcal{X})$ for the full subcategory of $\EE$ consisting of those objects $E$ such that there is an exact sequence $X_k\to \cdots \to X_2\to  X_1\to E\to 0$ with $X_i\in \mathcal{X}$ for $1\leq i\leq k$. If $k=1$ we write $\gen_1(\mathcal{X})=\gen(\mathcal{X})$.
		\item We write $\cogen_k(\mathcal{X})$ for the full subcategory of $\EE$ consisting of those objects $E$ such that there is an exact sequence $0\to E\to X_1\to X_{2}\to \cdots \to X_k$ with $X_i\in \mathcal{X}$ for $1\leq i\leq k$. If $k=1$ we write $\cogen_1(\mathcal{X})=\cogen(\mathcal{X})$.
	\end{enumerate}
\end{definition}

In the following ${}^{\perp}\QQ$ and $\cogen(\QQ)$ are defined inside the exact category $\smodad(\EE)$. 

\begin{proposition}\label{Proposition:TorsionTheoryProperties}
	Consider the torsion pair $(\eff(\EE),\FF)$ in $\smodad(\EE)$ as in \Cref{Proposition:EffAreTorsion}. The following hold:
	\begin{enumerate}
		\item\label{Proposition:TorsionTheoryPropertiesA} $\eff(\EE)={}^{\perp}\QQ$.
		\item\label{Proposition:TorsionTheoryPropertiesB} $\FF=\cogen(\QQ)$.
		\item\label{Proposition:TorsionTheoryPropertiesC} $\Ext^1_{\smodad(\EE)}(\eff(\EE),\QQ)=0$.
		\item\label{Proposition:TorsionTheoryPropertiesD} The Yoneda functor $\Upsilon\colon \EE\to \smodad(\EE)$ gives an equivalence $\EE\simeq \QQ$.
	\end{enumerate}
\end{proposition}

\begin{proof}
	\begin{enumerate}
		\item Clearly $\QQ\subseteq \FF$, hence the inclusion $\eff(\EE)\subseteq {}^{\perp}\QQ$ follows from $\Hom_{\smodad(\EE)}(\eff(\EE),\FF)=0$. Conversely, let $G\in {}^{\perp}\QQ$. Choose a projective presentation $\Upsilon(A)\xrightarrow{\Upsilon(f)}\Upsilon(B)\to G\to 0$ with $f$ admissible in $\EE$. Since the composite $\Upsilon(A)\xrightarrow{\Upsilon(f)}\Upsilon(B)\to \Upsilon(\coker (f))$ is $0$, the map $\Upsilon(B)\to \Upsilon(\coker(f))$ factors through $G$ and must therefore be $0$ since $G\in {}^{\perp}\QQ$. Using that $\Upsilon$ is faithful, it follows that the map $B\to \coker(f)$ is $0$, so $f$ must be a deflation. This shows that $G\in \eff(\EE)$. 
	
	\item We first show $\FF\subseteq \cogen(\QQ)$. Choose $F\in \FF$ and let $\Upsilon(X)\xrightarrow{\Upsilon(f)}\Upsilon(Y)\to F\to 0$ be a projective presentation with $f$ an inflation. Let $g\colon Y\deflation Z$ denote the cokernel of $f$. By considering the two rightmost squares in diagram \eqref{Thefamousdiagram} in \Cref{Lemma:TheFamousDiagramChase}, we get that the commutative diagram 
	\[\xymatrix{
		X\ar[r]\ar@{>->}[d]^f & 0\ar[r]\ar[d] & Y\ar@{->>}[d]^g\\
		Y\ar@{->>}[r]^{g} & Z\ar@{=}[r] & Z
	}\] yields a short exact sequence $F\inflation \Upsilon(Z)\deflation H$ where $H=\coker(\Upsilon(g))\in \eff(\EE)$. Thus $F\in\cogen(\QQ)$.
	
	Conversely, let $G\in \cogen(\QQ)$, then there is an inflation $G\inflation \Upsilon(Z)$ by definition. As $(\eff(\EE),\FF)$ is a torsion pair, $G$ fits into a conflation $F\inflation G\deflation H$ with $F\in \eff(\EE)$ and $H\in \FF$. As $\Upsilon(Z)\in \FF$, the composition $F\inflation G\inflation \Upsilon(Z)$ is zero. It follows that $F\cong 0$ and thus $G\cong H\in \FF$. This shows the equality $\FF=\cogen(\QQ)$.
	
	\item Let $E\in \eff(\EE)$. Then there exists a conflation $\ker(e)\stackrel{k}{\inflation} A\stackrel{e}{\deflation} B$ in $\EE$ which gives a projective resolution  
	\[
	0\to \Upsilon(\ker(e))\xrightarrow{\Upsilon(k)}\Upsilon(A)\xrightarrow{\Upsilon(e)}\Upsilon(B)\to E\to 0
	\]
	 of $E$. Let $\Upsilon(Z)\in \QQ$ be arbitrary. Applying the functor $\Hom(-, \Upsilon(Z))$ to the projective resolution and using that the $\Upsilon$ is fully faithful yields the left exact sequence
	\[
	0\to \Hom(B,Z)\xrightarrow{-\circ e}\Hom(A,Z)\xrightarrow{-\circ k} \Hom(\ker(e),Z)\]
of abelian groups. Since the sequence is exact at $\Hom(A,Z)$, it follows that  $\Ext^1(E, \Upsilon(Z))=0$.
	
	\item Let $P\in \QQ$. Note that $P$ is torsion-free and thus there exists a conflation $A\stackrel{f}{\inflation} B \deflation C$ such that $\Upsilon(A)\xrightarrow{\Upsilon(f)}\Upsilon(B)\to P$ is a short exact sequence. Since $P$ is projective, the latter sequence is split and thus $f$ is a split inflation as $\Upsilon$ is fully faithful. It follows that $P\cong \Upsilon(C)$, which proves the claim.\qedhere
	\end{enumerate}
\end{proof}

\begin{corollary}\label{Corollary:(Weak)IdempotentCompleteness}
	Let $\EE$ be an exact category. The following hold:
	\begin{enumerate}
		\item The category $\EE$ is weakly idempotent complete if and only if $\smodad(\EE)$ is weakly idempotent complete.
		\item The category $\EE$ is idempotent complete if and only if $\smodad(\EE)$ is idempotent complete.
	\end{enumerate}
\end{corollary}

\begin{proof}
The if direction follows as $\EE$ is equivalent to $\Proj(\smodad(\EE))$. Conversely, assume that $\EE$ is (weakly) idempotent complete, and let $f$ be an idempotent endomorphism (respectively a map which admits a right inverse) in $\smodad (\EE)$. Then $Q(f)$ is idempotent (respectively admits a right inverse), and hence $Q(f)$ is admissible by assumption. Therefore, by \Cref{Theorem:QReflectsAdmissibles} we have that $f$ is admissible in $\smodad (\EE)$. Hence $\smodad (\EE)$ is (weakly) idempotent complete, which proves the claim.
\end{proof}

\subsection{Admissibly presented functors and the category of sheaves}\label{Subsection:Interpretation}

Throughout this section, $\EE$ denotes an exact category. The category $\smodad(\EE)$ is defined as a subcategory of $\smod(\EE)$, depending on the exact structure.  In this section, we show how the category $\smodad(\EE)$ occurs naturally by comparing the localization sequence $\eff(\EE)\to \smodad(\EE)\to \smodad(\EE)/\eff(\EE)$ from \Cref{Theorem:AuslandersFormulaForExactCategories} to similar localization sequences appearing in the literature (see \Cref{{Proposition:InverseImage}}). In order to do so, we first recall the notion of a weakly effaceable functor as introduced by Grothendieck in \cite{Grothendieck57} (following the terminology of \cite{DmitryLowen15}).  

\begin{definition}\label{definition:WeaklyEffaceable}
	A functor $F\in \Mod(\EE)$ is called \emph{weakly effaceable} if for every $C\in \EE$ and every $x\in F(C)$, there exists a deflation $\lambda\colon C'\deflation C$ such that $F(\lambda)(x)=0$. The full subcategory of $\Mod(\EE)$ of weakly effaceable functors is denoted by $\Weff(\EE)$.
\end{definition}

The following proposition states that effaceable functors are weakly effaceable functors which are finitely presented (see also \cite[Lemma 2.13]{Enomoto2020A} or \cite[Lemma 9]{Schlichting06}).

\begin{proposition}\label{Proposition:WeaklyEffaceableAndFinitelyPresentedIsEffaceable}
	Let $F\in \Mod(\EE)$.  Then $F\in \eff(\EE)$ if and only if $F\in \Weff(\EE)\cap \smod(\EE)$.
\end{proposition}

\begin{proof}
	Assume that $F\in \eff(\EE)$ and let $f\colon A\deflation B$ be a deflation in $\EE$ such that $F\cong \coker(\Upsilon(f))$. For any $C\in \EE$ and $x'\in F(C)$, there exists a morphism $x\colon C\to B$ representing $x'$. By axiom \ref{R2}, we obtain the following pullback square:
	\[\xymatrix{
		C'\ar[r]^{y}\ar@{->>}[d]^{\lambda} & A\ar@{->>}[d]^{f}\\
		C\ar[r]^{x} & B
	}\] Commutativity of the square yields $F(\lambda)(x)=0$. This shows that $F\in \Weff(\EE)\cap \smod(\EE)$.
	
	Conversely, let $F\in \Weff(\EE)\cap\smod(\EE)$. As $F$ is finitely presented, there is a map $f\colon A\to B$ such that $F\cong \coker(\Upsilon(f))$. The identity $1_B$ yields an element $1\in F(B)$. As $F\in \Weff(\EE)$, there exists a deflation $\lambda\colon C\deflation B$ such that $F(\lambda)(1)=0$. Hence there exists a map $u\colon C\to A$ such that $fu=\lambda 1_B=\lambda$. By \cite[Proposition~3.4]{HenrardvanRoosmalen20Obscure}, the map $\begin{psmallmatrix}0&f\end{psmallmatrix}\colon C\oplus A\deflation B$ is a deflation in $\EE$. Clearly, $\coker(\Upsilon(\begin{psmallmatrix}0&f\end{psmallmatrix}))\cong F$ and thus $F\in \eff(\EE)$.
\end{proof}

\begin{remark}
	It is shown in \cite[Lemma~2.13]{Enomoto2020A} that $\eff(\EE)\subseteq \coh(\EE)$ where $\coh(\EE)$ are the \emph{coherent functors}, i.e.~finitely presented functors such that every finitely generated subfunctor is also finitely presented. Hence, $\eff(\EE)=\Weff(\EE)\cap \coh(\EE)$. In particular, $\eff(\EE)\subseteq \coh(\EE)$ is a Serre subcategory and hence $\eff(\EE)$ is abelian (see also \cite[Lemma 9]{Schlichting06}).
\end{remark}

The following definition is given in \cite[III.2]{Gabriel62}.

\begin{definition}\label{Definition:LocalizingSerreSubcategory}
	Let $\AA$ be an abelian category. A full subcategory $\SS\subseteq \AA$ is called a \emph{localizing Serre subcategory} of $\AA$ if $\SS\subseteq \AA$ is a Serre subcategory such that the localization functor $Q\colon\AA\to \AA/\SS$ has a right adjoint $R\colon \AA/\SS\to \AA$. 
\end{definition}

\begin{proposition}\label{Proposition:BasicPropertiesLocalizingSubcategories}
	The category $\Weff(\EE)$ is a localizing Serre subcategory of $\Mod(\EE)$. Furthermore, the following properties hold.
	\begin{enumerate}
		\item\label{Item:Proposition:BasicPropertiesLocalizingSubcategories1} We obtain a localization sequence 
	\[\Weff(\EE)\to \Mod(\EE)\xrightarrow{Q_w} \operatorname{Lex}(\EE)\]
	where $\operatorname{Lex}(\EE)$ is the category of left exact functors. 
		\item\label{Item:Proposition:BasicPropertiesLocalizingSubcategories2} The right adjoint $R\colon \operatorname{Lex}(\EE)\to \Mod(\EE)$ is the canonical inclusion.
		\item\label{Item:Proposition:BasicPropertiesLocalizingSubcategories3} Denote by $\mu\colon 1 \Rightarrow R \circ Q_w$ the unit of the adjunction $Q_w\dashv R$. For each $F\in \Mod(\EE)$, we have that $\ker(\mu_F),\coker(\mu_F)\in \Weff(\EE)$. Additionally, if $F\in \operatorname{Lex}(\EE)$ then $\mu_F\colon F \to R\circ Q_w(F)$ is an isomorphism.
	\end{enumerate}
\end{proposition}

\begin{proof}
	It is shown in \cite[Appendix A.2]{Keller90} that $\Weff(\EE)$ is a Serre subcategory of $\Mod(\EE)$ closed under direct limits that gives rise to the localisation sequence \eqref{Item:Proposition:BasicPropertiesLocalizingSubcategories1}. By \cite[III.4~Proposition~8]{Gabriel62}, $\Weff(\EE)$ is a localizing Serre subcategory of $\Mod(\EE)$. Property \eqref{Item:Proposition:BasicPropertiesLocalizingSubcategories2} follows from \cite[III.5.b]{Gabriel62}. Finally, \eqref{Item:Proposition:BasicPropertiesLocalizingSubcategories3} follows from \cite[III.3~Lemme~2, Proposition~3 and Corollaire]{Gabriel62}.
\end{proof}

\begin{corollary}
	The category $\eff(\EE)$ is a deflation-percolating subcategory of $\smod(\EE)$.
\end{corollary}

\begin{proof}
	By \Cref{Proposition:BasicPropertiesLocalizingSubcategories}, $\Weff(\EE)$ is a Serre subcategory of $\Mod(\EE)$. By \Cref{Proposition:WeaklyEffaceableAndFinitelyPresentedIsEffaceable}, $\eff(\EE)=\Weff(\EE)\cap \smod(\EE)$ and thus $\eff(\EE)\subseteq \smod(\EE)$ is a Serre subcategory. That $\eff(\EE)\subseteq \smod(\EE)$ satisfies axiom \ref{A2} follows from \Cref{Lemma:TheFamousDiagramChase}.
\end{proof}

\begin{remark}\label{Remark:RecoverRump}
	By \Cref{Proposition:PropertiesOfWeakIsomorphisms}, $\Sigma_{\eff(\EE)}\subseteq \text{Mor}(\smod(\EE))$ is a right multiplicative system and thus we obtain the localization functor $Q\colon \smod(\EE)\to \smod(\EE)[\Sigma^{-1}_{\eff(\EE)}]$. In general, it is not clear whether $\eff(\EE)\subseteq\smod(\EE)$ is an inflation-percolating subcategory and thus it is not clear whether $\mod(\EE)$ induces an exact structure on $\smod(\EE)[\Sigma^{-1}_{\eff(\EE)}]$.  Nonetheless, by \cite[Theorem~1.2]{HenrardVanRoosmalen19a}, the localization functor $Q\colon \smod(\EE)\to \smod(\EE)[\Sigma^{-1}_{\eff(\EE)}]$ endows $\smod(\EE)[\Sigma^{-1}_{\eff(\EE)}]$ with a deflation-exact structure (that is, $\smod(\EE)[\Sigma^{-1}_{\eff(\EE)}]$ satisfies axioms \ref{R0},\ref{R1}, and \ref{R2}). Furthermore, $\smod(\EE)[\Sigma^{-1}_{\eff(\EE)}]$ satisfies the universal property of a quotient of deflation-exact categories. Therefore, we still write $\smod(\EE)[\Sigma^{-1}_{\eff(\EE)}]\simeq \smod(\EE)/\eff(\EE)$.
\end{remark}

We now consider the following commutative diagram for an exact category $\EE$:
\[\xymatrix{
	\eff(\EE)\ar[r]\ar@{=}[d] & \smodad(\EE)\ar[r]\ar[d] & \smodad(\EE)/\eff(\EE)\ar[d]\simeq \EE \\
	\eff(\EE)\ar[r]\ar[d] & \smod(\EE)\ar[r]\ar[d]^{\iota} & \smod(\EE)/\eff(\EE)\ar[d]\\
	\Weff(\EE)\ar[r] & \Mod(\EE)\ar[r]_-{Q_w} & \text{Lex}(\EE) 
}\]
Here, the bottom row is obtained by \Cref{Proposition:BasicPropertiesLocalizingSubcategories}. 
The middle row is obtained by \Cref{Remark:RecoverRump} and the top row is Auslander's formula as in \Cref{Theorem:AuslandersFormulaForExactCategories}. 

\begin{remark}
\begin{enumerate}
	\item Following \cite[Sections~2.4 and 2.5]{DmitryLowen15}, the exact structure on $\EE$ induces a Grothendieck topology (there called the additive single morphism topology) on $\EE$.  Using this topology, one can define a category $\text{Sh}^{\text{add}}(\EE)$ of (additive) sheaves on $\EE$ and show that $\text{Sh}^{\text{add}}(\EE)\simeq\text{Lex}(\EE)$ (see \cite[Section~2.1]{DmitryLowen15}).
	\item In \cite{Rump20}, Rump considers a localization theory of left abelian categories. By \cite[Proposition~5]{Rump04}, left abelian categories with enough projectives are precisely categories of the form $\smod(\AA)$ for an additive category $\AA$. Any left abelian category has an exact structure given by all kernel-cokernel pairs \cite[Proposition~2]{Rump20}. Rump defines a category $Q_l(\EE)$ which is obtained as a localization of the left abelian category $\smod(\EE)$. Moveover, $Q_l(\EE)$ inherits an exact structure from $\smod(\EE)$ (see \cite[Theorem~1]{Rump20}).
	
	There is a natural functor $\theta\colon\smod(\EE)/\eff(\EE)\to Q_l(\EE)$, which is a localization functor by \cite[Lemma~5.2]{HenrardVanRoosmalen19b}. By \cite[Example~3]{Rump20}, $\eff(\EE)$ is a percolating subcategory if $\EE$ is left quasi-abelian and then $\theta$ is an equivalence. Hence, the second row reduces to Rump's localization sequence in some interesting cases. Furthermore, \cite[Corollary~2]{Rump20} shows that the embedding $\EE\to \textbf{Q}_l(\EE)$ is an equivalence if and only if $\EE$ is abelian.  In this case, the top and middle line of this diagram are equivalent.  
	
	\item When $\EE$ has weak kernels (so that $\smod(\EE)$ is abelian), the middle line is considered in \cite{Ogawa19}. In \cite[Theorem~2.9]{Ogawa19}, Ogawa shows that if the localization functor $\smod(\EE)\to \smod(\EE)/\eff(\EE)$ has a right adjoint, then there is an equivalence $\smod(\EE)/\eff(\EE)\simeq \operatorname{lex}(\EE)$, where $\operatorname{lex}(\EE)$ is the category finitely presented left exact functors on $\EE$.  The right adjoint $\operatorname{lex}(\EE)\to \smod(\EE)$ is then the canonical inclusion.
	\item When $\EE$ has kernels, the quotient $\smod(\EE)/\eff(\EE)$ is equivalent to $\mathcal{LH}(\EE)$ where $\mathcal{LH}(\EE)$ is the left heart of the category $\EE$ (see \cite[Theorem~3.17]{HenrardKvammevanRoosmalenWegner21}). Here, the left heart $\mathcal{LH}(\EE)$ generalizes the construction of the left heart of a quasi-abelian category as in \cite{Schneiders99}.
	\item Despite the above remarks, it is unclear whether $\smod(\EE)/\eff(\EE)$ is an exact category in general.
	\end{enumerate}
\end{remark}

We conclude this section by recovering $\smodad(\EE)$ directly from the bottom two localization sequences. We need the following lemma.

\begin{lemma}\label{Lemma:ResolutionsOfEffaceables}
	Let $E\in \eff(\EE)$ and let $\Upsilon(\ker(g))\inflation \Upsilon(X)\xrightarrow{\Upsilon(g)}\Upsilon(Y)\deflation E$ be a resolution of $E$.
	\begin{enumerate}
		\item There exists an $X'\in \EE$ such that $\begin{psmallmatrix}0&g\end{psmallmatrix}\colon X'\oplus X\deflation Y$ is a deflation. 
		\item If $\EE$ is weakly idempotent complete, $g$ is a deflation.
	\end{enumerate}
\end{lemma}

\begin{proof}
	As $E\in \eff(\EE)$, there is a resolution $\Upsilon(\ker(f))\to \Upsilon(A)\xrightarrow{\Upsilon(f)}\Upsilon(B) \deflation E$ with $f$ a deflation. By the Comparison Theorem (see for example \cite[Theorem~12.4]{Buhler10}), there is a homotopy equivalence between the two resolutions. As the Yoneda embedding is faithful, we obtain a homotopy equivalence between $\ker(g)\to X\xrightarrow{g} Y\to 0$ and $\ker(f) \to A\xrightarrow{f} B\to 0$. By \cite[Proposition~10.14]{Buhler10}, $g$ is a deflation in the weak idempotent completion $\widehat{\EE}$ of $\EE$. This shows the first part, the second now follows from \cite[Proposition~3.4]{HenrardvanRoosmalen20Obscure}.	
\end{proof}

\begin{proposition}\label{Proposition:InverseImage}
	The category $\smodad(\EE)$ is obtained as the inverse image $(Q_w\circ \iota)^{-1}(Q_w\circ\Upsilon(\EE))$.
\end{proposition}

\begin{proof}
Let $F\in (Q_w\circ \iota)^{-1}(Q_w\circ\Upsilon(\EE))$. As $F\in \smod(\EE)$, there is a map $f\colon A\to B$ such that $F\cong \coker(\Upsilon(f))$. As $Q_w\circ \iota(F)\in Q_w\circ\Upsilon(\EE)$, there is a $C\in \EE$ such that $Q_w\circ \iota(F)\cong Q_w\circ\Upsilon(C)$. The unit of the adjunction $Q_w\dashv R$ yields a map $\mu_{F}\colon \iota(F)\to RQ_w\iota(F)$. Note that $RQ_w\iota(F)\cong RQ_w\Upsilon(C)$ and $RQ_w\Upsilon(C)\cong \Upsilon(C)$ by \Cref{Proposition:BasicPropertiesLocalizingSubcategories}.\eqref{Item:Proposition:BasicPropertiesLocalizingSubcategories3} as $\Upsilon(C)\in \operatorname{Lex}(\EE)$. Hence we obtain the following commutative diagram
	\[\xymatrix{
	& \Upsilon(A)\ar[d]^{\Upsilon(f)} &&\\
	& \Upsilon(B)\ar@{.>}[rd]^{\Upsilon(g)}\ar@{->>}[d] &&\\
		\ker(\mu_F)\ar@{>->}[r] & F\ar[r]^-{\sim} & \Upsilon(C)\ar@{->>}[r] & \coker(\mu_F)\\
	}\] in $\Mod(\EE)$. Here, $\ker(\mu_F),\coker(\mu_F)\in \Weff(\EE)$ by \Cref{Proposition:BasicPropertiesLocalizingSubcategories}.\eqref{Item:Proposition:BasicPropertiesLocalizingSubcategories3}. As the Yoneda functor is fully faithful, there is a map $g\colon B\to C$ such that the above diagram commutes. Note that $\coker(\mu_F)\cong \coker(\Upsilon(g))$ and thus $\coker(\mu_F)\in \eff(\EE)$ as it is finitely presented (see \Cref{Proposition:WeaklyEffaceableAndFinitelyPresentedIsEffaceable}). By \Cref{Lemma:ResolutionsOfEffaceables}, there is an object $B'\in \EE$ such that $\begin{psmallmatrix}0&g\end{psmallmatrix}\colon B'\oplus B\to C$ is a deflation. Write $K$ for the kernel of $\begin{psmallmatrix}0&g\end{psmallmatrix}$. Consider the commutative diagram:
	\[\xymatrix{
		B'\oplus A\ar@{=}[r]\ar[d]^h & B'\oplus A\ar[r]\ar[d]^{\begin{psmallmatrix}1_B' &0\\0&f\end{psmallmatrix}} & K\ar[r]\ar[d] & 0\ar[d]\\
		K\ar@{>->}[r] & B'\oplus B\ar@{=}[r] & B'\oplus B\ar@{->>}[r]^-{\begin{psmallmatrix}0&g\end{psmallmatrix}} & C\\
	}\] Using \Cref{Lemma:TheFamousDiagramChase}, the above diagram yields that $\ker(\mu_F)\cong \coker(\Upsilon(h))$ and thus $\ker(\mu_F)\in \eff(\EE)$. Using \Cref{Lemma:ResolutionsOfEffaceables} once more, one readily verifies that $F\in \smodad(\EE)$ as required. This shows $(Q_w\circ \iota)^{-1}(Q_W\circ\Upsilon(\EE))\subseteq \smodad(\EE)$.
	
	Conversely, let $F\in \smodad(\EE)$ and let $f\colon X\to Y$ be an admissible morphism in $\EE$ such that $F\cong \coker(\Upsilon(f))$. Commutativity of the diagram yields that $Q_w\circ \iota(F)=Q_w\circ\Upsilon\circ L(F)$. By the proof of \Cref{corollary:YonedaLeftAdjoint}, we find $L(F)\cong \coker(f)$. The result follows.	
\end{proof}
\section{Auslander correspondence for exact categories}\label{section:Correspondence}

In this section we show that $\smodad(-)\colon \Ex_L\to \Ex$ gives an equivalence onto a subcategory $\AEx$ of $\Ex$, which we describe explicitly (see below). In particular, we obtain an extension of Auslander correspondence to exact categories. 

\subsection{\texorpdfstring{The $2$-category $\AEx$}{The 2-category AEx}}

For an exact category $\EE$ with enough projective objects, we let $\gldim(\EE)$ denote the global dimension of $\EE$.

\begin{definition}\label{AuslanderExactCategories}
Let $\EE$ be an exact category with enough projectives $\PP=\Proj(\EE)$. We say that $\EE$ is an \emph{Auslander exact category} if it satisfies the following conditions:
	\begin{enumerate}
		\item[(i)] $({}^{\perp}\PP,\cogen(\PP))$ is a torsion pair in $\EE$.
		\item[(ii)] ${}^{\perp}\PP$ satisfies axiom \ref{A2}, i.e. all morphisms $f\colon E\to E'$ with $E'\in {}^{\perp}\PP$ are admissible with image in ${}^{\perp}\PP$.
		\item[(iii)]\label{Item:AuslanderExactCategories} $\Ext^1_{\EE}({}^{\perp}\PP,\PP)=0$.
		\item[(iv)] $\gldim(\EE)\leq 2$.
	\end{enumerate}
\end{definition}

\begin{definition}\label{Definition:TheTwo2Categories}
Let $\AEx$ denote the $2$-category with objects Auslander exact categories, with $1$-morphisms exact functors preserving projective objects, and with $2$-morphisms natural transformations. 
\end{definition}

\begin{corollary}\label{Corollary:ImageModAdmissible}
$\smodad(-)\colon \Ex_L\to \Ex$ factors through $\AEx$. 
\end{corollary}

\begin{proof}
By \Cref{Lemma:YonedaEmbeddingIsLeftExact} we know that $\gldim(\smodad(\EE))\leq 2$. The claim now follows from \Cref{Proposition:EffAreTorsion} and \Cref{Proposition:TorsionTheoryProperties}.
\end{proof}

\subsection{\texorpdfstring{The equivalence $\Ex_L\simeq\AEx$}{The equivalence ExL = AEx}}

In this section we show that $\smodad(-)\colon \Ex_L\to \AEx$ is an equivalence of $2$-categories.

\begin{proposition}\label{Lemma:TorsionFreeHasWeakInflationToProjective}
Let $\EE\in \AEx$ and let $\PP=\Proj(\EE)$. The following hold:
\begin{enumerate}
\item ${}^{\perp}\PP\subseteq \EE$ is a percolating subcategory of $\EE$.
\item\label{Item:TorsionFreeHasWeakInflationToProjective1} Any weak ${}^{\perp}\PP$-isomorphism $P\stackrel{\sim}{\rightarrow}M$ with $P\in \PP$ is a split inflation.
		\item\label{Item:TorsionFreeHasWeakInflationToProjective3} For every $M\in \cogen(\PP)$, there exists a conflation $M\stackrel{\sim}{\inflation}P\deflation T$ with $P\in \PP$ and $T\in {}^{\perp}\PP$.
\end{enumerate} 
\end{proposition}

\begin{proof}
\begin{enumerate}
\item This follows from \Cref{Proposition:Torsion+A2IsTwoSidedPercolating}.
\item  If $f\colon P\stackrel{\sim}{\rightarrow}M$ be a weak ${}^{\perp}\PP$-isomorphism, then $f$ is admissible and $\ker(f),\coker(f)\in {}^{\perp}\PP$. Since $\Hom({}^{\perp}\PP,P)=0$, it follows that $\ker(f)=0$, and hence $f$ is an inflation. Since $\Ext^1_{\EE}(\coker (f),P)=0$, we get that $f$ must be a split inflation.

		\item Let $M\in \cogen(\PP)$. By definition there is a conflation $M\inflation Q \deflation N$ with $Q\in \PP$. Choose a conflation $\tt N\inflation N \deflation \ff N$ with $\tt N\in {}^{\perp}\PP$ and $\ff N\in \cogen(\PP)$. Taking the pullback of $Q\deflation N$ along $\tt N\inflation N$, we obtain the following commutative diagram:
		\[\xymatrix{
			& \ff N\ar@{=}[r] & \ff N\\
			M\ar@{>->}[r] & Q\ar@{->>}[r]\ar@{->>}[u] & N\ar@{->>}[u]\\
			M\ar@{>->}[r]\ar@{=}[u] & P\ar@{->>}[r]\ar@{>->}[u] & \tt N\ar@{>->}[u]
		}\] Since $\ff N\in \cogen(\PP)$ and $\gldim (\EE)\leq 2$, it follows that $\ff N$ have projective dimension $\leq 1$. Thus $P$ must be projective. The lower row is therefore the desired conflation. \qedhere
\end{enumerate}
\end{proof}

\begin{proposition}\label{Proposition:KillingTorsionGivesProjectives}
	Let $\EE\in \AEx$ and let $\PP=\Proj(\EE)$. The composition $\PP\stackrel{\iota}{\hookrightarrow} \EE\xrightarrow{Q} \EE/{}^{\perp}\PP$ is an equivalence of categories.
\end{proposition}

\begin{proof}
	We first show that the composition $Q\circ \iota$ is essentially surjective. Let $M\in \Ob(\EE/{}^{\perp}\PP)=\Ob(\EE)$. Choose a conflation $\tt M\inflation M\stackrel{\sim}{\deflation} \ff M$ with $\tt M\in {}^{\perp}\PP$ and $\ff M\in \cogen(\PP)$. By \Cref{Lemma:TorsionFreeHasWeakInflationToProjective} . \eqref{Item:TorsionFreeHasWeakInflationToProjective3} there exists a conflation $\ff M\stackrel{\sim}{\inflation}P\deflation T$ with $P\in \PP$ and $T\in {}^{\perp}\PP$. The composition $M \stackrel{\sim}{\deflation} \ff M \stackrel{\sim}{\inflation}P$ yield that $M\cong \ff M\cong P$ in $\Sigma_{{}^{\perp}\PP}^{-1}\EE$. Hence $Q\circ \iota$ is essentially surjective.
	
	We now show that $Q\circ \iota$ is faithful. Let $f\colon P_1\to P_2$ be a map in $\PP$ such that $Q(f)=0$. As ${}^{\perp}\PP$ is a percolating subcategory of $\EE$, the set $\Sigma_{{}^{\perp}\PP}$ is a left multiplicative system. It follows that there exists a weak isomorphism $P_2\stackrel{\sim}{\rightarrow}Y$ such that the composition $P_1\stackrel{f}{\rightarrow}P_2\stackrel{\sim}{\rightarrow}Y$ is zero. By \Cref{Lemma:TorsionFreeHasWeakInflationToProjective}.\eqref{Item:TorsionFreeHasWeakInflationToProjective1}, $P_2\stackrel{\sim}{\rightarrow}Y$ is monic. It follows that $f$ is zero.
	
	Lastly, we show that $Q\circ \iota$ is full. As $\Sigma_{{}^{\perp}\PP}$ is a left multiplicative system, we can represent a morphism $f\in \Hom_{\EE/{}^{\perp}\PP}(P_1,P_2)$ as a roof $\xymatrix{P_1\ar[r] & X &\ar[l]_{\sim} P_2}$. Then $P_2\stackrel{\sim}{\to}X$ has a left inverse $h\colon X\stackrel{\sim}{\to} P_2$ by \Cref{Lemma:TorsionFreeHasWeakInflationToProjective}.\eqref{Item:TorsionFreeHasWeakInflationToProjective1}.  The composite $P_1\to X\xrightarrow{h}P_2$ yields the desired morphism $g\colon P_1\to P_2$ such that $Q\circ \iota(g)=f$.
\end{proof}

\begin{remark}\label{remark:WhichExactStructureOnProjectives}
In light of \Cref{Proposition:KillingTorsionGivesProjectives}, there are two natural exact structures on $\PP$.  Firstly, as subcategory of $\EE$, the category $\PP$ inherits the split exact structure.  Secondly, via the quotient $Q\colon \EE \to \EE/{}^{\perp}\PP$, the category $\EE/{}^{\perp}\PP$ is equipped with the structure of an exact category.  We can use the equivalence $\PP\stackrel{\iota}{\hookrightarrow} \EE\xrightarrow{Q} \EE/{}^{\perp}\PP$ of \Cref{Proposition:KillingTorsionGivesProjectives} to transfer this structure to $\PP$.

This last structure can be understood as follows.  Recall from \Cref{Theorem:QReflectsAdmissibles} that a morphism $f$ in $\EE$ is admissible if and only if $Q(f)$ is admissible.  We find that a morphism $g$ in $\PP$ is admissible if and only if $\iota(g)$ is admissible.  This is sufficient to recover the conflation structure on $\PP$.  Indeed, the inflations are the monic admissible morphisms and the deflations are the epic admissible morphisms.
\end{remark}

%

\begin{proposition}\label{Proposition:ProjectivesLeftExact}
Let $\EE\in \AEx$ and let $\PP=\Proj(\EE)$ be endowed with the exact structure coming from the equivalence in \Cref{Proposition:KillingTorsionGivesProjectives}. The following hold:
\begin{enumerate}
\item The functor $\PP\to \EE$ is left exact.
\item The induced exact functor $\smodad(\PP)\to \EE$ is an equivalence.
\end{enumerate}
\end{proposition}

\begin{proof}
\begin{enumerate}
\item Let $P_1\xrightarrow{f}P_2\xrightarrow{g}P_3$ be a sequence in $\PP$, and assume that $Q(P_1)\xrightarrow{Q(f)}Q(P_2)\xrightarrow{Q(g)}Q(P_3)$ is a conflation in $\EE/{}^{\perp}\PP$.  Since $Q(f)$ and $Q(g)$ are admissible in $\EE/{}^{\perp}\PP$, it follows that $f$ and $g$ are admissible in $\EE$ by \Cref{Theorem:QReflectsAdmissibles}. Hence, it only remains to show that $f$ is a kernel of $g$. Since $\EE$ has enough projectives, it suffices to check that any morphism $h\colon P\to P_2$ with $P\in \PP$ satisfying $g\circ h=0$ must factor uniquely through $f$. In other words, it suffices to show that the sequence
\[
\Hom_{\EE}(P,P_1)\xrightarrow{f\circ -}\Hom_{\EE}(P,P_2)\xrightarrow{g\circ -}\Hom_{\EE}(P,P_3)
\]
is left exact. By \Cref{Proposition:KillingTorsionGivesProjectives} we have isomorphisms $\Hom_{\EE}(P,P_i)\cong \Hom_{\EE/{}^{\perp}\PP}(Q(P),Q(P_i))$ for $1\leq i\leq 3$. Hence the sequence above is isomorphic to the sequence
\[
\Hom_{\EE/{}^{\perp}\PP}(Q(P),Q(P_1))\xrightarrow{Q(f)\circ -}\Hom_{\EE/{}^{\perp}\PP}(Q(P),Q(P_2))\xrightarrow{Q(g)\circ -}\Hom_{\EE/{}^{\perp}\PP}(Q(P),Q(P_3))
\]
which is left exact since $Q(P_1)\xrightarrow{Q(f)}Q(P_2)\xrightarrow{Q(g)}Q(P_3)$ is a conflation. This proves the claim.

\item By \Cref{Theorem:Universal property} there exists an exact functor $\overline{\iota}\colon \smodad(\PP)\to \EE$ extending the inclusion $\iota\colon \PP\to \EE$; the functor $\overline{\iota}$ is given by mapping an object $M \cong \coker (\Upsilon(f))$ (where $f$ is an admissible morphism in $\PP$) to $\coker(\iota(f))$. Since any object $E\in \EE$ is the cokernel of an admissible morphism in $\PP$, the functor $\overline{\iota}\colon \smodad(\PP)\to \EE$ is dense. Since  $\smodad(\PP)$ is an exact category with enough projectives, and since the exact functor $\smodad(\PP)\to \EE$ restricts to an fully faithful functor on the projective objects, it must be fully faithful itself. This shows that $\overline{\iota}$ is an equivalence. \qedhere
\end{enumerate}
\end{proof}

\begin{theorem}[First version of Auslander correspondence]\label{FirstAuslanderCorrespondence}
	 $\smodad(-)\colon \Ex_L\to \AEx$ is an equivalence of $2$-categories.
\end{theorem}

\begin{proof}
We define a $2$-functor $\AEx\to \Ex_L$ as follows: An object $\EE\in \AEx$ is sent to $\Proj(\EE)$ endowed with the exact structure coming from the equivalence in \Cref{Proposition:KillingTorsionGivesProjectives}. A $1$-morphism $\Psi\colon \EE\to \EE'$ is sent to the restriction $\Psi|_{\Proj(\EE)}\colon \Proj(\EE)\to \Proj(\EE')$ (which can be seen to be left exact by using that the inclusion $\Proj(\EE)\to \EE$ is left exact). Finally, a natural transformation $\Psi\to \Psi'$ is sent to the restriction $\Psi|_{\Proj(\EE)}\to \Psi'|_{\Proj(\EE)}$. It follows from \Cref{Theorem:AuslandersFormulaForExactCategories} and  \Cref{Proposition:ProjectivesLeftExact} that this gives a quasi-inverse to $\smodad (-)$, which proves the claim.
\end{proof}

\begin{remark}\label{AbelianCase}
Note that $\EE$ is an abelian category if and only if $\smodad(\EE)$ is an abelian category. Indeed, if $\EE$ is abelian, then $\smodad(\EE)=\operatorname{mod}(\EE)$ is abelian since $\EE$ has weak kernels \cite{Freyd66}. Conversely, if $\smodad(\EE)$ is abelian, then $\eff(\EE)$ is a Serre subcategory of $\smodad(\EE)$, and hence the localization $\EE\cong \smodad(\EE)/\eff(\EE)$ is an abelian category. It follows from this that $\smodad(-)\colon \Ex_L\to \AEx$ restricts to an equivalence of $2$-categories
\[
\operatorname{mod}(-)\colon \Abelian_L \to \AAb
\]
where $\Abelian_L$ and $\AAb$ are the $2$-subcategories of $\Ex_L$ and $\AEx$ consisting of the abelian categories. Note that axiom \ref{A2} in the definition of Auslander exact categories reduces to requiring ${}^{\perp}\PP$ to be closed under subobjects in this case. 
\end{remark}

We use the equivalence from \Cref{FirstAuslanderCorrespondence} to deduce some further properties of Auslander exact categories.

\begin{corollary}
	Let $\EE\in \AEx$ and $\PP=\Proj(\EE)$. The following hold:
	\begin{enumerate}
	\item The subcategory $\cogen(\PP)$ consists precisely of the objects of projective dimension at most one.
	\item the inclusion functor $\PP\to \EE$ has a left adjoint.
	\end{enumerate} 
\end{corollary}

\begin{proof}
This holds since the analogous properties are true for $\smodad(\EE')$ for an exact category $\EE'$ by \Cref{corollary:YonedaLeftAdjoint} and \Cref{Proposition:TorsionTheoryProperties}.\eqref{Proposition:TorsionTheoryPropertiesB}.
\end{proof}

\subsection{Injectives and dominant dimension}

In this section we investigate how the property of an exact category $\EE$ having enough injectives is reflected in $\smodad (\EE)$. In particular, we obtain a more familiar version of Auslander correspondence.

\begin{lemma}\label{injectives}
Let $\EE$ be an exact category. An object $I$ is injective in $\EE$ if and only if $\Upsilon(I)$ is injective in $\smodad \EE$.
\end{lemma}

\begin{proof}
Assume $I$ is injective in $\EE$. Let $F\in \smodad (\EE)$, and choose an admissible morphism $f\colon X\to Y$ such that
\[
0\to \Upsilon(\ker f)\to \Upsilon (X)\xrightarrow{\Upsilon(f)}\Upsilon (Y)\to F\to 0
\]
is exact. Applying $\Hom_{\smodad(\EE)}(-,\Upsilon (I))$ and using that the Yoneda embedding is fully faithful, we get the sequence
\[
0\to \Hom_{\smodad(\EE)}(F,\Upsilon (I))\to \Hom_{\EE}(Y,I)\to \Hom_{\EE}(X,I)\to \Hom_{\EE}(\ker f,I)\to 0.
\]
It is exact since $I$ is injective in $\EE$. Hence $\Ext^i_{\smodad (\EE)}(F,\Upsilon (I))=0$ for all $i>0$ and all $F\in \smodad (\EE)$, so $\Upsilon (I)$ is injective in $\smodad (\EE)$.

Conversely, assume $\Upsilon (I)$ is injective in $\smodad (\EE)$, and let $0\to X\xrightarrow{f}Y\xrightarrow{g}Z\to 0$ be a conflation in $\EE$. Then the sequence
\[
0\to \Upsilon (X)\xrightarrow{\Upsilon(f)}\Upsilon (Y)\xrightarrow{\Upsilon(g)}\Upsilon (Z)\to \coker (\Upsilon(g))\to 0
\]
is exact in $\smodad (\EE)$. Since $\coker (\Upsilon(g))\in \eff (\EE)$, we get that $\Hom_{\smodad (\EE)}(\coker (\Upsilon(g)),\Upsilon (I))=0$. Therefore, applying $\Hom_{\smodad (\EE)}(-,\Upsilon (I))$ to the sequence above and using that 
\[
\Ext^i_{\smodad(\EE)}(\coker (\Upsilon(g)),\Upsilon (I))=0
\]
 for all $i>0$ since $\Upsilon (I)$ is injective, we get an exact sequence
\[
0\to \Hom_{\EE}(Z,I)\xrightarrow{-\circ g}\Hom_{\EE}(Y,I)\xrightarrow{-\circ f}\Hom_{\EE}(X,I)\to 0.
\]
This shows that $I$ is injective in $\EE$.
\end{proof}

\begin{definition}\label{definition:DominantDimension}
	Let $\EE$ be an exact category with enough projectives. The \emph{dominant dimension} of $\EE$, denoted $\domdim (\EE)$, is the largest integer $n$ such that for any projective object $P\in \EE$ there exists an exact sequence
\[
0\to P\to I_1\to \cdots \to I_n \to C\to 0
\]
with $I_k$ being projective and injective for $1\leq k\leq n$. 
\end{definition}

\begin{lemma}\label{Lemma:EquivalentCharacterizationOfDomDim}
Let $\EE$ be an exact category. The following are equivalent:
\begin{enumerate}
\item\label{Item:EquivalentCharacterizationOfDomDim1} $\EE$ has enough injectives.
\item\label{Item:EquivalentCharacterizationOfDomDim2} $\domdim (\smodad (\EE))\geq 2$.
\item\label{Item:EquivalentCharacterizationOfDomDim3} $\domdim (\smodad (\EE))\geq 1$.
\end{enumerate}
\end{lemma}

\begin{proof}
Assume $\EE$ has enough injectives. By \Cref{Proposition:TorsionTheoryProperties}.\eqref{Proposition:TorsionTheoryPropertiesD}, a projective object in $\smodad(\EE)$ is of the form $\Upsilon (X)$. As $\EE$ has enough injectives, there is an exact sequence
\[
0\to X\to I_1\xrightarrow{f} I_2 \to \coker(f)\to 0
\]
with $I_1$ and $I_2$ injective. Applying the Yoneda embedding we get a left exact sequence 
\[
0\to \Upsilon (X)\to \Upsilon (I_1)\to \Upsilon (I_2)
\]
in $\smodad (\EE)$, and by \Cref{injectives} we know that $\Upsilon (I_1)$ and $\Upsilon (I_2)$ are projective and injective in $\smodad(\EE)$. This shows that $\domdim (\smodad(\EE))\geq 2$. Since the implication $\eqref{Item:EquivalentCharacterizationOfDomDim2}\Rightarrow \eqref{Item:EquivalentCharacterizationOfDomDim3}$ is clear, it only remains to show $\eqref{Item:EquivalentCharacterizationOfDomDim3}\Rightarrow \eqref{Item:EquivalentCharacterizationOfDomDim1}$. Let $X$ be an arbitrary object in $\EE$. By assumption there exists an inflation $\Upsilon (X)\to \Upsilon (I)$ where $\Upsilon (I)$ is injective in $\smodad (\EE)$, and thus $I$ is injective in $\EE$ by \Cref{injectives}. Since the localization functor $Q\colon \smodad (\EE)\to \smodad(\EE)/\eff(\EE)\simeq\EE$ is exact, it follows that the induced morphism $X\to I$ is an inflation in $\EE$. This proves the claim.
\end{proof}

Our next goal is to show that $\domdim (\EE)\geq 2$ implies some of the other criteria for categories in $\AEx$.

\begin{lemma}\label{Lemma:DomdimImpliesRigid}
Let $\EE$ be an exact category with enough projectives $\PP=\Proj(\EE)$. Assume $\domdim (\EE)\geq 2$. Then $\Ext^1_{\EE}({}^{\perp}\PP,\PP)=0$.
\end{lemma}

\begin{proof}
Let $X\in {}^{\perp}\PP$ and $P\in \PP$ be arbitrary. Choose a conflation $0\to P\to I_0\to C\to 0$ and $0\to C\to I_1\to C'\to 0$ with $I_0$ and $I_1$ projective and injective in $\EE$. Applying $\Hom_{\EE}(X,-)$ to the second sequence, we get a monomorphism $\Hom_{\EE}(X,C)\to \Hom_{\EE}(X,I_1)$ and since $\Hom_{\EE}(X,I_1)=0$ by assumption on $X$, it follows that  $\Hom_{\EE}(X,C)=0$. Applying $\Hom_{\EE}(X,-)$ to the first sequence we get an exact sequence $\Hom_{\EE}(X,C)\to \Ext^1_{\EE}(X,P)\to \Ext^1_{\EE}(X,I)$. Now $\Hom_{\EE}(X,C)=0$ by the discussion above and $\Ext_{\EE}^1(X,I)=0$ since $I$ is injective. Hence $\Ext^1_{\EE}(X,P)=0$, which proves the claim.
\end{proof}

\begin{lemma}\label{Lemma:DomdimImpliesCloseUnderSubmodules}
Let $\EE$ be an exact category with enough projectives $\PP=\Proj(\EE)$. Assume $\domdim (\EE)\geq 1$. Then ${}^{\perp}\PP$ is closed under admissible subobjects.
\end{lemma}

\begin{proof}
Let $E\inflation X$ be an inflation with $X\in{}^{\perp}\PP$. We want to show that $E\in {}^{\perp}\PP$. Since $\domdim (\EE)\geq 1$, it suffices to show that $\Hom_{\EE}(E,P)=0$ for all $P$ which are projective and injective. But since $P$ is injective, the induced map $\Hom_{\EE}(X,P)\to \Hom_{\EE}(E,P)$ is an epimorphism. Since $\Hom_{\EE}(X,P)=0$ by assumption, it follows that  $\Hom_{\EE}(E,P)=0$.
\end{proof}

\begin{definition}
Let $\YY$ be a subcategory of an exact category $\EE$. A morphism $f\colon X\to Y$ in $\EE$ with $Y\in \YY$ is called an \emph{(admissible) left-}$\YY$ \emph{approximation} if $f$ is an admissible morphism and any morphism $X\to Y'$ with $Y'\in \YY$ factors through $f$. The subcategory $\YY$ is called \emph{(admissibly) covariantly finite} if for all objects $X$ in $\EE$ there exists an admissible left $\YY$-approximation $X\to Y$.
\end{definition}

\begin{lemma}\label{Lemma:CovariantlyFiniteTorsion}
Let $\EE$ be an exact category with enough projectives $\PP=\Proj(\EE)$. The following hold:
\begin{enumerate}
\item\label{Item:AuslanderGivesCovariantlyFinite} If $\EE$ is an Auslander exact category, then $\PP$ is admissibly covariantly finite in $\EE$.
\item\label{Item:CovariantlyFiniteGivesTorsion} If $\PP$ is admissibly covariantly finite in $\EE$ and $\domdim (\EE)\geq 1$, then $({}^{\perp}\PP,\cogen (\PP))$ is a torsion pair.
\end{enumerate} 
\end{lemma}  

\begin{proof}
\begin{enumerate}
\item For $E\in \EE$, choose a conflation $\tt E\inflation E\deflation \ff E$ with $\tt E\in {}^{\perp}\PP$ and $\ff E\in \cogen(\PP)$. By \Cref{Lemma:TorsionFreeHasWeakInflationToProjective} . \eqref{Item:TorsionFreeHasWeakInflationToProjective3} we can choose a conflation $\ff E\inflation P\deflation T$ with $P\in \PP$ and $T\in {}^{\perp}\PP$. Since $\Hom_{\EE}(\tt E,\PP)=0$, any morphism $E\to Q$ with $Q\in \PP$ must factor through $\ff  E$. Also, since $\Ext^1_{\EE}(T,\PP)=0$, any morphism $\ff E\to Q$ with $Q\in \PP$ must factor through $P$. This shows that $E\deflation \ff E\inflation P$ is the required left $\PP$-approximation.
\item Let $X$ be an arbitrary object in $\EE$, and choose an admissible left $\PP$-approximation $f\colon X\to P$. Consider the conflation $\ker (f)\inflation X\deflation \im (f)$. Let $I$ be a projective and injective object in $\EE$. Applying $\Hom_{\EE}(-,I)$ to the conflation above, we get a short exact sequence of abelian groups
\[
\Hom_{\EE}(\im(f),I)\inflation \Hom_{\EE}(X,I)\deflation \Hom_{\EE}(\ker(f),I).
\]
Since $f$ is a left $\PP$-approximation and $I$ is projective, it follows that $\Hom_{\EE}(\im(f),I)\inflation \Hom_{\EE}(X,I)$ is surjective, so $\Hom_{\EE}(\ker(f),I)=0$. Since $\domdim (\EE)\geq 1$, any projective object is a subobject of a projective and injective object. Hence $\ker (f)\in {}^{\perp}\PP$. Since $\im(f)\in \cogen (\PP)$, this proves the claim.\qedhere
\end{enumerate}
\end{proof}

We finish this section by showing that if $\EE$ has enough injectives, then $\smodad(\EE)$ has enough injectives. In the following we let $\II$ denote the subcategory of injective objects in $\EE$.

\begin{remark}\label{Remark:YouCanTruncateThoseSequences}
 For all $n\geq 1$, $\cogen_n(\II)\subseteq \cogen_{n-1}(\II)\subseteq \cdots \subseteq \cogen_{1}(\II)$. If $\EE$ has enough injectives, then $\EE=\cogen_{1}(\II)=\cogen_{2}(\II)=\cdots$.
\end{remark}

\begin{lemma}\label{injectivesyzygies}
	Let $\EE$ be an exact category, let $n\geq 1$ be an integer, and let $X\inflation Y\deflation Z$ be a conflation in $\EE$. The following hold:
		\begin{enumerate}
			\item If $X\in \cogen_{n}(\II)$ and $Z\in \cogen_{n}(\II)$, then $Y\in \cogen_{n}(\II)$.
			\item If $X\in \cogen_{n}(\II)$ and $Y\in \cogen_{n-1}(\II)$, then $Z\in \cogen_{n-1}(\II)$.
		\end{enumerate} 
\end{lemma}

\begin{proof}
	\begin{enumerate}
		\item Choose exact sequences $0\to X \to I^X_1\to \cdots \to I^X_{n}\to C^X\to 0$ and $0\to Z \to I^Z_1\to \cdots \to I^Z_{n}\to C^Z\to 0$ with each $I^X_i,I^Z_j$ injective.  By the Horseshoe Lemma, there is an exact sequence 
		\[0\to Y\to I^X_1\oplus I^Z_{1} \to \cdots \to I^X_{n}\oplus I^Z_n\to C^X\oplus C^Z\to 0.\] This shows that $Y\in \cogen_{n}(\II)$.

		\item As $X\in \cogen_{n}(\II)$ and $n\geq 1$, there is a conflation $X\stackrel{i}{\inflation}I\deflation C$ with $I\in \Inj(\EE)$ and $C\in \cogen_{n-1}(\II)$.  Taking the pushout of $i$ along $X\inflation Y$ we obtain the commutative diagram
		\[\xymatrix{
			X\ar@{>->}[r]\ar@{>->}[d]^i & Y\ar@{->>}[r]\ar@{>->}[d] & Z\ar@{=}[d]\\
			I\ar@{>->}[r]\ar@{->>}[d] & U\ar@{->>}[r]\ar@{->>}[d] & Z\\
			C \ar@{=}[r]& C &
		}\] where the rows and columns are conflations. As $I$ is injective, the middle row is a split conflation and thus $U\cong I\oplus Z$. Applying the first part of this lemma to the conflation $Y\inflation I\oplus Z\deflation C$ yields that $I\oplus Z\in \cogen_{n-1}(\II)$. Hence there is an exact sequence 
		\[0\to I\oplus Z\to I_1\to \cdots \to I_{n-1}\to D\to 0\] with each $I_j\in \Inj(\EE)$. The composition $\xymatrix{I\ar@{>->}[r]^-{\begin{psmallmatrix}1 \\0\end{psmallmatrix}} & I\oplus Z\ar@{>->}[r] & I_1}$ is a split inflation as $I$ is injective. It follows that $I_1\cong I\oplus I_1'$ and $I_1'\in \Inj(\EE)$ as injectives are closed under direct summands. The exact sequence above becomes
\[0\to I\oplus Z\xrightarrow{\begin{psmallmatrix}1&0\\0&f\end{psmallmatrix}}I\oplus I_1'\to \cdots \to I_{n-1}\to D\to 0\] for some map $f\colon Z\to I_1'$. Taking the pushout of the canonical projection $I\oplus Z\deflation Z$ along $I\oplus Z\stackrel{\begin{psmallmatrix}1&0\\0&f\end{psmallmatrix}}{\inflation}I\oplus I_1'$, one sees that $f\colon Z \inflation I_1'$ is an inflation. Hence the sequence
\[0\to  Z\xrightarrow{f} I_1'\to I_2\to \cdots \to I_{n-1}\to D\to 0\]
is exact, which shows that $Z\in \cogen_{n-1}(\II)$.\qedhere
	\end{enumerate}
\end{proof}

\begin{proposition}\label{Proposition:EnoughInjectives}
Let $\EE$ be an exact category with enough projectives and assume 
\[\domdim (\EE)\geq 2\geq \gldim (\EE).\]
Then $\EE$ has enough injectives.
\end{proposition}

\begin{proof}
	For any projective $P\in \EE$, there exists an exact sequence
		\[0\to P\to I_1\to I_2\to C\to 0\]
	where $I_1$ and $I_2$ are projective and injective. Since $\gldim (\EE)\leq 2$ and $C$ is a cokernel of an admissible morphism between injective objects, it 	follows that $C$ is injective. Hence $P\in \cogen_{3}(\II)$. 

	Now let $X\in \EE$ be an arbitrary object, and choose an exact sequence $0\to P_2\xrightarrow{f}P_1\to P_0\to X\to 0$ with $P_0,P_1,P_2$ projective in $\EE$. Applying \Cref{injectivesyzygies} to the conflation $P_2\xrightarrow{f}P_1\to \coker (f)$ and using that $P_2,P_1\in \cogen_{3}(\II)$, it follows that $\coker (f)\in \cogen_{2}(\II)$. Similarly, applying \Cref{injectivesyzygies} to the conflation $\coker (f)\to P_0\to X$ we get that $X\in \cogen_{1}(\II)$. This proves the claim.
\end{proof}

\begin{definition}
Let $\IEx_L$ be the $2$-subcategory of $\Ex_L$ consisting of the exact categories with enough injectives.
\end{definition}

Combining the results of this section, we obtain the following theorem. 

\begin{theorem}[Second version of Auslander correspondence]\label{Theorem:AuslanderSecond} Let $\EE$ be an exact category with enough projectives $\PP=\Proj(\EE)$ and satisfying the following conditions:
\begin{enumerate}[label=(\alph*)]
		\item\label{Item:Dimension} $\domdim (\EE)\geq 2\geq \gldim (\EE)$.
		\item\label{Item:A2} Any morphism $X\to E$ with $E\in{}^{\perp}\PP$ is admissible.
		\item\label{Item:ContravariantlyFinite} $\PP$ is admissibly covariantly finite.
	\end{enumerate} The following hold:
\begin{enumerate}
\item $\EE$ is an Auslander exact category with enough injectives.
\item The equivalence of $2$-categories $\smodad(-)\colon \Ex_L\to \AEx$  restricts to an equivalence from $\IEx_L$ onto the subcategory of $\AEx$ consisting of the exact categories satisfying \ref{Item:Dimension}, \ref{Item:A2} and \ref{Item:ContravariantlyFinite} above.  
\end{enumerate}
\end{theorem}

\begin{proof}
\begin{enumerate}
\item  The category $\EE$ has enough injectives by \Cref{Proposition:EnoughInjectives}, and is an Auslander exact category by \Cref{Lemma:DomdimImpliesRigid}, \Cref{Lemma:DomdimImpliesCloseUnderSubmodules} and \Cref{Lemma:CovariantlyFiniteTorsion}. \eqref{Item:CovariantlyFiniteGivesTorsion}.

\item We only need to verify that $\EE$ has enough injectives if and only if $\smodad(\EE)$ satisfies \ref{Item:Dimension}, \ref{Item:A2} and \ref{Item:ContravariantlyFinite} above. The "only if" direction follows from \Cref{Lemma:EquivalentCharacterizationOfDomDim}, and the "if" direction follows from \Cref{Lemma:EquivalentCharacterizationOfDomDim} and \Cref{Lemma:CovariantlyFiniteTorsion}. \eqref{Item:AuslanderGivesCovariantlyFinite}. \qedhere
\end{enumerate}
\end{proof}

\begin{remark}
Note that the abelian categories satisfying \ref{Item:Dimension} in \Cref{Theorem:AuslanderSecond} are precisely the free abelian categories, see \cite[Theorem 6.6]{Beligiannis00}.
\end{remark}

\begin{remark}
We explain how to recover the classical Auslander correspondence. Assume $\Lambda$ is a representation-finite artin algebra over a commutative artin ring $k$, and choose an additive generator $M$ of $\operatorname{mod}\Lambda$.  We then have an equivalence of categories
\[
\operatorname{mod}\Gamma\cong \operatorname{mod}(\operatorname{mod}(\Lambda))
\]
where $\Gamma=\End_{\Lambda}(M)$. Since $\operatorname{mod}(\Lambda)$ is an abelian category with enough injectives, we get that \[
\domdim \Gamma \geq 2\geq \gldim \Gamma
\]
 by \Cref{Theorem:AuslanderSecond}. This gives one direction of the classical Auslander correspondence.

Conversely, assume $\Gamma$ is an artin $k$-algebra satisfying $\domdim \Gamma \geq 2\geq \gldim \Gamma$. Since $\Gamma$ is an artin algebra, the subcategory $\operatorname{proj}\Gamma$ of projective objects in $\operatorname{mod}(\Gamma)$ is a covariantly finite subcategory. Hence, $\operatorname{mod}(\Gamma)$ satisfies \ref{Item:Dimension}, \ref{Item:A2} and \ref{Item:ContravariantlyFinite} in \Cref{Theorem:AuslanderSecond}, and combining this with \Cref{AbelianCase} we therefore get that $\operatorname{proj}\Gamma$ is an abelian category with enough injectives. Hence, $(\operatorname{proj}\Gamma)^{\operatorname{op}}\cong \operatorname{mod}(\Lambda')$ for some finite-dimensional algebra $\Lambda'$. Composing with the duality 
\[
D=\Hom_k(-,I)\colon \operatorname{mod}(\Lambda'^{\operatorname{op}})\xrightarrow{\cong} (\operatorname{mod}(\Lambda'))^{\operatorname{op}}
\]
where $I=\bigoplus_{i=1}^nI(S_i)$ is the sum of the injective envelopes $I(S_i)$ of the simple $k$-modules $S_i$, we get that $\operatorname{proj}\Gamma\cong \operatorname{mod}(\Lambda'^{\operatorname{op}})$. This recovers the other direction of the classical Auslander correspondence. 
\end{remark}

\subsection{Exact categories with enough projectives}\label{subsection:EnoughProjectives}

In this section we investigate how the property of an exact category $\EE$ having enough projectives is reflected in $\smodad (\EE)$. In particular, we show that if $\EE$ is an exact category with enough projectives, the category $\eff(\EE)$ induces a torsion-torsion-free triple ($\mathsf{TTF}$-triple) on $\smodad(\EE)$. Furthermore, this $\mathsf{TTF}$-triple induces a recollement of exact categories (in the sense of \cite{WangWeiZhang20}) analogous to \cite[Proposition~2.17]{Ogawa19}.

\begin{lemma}\label{Lemma:ProjectiveObjectsInEETranslateToPerpendicularsOfTorsion}
	An object $P$ of an exact category $\EE$ is projective if and only if $\Upsilon(P)\in {}^{\perp}\eff(\EE)$.
\end{lemma}

\begin{proof}
This is a reformulation of \cite[Proposition~2.11]{Enomoto18}.  Note that the proof does not require $\EE$ to be idempotent complete.
\end{proof}

\begin{proposition}\label{Proposition:EnoughProjectivesInEEMakesTTTorsionFree}
	Let $\EE$ be an exact category and write $\QQ$ for $\Proj(\smodad(\EE))$. The following are equivalent:
	\begin{enumerate}
		\item\label{Item:EnoughProjectivesInEEMakesTTTorsionFree1} The category $\EE$ has enough projectives.
		\item\label{Item:EnoughProjectivesInEEMakesTTTorsionFree2}  $(\gen(\QQ\cap {}^{\perp}\eff(\EE)), \eff(\EE))$ is a torsion pair in $\smodad(\EE)$.
		\item\label{Item:EnoughProjectivesInEEMakesTTTorsionFree3} For all $P\in \QQ$, there exists an exact sequence $$0\to F'\to P\to F''\to 0$$ with $F'\in \gen(\QQ\cap {}^{\perp}\eff(\EE))$ and $F''\in \eff(\EE)$.
	\end{enumerate}
\end{proposition}

\begin{proof}
	Assume \eqref{Item:EnoughProjectivesInEEMakesTTTorsionFree1}. We first show that $\Hom(\gen(\QQ\cap {}^{\perp}\eff(\EE)), \eff(\EE))=0$. Let $g\colon E\to F$ be a map with $E\in \gen(\QQ\cap {}^{\perp}\eff(\EE))$ and $F\in \eff(\EE)$. By definition, there is an object $E'\in \QQ\cap {}^{\perp}\eff(\EE)$ and a deflation $E'\deflation E$.	By \Cref{Proposition:TorsionTheoryProperties}. \eqref{Proposition:TorsionTheoryPropertiesD} the functor $\Upsilon\colon \EE\to \QQ$ is an equivalence and thus there is an object $P\in \EE$ such that $\Upsilon(P)\cong E'$. By \Cref{Lemma:ProjectiveObjectsInEETranslateToPerpendicularsOfTorsion}, $P\in \Proj(\EE)$. Let $f\colon X\deflation Y$ be a deflation in $\EE$ such that $\coker(\Upsilon(f))\cong F$. Since $\Upsilon(P)$ is projective and $\Upsilon$ is fully faithful, there exists a map $h\colon P\to Y$ such that the composite $\Upsilon(P)\deflation E\xrightarrow{g}F$ is equal to $\Upsilon(P)\xrightarrow{\Upsilon(h)}\Upsilon(Y)\deflation F$. Now since $f$ is a deflation and $P$ is projective, $h$ must factor through $f$, and hence the composite above must be $0$. As $\Upsilon(P)\deflation E$ is surjective, $g$ must be zero. We conclude that $\Hom(\gen(\QQ\cap {}^{\perp}\eff(\EE)), \eff(\EE))=0$.
	
	Let $F\in \smodad(\EE)$ and let $F\cong\coker(\Upsilon(f))$ for some admissible morphism $f\colon X\to Y$ in $\EE$. As $\EE$ has enough projectives, there is a deflation $P\stackrel{p}{\deflation} Y$ with $P\in \Proj(\EE)$. The pullback $Z$ of $f$ along $p$ yields an admissible morphism $g\colon Z\to P$. Note that this pullback square induces a conflation $Z\inflation X\oplus P\stackrel{\begin{psmallmatrix}f&p\end{psmallmatrix}}{\deflation} Y$ by (the dual of) \cite[Proposition~2.12]{Buhler10}. Let $F'=\coker(\Upsilon(g))$ and let $F''=\coker(\Upsilon\begin{psmallmatrix}f&p\end{psmallmatrix})\in \eff(\EE)$. Considering the two rightmost squares in diagram \eqref{Thefamousdiagram} in \Cref{Lemma:TheFamousDiagramChase}, we see that we have a conflation $F'\inflation F\deflation F''$. It remains to show that $F'\in \gen(\QQ\cap {}^{\perp}\eff(\EE))$. By construction there is a deflation $\Upsilon(P)\deflation F'$ and as $P\in \Proj(\EE)$, \Cref{Lemma:ProjectiveObjectsInEETranslateToPerpendicularsOfTorsion} yields that $\Upsilon(P)\in {}^{\perp}\eff(\EE)$, thus $F'\in \gen(\QQ\cap {}^{\perp}\eff(\EE))$. This shows the implication $\eqref{Item:EnoughProjectivesInEEMakesTTTorsionFree1}\Rightarrow \eqref{Item:EnoughProjectivesInEEMakesTTTorsionFree2}$.
	
	The implication $\eqref{Item:EnoughProjectivesInEEMakesTTTorsionFree2}\Rightarrow \eqref{Item:EnoughProjectivesInEEMakesTTTorsionFree3}$ is trivial. Now assume \eqref{Item:EnoughProjectivesInEEMakesTTTorsionFree3}. Let $A\in \EE$. By assumption there is a short exact sequence 
	\[F'\inflation \Upsilon(A)\deflation F''\]
	with $F'\in \gen(\QQ\cap {}^{\perp}\eff(\EE))$ and $F''\in \eff(\EE)$. By definition there is a deflation $\Upsilon(P)\deflation F'$ with $P\in \Proj(\EE)$. Letting $K$ be the kernel of the map $\Upsilon(P)\deflation F'\inflation \Upsilon(A)$, we get an exact sequence 
	\[
	0\to K\to \Upsilon(P)\to \Upsilon(A)\to F''\to 0
	\]
	in $\smodad(\EE)$. Applying the functor $L\colon \smodad(\EE)\to \EE$ in \Cref{corollary:YonedaLeftAdjoint} and using that $L(F'')=0$ since $F''\in \eff(\EE)$, we get the conflation $L(K)\inflation P\deflation A$ in $\EE$. Hence $\EE$ has enough projectives.
\end{proof}

We recall the notion of a torsion-torsion-free triple (see \cite{Jans65}). 

\begin{definition}
	Let $\EE$ be an exact category. A triple $(\mathcal{X},\mathcal{Y},\mathcal{Z})$ of full subcategories is called a \emph{torsion-torsion-free triple} (or $\mathsf{TTF}$-triple) if both $(\mathcal{X},\mathcal{Y})$ and $(\mathcal{Y},\mathcal{Z})$ are torsion pairs. 
\end{definition}

\begin{corollary}\label{Corollary:TTFTriple}
	Let $\EE$ be an exact category and write $\QQ$ for $\Proj(\smodad(\EE))$. If $\EE$ has enough projectives, then the triple $(\gen(\QQ\cap {}^{\perp}\eff(\EE)),\eff(\EE),\cogen(\QQ))$ is a $\mathsf{TTF}$-triple. 
\end{corollary}

\begin{proof}
	This follows directly from \Cref{Proposition:EnoughProjectivesInEEMakesTTTorsionFree,Proposition:EffAreTorsion} and \Cref{Proposition:TorsionTheoryProperties}.\eqref{Proposition:TorsionTheoryPropertiesB}.
\end{proof}

In \cite{PsaroudakisVitoria14}, a bijection between certain $\mathsf{TTF}$-triples and recollements of abelian categories is established.  In \cite[Definition~3.1]{WangWeiZhang20} a general definition of a recollement of extriangulated categories is given. \Cref{Definition:RecollementOfExactCategory} below is the specialization of this general definition to the case of exact categories. In \Cref{Proposition:TTFTripleYieldsRecollement} below, we show that the $\mathsf{TTF}$-triple of \Cref{Corollary:TTFTriple} induces a recollement of exact categories.

\begin{definition}\label{Definition:RecollementOfExactCategory}
	A \emph{recollement of exact categories} is a triple of exact categories $(\AA,\BB,\CC)$ fitting together in a diagram
	\[\xymatrix{
		\AA\ar[rr]|{i_*} && \BB \ar@/^1pc/[ll]^{i^!}\ar@/_1pc/[ll]_{i^*} \ar[rr]|{j^*} && \CC\ar@/^1pc/[ll]^{j_*}\ar@/_1pc/[ll]_{j_!}
	}\] with $6$ additive functors satisfying the following conditions:
	\begin{enumerate}
		\item The triples $(i^*,i_*,i^!)$ and $(j_!,j^*,j_*)$ are \emph{adjoint triples}, i.e.~$i^*\dashv i_*\dashv i^!$ and $ j_!\dashv j^*\dashv j_*$.
		\item $\im(i_*)=\ker(j^*)$.
		\item $i_*,j_!$ and $j_*$ are fully faithful.
		\item For each $X\in \BB$, there exists an exact sequence
		\[\xymatrix{
			i_*i^!(X) \ar@{>->}[r]^-{\theta_X} & X\ar[r]|-{\circ}^-{\vartheta_X} & j_*j^*(X)\ar@{->>}[r] & i_* A.
		}\] Here, $\theta$ is the counit of the adjunction $i_*\dashv i^!$, $\vartheta$ is the unit of the adjunction $j_*\dashv j^*$ and $A\in \AA$.
		\item For each $X\in \BB$, there exists an exact sequence
		\[\xymatrix{
			i_*A'\ar@{>->}[r] & j_!j^*(X)\ar[r]|-{\circ}^-{v_X} & X\ar@{->>}[r]^-{\nu_X} & i_*i^*(X).
		}\] Here $v$ is the unit of the adjunction $j_!\dashv j^*$, $\nu$ is the counit of the adjunction $i^*\dashv i_*$ and $A'\in \AA$.
	\end{enumerate}
\end{definition}

\begin{proposition}\label{Proposition:TTFTripleYieldsRecollement}
	Let $\EE$ be an exact category. If $\EE$ has enough projectives, we obtain a recollement
	\[\xymatrix{
		\eff(\EE)\ar[rr]|(.49){i_*} && \smodad(\EE) \ar@/^1pc/[ll]^-{i^!}\ar@/_1pc/[ll]_-{i^*} \ar[rr]|(.57){L} && \EE\ar@/^1pc/[ll]^-{\Upsilon}\ar@/_1pc/[ll]_-{j_!}
	}\] of exact categories.
\end{proposition}

\begin{proof}
	We first describe all $6$ functors in the above diagram. The localization sequence $\eff(\EE)\xrightarrow{i_*}\smodad(\EE)\xrightarrow{L}\EE$ is obtained in \Cref{Theorem:AuslandersFormulaForExactCategories}. As $(\eff(\EE),\cogen(\QQ))$ is a torsion pair in $\smodad(\EE)$, the inclusion functor $i_*\colon \eff(\EE)\to \smodad(\EE)$ has a right adjoint $i^!\colon \smodad(\EE)\to \eff(\EE)$ by \Cref{Proposition:BasicPropertiesTorsionPair}. Similarly, as $(\gen(\QQ\cap {}^{\perp}\eff(\EE)),\eff(\EE))$ is a torsion pair in $\smodad(\EE)$, the inclusion $i_*$ has a left adjoint $i^*\colon \smodad(\EE)\to \eff(\EE)$ by \Cref{Proposition:BasicPropertiesTorsionPair}. This yields the adjoint triple $(i^*,i_*,i^!)$. The adjunction $L\dashv \Upsilon$ is given by \Cref{corollary:YonedaLeftAdjoint}. As $\PP$ is a subcategory of $\EE$, we obtain the restriction functor $\text{Res}\colon \Mod(\EE)\to \Mod(\PP)$. The restriction functor has a left adjoint $-\otimes_\PP\EE\colon \Mod(\PP)\to \Mod(\EE)$ which is the unique right exact functor preserving representables. Consider now the diagram:
	\[\xymatrix{
		\smodad(\EE)\ar[r]^L\ar[d] & \EE\ar[d]\\
		\Mod(\EE)\ar[r]_{\text{Res}} & \Mod(\PP) \ar@/_/[l]_{-\otimes_\PP\EE}
	}\] The functor $\EE\to \Mod(\PP)$ is determined by mapping an object $X\in \EE$ to $\Hom(-,X)_{\mid \PP}$. We now show that $-\otimes_\PP\EE$ restricts to a well-defined functor $j_!\colon\EE\to \smodad(\EE)$. Note that $-\otimes_\PP\EE$ maps $\Upsilon(P)$ to itself for any projective $P$. Now, let $C\in \EE$ and let $Q\xrightarrow{f} P\deflation C$ be a projective resolution of $C$ (where $f$ is admissible). As $-\otimes_\EE\PP$ is a right exact functor, we find that $j_!(C)\cong \coker(\Upsilon(f))\in \smodad(\EE)$. This shows that $j_!$ is well-defined and thus $j_!\dashv L$ as $-\otimes_\PP\EE\dashv \text{Res}$. Hence we obtained the adjoint triple $(j_!,L,\Upsilon)$.
	
	By \Cref{corollary:YonedaLeftAdjoint}, $\ker(L)=\eff(\EE)$ and thus $\im(i_*)= \ker(L)$. Clearly $i_*$ is fully faithful. As $L$ is a localization functor (see \Cref{Remark:WeSeeLAsALocalizationFunctor}), \cite[Proposition~1.3]{GabrielZisman67} yields that $j_!$ and $\Upsilon$ are fully faithful.
	
	Now choose $F\in \smodad(\EE)$ and let $f\colon X\to Y$ be an admissible morphism such that $F\cong \coker(\Upsilon(f))$. Consider the commutative diagram
	\[\xymatrix{
		X\ar@{=}[r]\ar@{->>}[d]^{f'} & X\ar@{->>}[r]^{f'}\ar[d]|{\circ}^{f} & \im(f)\ar[r]\ar@{>->}[d]^{f''} & 0\ar[r]\ar[d] & Y\ar@{->>}[d]^{c}\\
		\im(f)\ar@{>->}[r]^-{f''} & Y\ar@{=}[r] & Y\ar@{->>}[r]^-{c} & \coker(f)\ar@{=}[r] & \coker(f)
	}\] in $\EE$ where $f''\circ f'=f$. By \Cref{Lemma:TheFamousDiagramChase} and the proof of \Cref{corollary:YonedaLeftAdjoint}, the above diagram induces the exact sequence $i_*i^!(F)\inflation F \to \Upsilon L(F)\deflation i_*(G)$ where $G=\coker(\Upsilon(c))\in \eff(\EE)$.
	
	To obtain the last exact sequence, choose a deflation $P\stackrel{p}{\deflation} Y$ with $P\in \PP$. Taking the pullback of $f$ along $p$ yields a deflation $Z\stackrel{p'}{\deflation} X$ and an admissible morphism $Z\xrightarrow{g} Y$. Again, choose a deflation $Q\stackrel{q}{\deflation} Z$ with $Q\in \PP$. Note that $g\circ q$ is admissible and yields a projective resolution $Q\xrightarrow{g\circ q}P \deflation \coker(f)$. Consider the commutative diagram 
	\[\xymatrix{
		Q\oplus \ker(g)\ar@{->>}[r]^-{\begin{psmallmatrix}1&0\end{psmallmatrix}}\ar[d]^{\begin{psmallmatrix}q& i\end{psmallmatrix}} & Q\ar@{->>}[r]^{q}\ar[d]^{gq} & Z\ar@{->>}[r]^{p'}\ar[d]^{g} & X\ar[r]^-{\begin{psmallmatrix}1\\0\end{psmallmatrix}}\ar[d]^{f} & X\oplus P\ar[d]^{\begin{psmallmatrix}f&p\end{psmallmatrix}}\\
		Z\ar[r]^{g} & P\ar@{=}[r] & P\ar@{->>}[r]^{p} & Y\ar@{=}[r] & Y
	}\] in $\EE$. By \Cref{Lemma:TheFamousDiagramChase} and the proof of \Cref{corollary:YonedaLeftAdjoint}, we obtain the exact sequence $i_*(H)\inflation j_!L(F) \to F \deflation i_*i^*(F)$ in $\EE$ where $H=\coker(\Upsilon(\begin{psmallmatrix}q& i\end{psmallmatrix}))$ where $i\colon \ker(g)\inflation Z$ is the kernel of $g$. By (the dual of) \cite[Proposition~2.12]{Buhler10}, the map $\begin{psmallmatrix}q& i\end{psmallmatrix}$ is a deflation and thus $H\in \eff(\EE)$. This completes the proof.	
\end{proof}

\subsection{Gorenstein projectives and Cohen-Macaulay modules}

If $\EE'$ is an extension closed subcategory of an exact category $\EE$, then $\EE'$ inherits an exact structure from $\EE$. In this case, since the inclusion $\EE'\to \EE$ is fully faithful and exact, applying the $2$-functor $\smodad (-)$ we get a fully faithful exact functor
\[
\smodad(\EE')\to \smodad(\EE).
\]
Hence $\smodad(\EE')$ is equivalent to a full subcategory of $\smodad(\EE)$. Our goal in this section is to characterize this subcategory in certain examples. In the following $Q\colon \smodad(\EE)\to \EE$ denotes the localization functor, $\Upsilon(\EE')\subseteq \smodad(\EE)$ denotes the subcategory of objects of the form $\Upsilon(E)$ with $E\in \EE'$, and $Q^{-1}(\EE')\subseteq \smodad(\EE)$ denotes the subcategory of objects $F$ satisfying $Q(F)\in \EE'$.

\begin{lemma}\label{Lemma:AdmissibleModSubcategories}
Let $\EE$ be an exact category, and let $\EE'$ be a subcategory of $\EE$ closed under extensions and kernels of deflations. Then $\smodad(\EE')$ is equivalent to the subcategory $Q^{-1}(\EE')\cap \operatorname{gen}_2(\Upsilon(\EE'))$ of $\smodad(\EE)$.
\end{lemma}

\begin{proof}
An object $F\in \smodad(\EE)$ is in the essential image of the functor $\smodad(\EE')\to \smodad(\EE)$ if and only if there exists a projective presentation 
\begin{equation}\label{projective presentation}
\Upsilon(X)\xrightarrow{\Upsilon(f)}\Upsilon(Y)\to F\to 0
\end{equation}
where $X,Y\in \EE'$ and $f$ is admissible in $\EE'$. Since $Q(F)=\coker (f)$, it is clear that the essential image is contained in $Q^{-1}(\EE')\cap \operatorname{gen}_2(\Upsilon(\EE'))$. Conversely, if $F$ is in $Q^{-1}(\EE')\cap \operatorname{gen}_2(\Upsilon(\EE'))$, then there exists a projective presentation \eqref{projective presentation} where $X,Y\in \EE'$. Since $Q$ preserves admissible morphisms, it follows that $f$ is admissible and $\coker (f)\in \EE'$. But then $f$ must be admissible in $\EE'$, since $\EE'$ is closed under kernels of deflations. This proves the claim.
\end{proof}

We get the following corollary for subcategories of abelian categories with enough projectives.

\begin{proposition}\label{Corollary:SubcategoriesAbelian}
Let $\PP$ be an additive category, and assume $\smod(\PP)$ is an abelian category. Let $\XX$ be a subcategory of $\smod(\PP)$ closed under extensions, kernels of deflations, and which contain $\Upsilon(\PP)$. Then 
\[
\smodad(\XX)=\{F\in \smod(\XX)\mid F|_{\PP}\in \XX\}.
\]
\end{proposition}

\begin{proof}
Note that the localization $Q\colon \smod(\smod(\PP))\to \smod(\PP)$ is given by the restriction functor $Q(F)=F|_{\PP}$. The claim follows therefore from \Cref{Lemma:AdmissibleModSubcategories} and the fact that $\smod(\XX)$ can be identified with the subcategory $\operatorname{gen}_2(\Upsilon(\XX))$ of $\smod(\AA)$.
\end{proof}

In particular, \Cref{Corollary:SubcategoriesAbelian} holds in the following cases:
\begin{itemize}
\item $\XX$ is the subcategory 
\[
\operatorname{GP}(\PP)=\{M\in \smod(\PP)\mid M=Z^0(P_{\bullet}) \mbox{ for a totally acyclic complex }P_{\bullet}\}
\] of Gorenstein-projective objects in $\smod(\PP)$, where an acyclic complex $P_{\bullet}$ with projective components is called totally acyclic if the complex $\operatorname{Hom}_{\smod(\PP)}(P_{\bullet},Q)$ is acyclic for all projective objects $Q$ in $\smod(\PP)$.
\item $\XX$ is the subcategory of objects of projective dimension $\leq i$ for some integer $i$.
\item $\XX$ is the subcategory of objects of finite projective dimension.
\item $\XX={}^{{}_{>0}\perp}U=\{M\in \smod(\PP)\mid \Ext^i_{\smod(\PP)}(M,U)=0$ for all $i>0\}$ for some object $U\in \smod(\PP)$. In particular, this covers orthogonal categories of cotilting modules as studied in \cite{Iyama07} and \cite{Enomoto18}.

\item $R$ is a commutative Cohen-Macaulay local ring, $\Lambda$ is an $R$-order, i.e. a noetherian $R$-algebra which is maximal Cohen-Macaulay as an $R$-module, see \cite[Section 2.2]{Iyama18} and $\XX$ is the subcategory
\[
\operatorname{CM}(\Lambda)=\{M\in \smod(\Lambda)\mid M \text{ is maximal Cohen-Macaulay as an } R \text{-module}\}  
\]
of maximal Cohen-Macaulay modules of $\Lambda$.
\end{itemize}

We end this section by restating \Cref{Corollary:SubcategoriesAbelian} when the subcategories of Gorenstein projective modules or Cohen-Macaulay modules are of finite type. Here, for an exact category $\EE'$ with an additive generator $M$ and endomorphism ring $\Gamma=\End_{\EE'}(M)$, we let $\smodad(\Gamma)$ be the subcategory of $\smod(\Gamma)$ which corresponds to the subcategory $\smodad(\EE')$ under the equivalence $\smod(\Gamma)\cong \smod(\EE')$.

\begin{corollary}\label{Corollary:GorensteinProjective}
Let $\Lambda$ be a noetherian ring, and assume the subcategory $\operatorname{GP}(\Lambda)$ of Gorenstein projective modules in $\smod(\Lambda)$ has an additive generator $M\in \operatorname{GP}(\Lambda)$. Assume also for simplicity that $\Lambda$ is a summand of $M$. Set $\Gamma=\operatorname{End}_{\Lambda}(M)$ and let $e\in \Gamma$ be the idempotent corresponding to $\Lambda$, so that $e\Gamma e\cong \Lambda$. Then 
\[
\smodad(\Gamma)=\{M\in \smod(\Gamma)\mid Me\in \operatorname{GP}(\Lambda)\}.
\]
\end{corollary}

\begin{corollary}\label{Corollary:CohenMacaulay}
Let $\Lambda$ be an $R$-order over a local commutative Cohen-Macaulay ring $R$, and assume the subcategory $\operatorname{CM}(\Lambda)$ of maximal Cohen-Macaulay has an additive generator $M\in \operatorname{CM}(\Lambda)$. Assume also for simplicity that $\Lambda$ is a summand of $M$. Set $\Gamma=\operatorname{End}_{\Lambda}(M)$ and let $e\in \Gamma$ be the idempotent corresponding to $\Lambda$. Then 
\[
\smodad(\Gamma)=\{M\in \smod(\Gamma)\mid Me\in \operatorname{CM}(\Lambda)\}.
\]
\end{corollary}

We end this section with a description of $\modad(\EE)$ where $\EE$ does not have enough projectives.

\begin{example}
We will use notation and conventions from \cite{Hartshorne}.  Let $k$ be an algebraically closed field.  Let $S$ be a noetherian graded $k$-algebra generated in degrees zero and one.  We write $\bX = \Proj S$ for the associated projective scheme.

Let $\Gr S$ be the category of graded $S$-modules and $\Tor S$ be the full subcategory of torsion modules. We write
\[\QGr \bX \coloneqq \frac{\Gr S}{\Tor S},\]
and we write $\pi$ for the quotient $\Gr S \to \QGr S$.  For each $n \in \bZ$, we write $\OO_\bX(n)$ for $\pi(S(n))$.  By Serre's theorem (see \cite[Proposition~7.8]{Serre55}, \cite[Proposition II.5.15]{Hartshorne}, \cite[EGAII, 3.3.5]{EGAII}), the functor
\[\Gamma_*\colon \QCoh \bX \to \Gr S\colon \FF \mapsto \oplus_{n \in \bZ} \Hom(\OO_\bX(-n), \FF)\]
induces an equivalence $\pi \circ \Gamma_*\colon \QCoh \bX \xrightarrow{\simeq} \QGr S$.  The $S$-module structure of $\Gamma_*(\FF)$ is as in \cite[p.118]{Hartshorne}.  The functor $\Gamma_*$ factors as $\Coh \bX \xrightarrow{\Yoneda} \mod (\Coh \bX) \xrightarrow{G} \Gr S$, where $G$ is the exact functor given by $M \mapsto \oplus_{n \in \bZ} M(\OO_\bX(-n)).$

A map $f \colon \FF \to \GG$ in $\Coh \bX$ is a deflation if and only if $\coker(\Gamma_*(f)) \in \Tor S.$  As $\GG \in \Coh \bX$, we know that $\Gamma_{\geq 0}(\GG) \coloneqq \oplus_{n \geq 0} \Hom(\OO_X(-n), \GG)$ is a finitely generated graded $S$-module (see \cite[\S60, Th\'{e}or\`{e}me~2]{Serre55}).  Hence, $C \coloneqq \coker \Yoneda(f)$ is effaceable if and only if $C(\OO_\bX(-n)) = 0$ for $-n \ll 0$.  We find the following description of the effaceable modules:
\[\eff (\Coh \bX) = \{C \in \mod (\Coh \bX) \mid \mbox{$C(\OO_\bX(n)) = 0$, for $0 \ll n$}\}.\]
Consider the localization sequence $\eff(\Coh \bX) \to \mod (\Coh \bX) \xrightarrow{Q} \Coh \bX$, given by Auslander's formula (see, for example, \Cref{Theorem:AuslandersFormulaForExactCategories}).  Here, $Q$ is the unique right exact functor satisfying $Q \circ \Yoneda \cong 1$.  For notational convenience, we identify $\QCoh \bX$ and $\QGr S$ using $\pi \circ \Gamma_*$.  As $1 \cong \pi \circ \Gamma_* \cong \pi \circ G \circ \Yoneda$, and both $\pi$ and $G$ are right exact, we find that $Q \cong \pi \circ G$.

Write $\EE$ for the full subcategory of $\Coh \bX$ consisting of the locally free sheaves.  It follows from \Cref{Lemma:AdmissibleModSubcategories} that
\[\modad(\EE) = \left\{ M \in \mod(\EE) \mid \pi \left( \oplus_{n \in \bZ} M(\OO_\bX(-n)) \right) \in \EE \right\}.\]
\end{example}

\begin{remark}
We now consider the case where $S = k[x,y]$, and hence $\bX \cong \bP^1_k$.  It is well-known that each line bundle is isomorphic to $\OO(n)$, for some $n \in \bZ$ and that each vector bundle is a direct sum of line bundles (see \cite{Grothendieck57b} when $k = \bC$ or \cite{HazewinkelMartin82} for more general fields $k$).  For each $m,n \in \bZ$, we have $\Hom(\OO(m), \OO(n)) \cong S_{m-n}$ (see, for example, \cite[Proposition~II.5.13]{Hartshorne}).  From these results, we find $\mod(\EE) \cong \gr S$, where $\gr S$ is the category of finitely generated graded $S$-modules.  In particular, $\eff(\EE)$ is a length category.  Let $M = k(n)$ be an arbitrary simple $S$-module (that is, $M_d = k$ when $d = -n$ and zero otherwise).  Consider the Koszul resolution
\[0 \to S(n-2) \to S(n-1)^{\oplus 2} \to S(n) \to k(n) \to 0\]
in $\Gr(S).$  Applying $\pi$, we find the exact sequence $0 \to \OO(n-2) \to \OO(n-1)^{\oplus 2} \to \OO(n) \to 0$, which is the Euler sequence for $\bP^1_k$ when $n=2$, see \cite[Proposition~II.8.13]{Hartshorne} or \cite[Proposition~2.4.4]{Huybrechts}.  It follows from \cite[Proposition~A.2]{Enomoto18} that the morphism $\OO(n-1)^{\oplus 2} \to \OO(n)$ is right almost split, so that the exact sequence is an almost split sequence in $\EE$ (see \cite[Proposition~2.12]{IyamaNakaokaPalu18} \cite[Theorem~4.19]{Shah20}).  Note that it is known (\cite[Corollary~5]{Jorgensen06}, see also \cite{ReitenVendenBergh02}) that the Euler sequence is an almost split sequence in $\Coh \bP^1_k$.
\end{remark}
\section{Characterizing exact structures via resolving subcategories}\label{section:ViaResolvingSubcategories}

Let $\CC$ be an idempotent complete additive category. In this section we characterize the subcategories of $\smod (\CC)$ of the form $\smodad (\EE)$ where $\EE$ is an exact category obtained by choosing an exact structure on $\CC$. 
\subsection{Resolving subcategories}

Fix an idempotent complete exact category $\EE$. Our goal in this section is to investigate the relationship between $\eff (\EE)$ and $\smodad (\EE)$ as subcategories of $\smod (\EE)$.

\begin{definition}\label{Definition:Resolving}
A subcategory $\XX$ of $\EE$ is called \emph{resolving} if it is closed under extensions, direct summands, and kernel of deflations, and if it is \emph{generating}, i.e. $\gen(\XX)=\EE$.
\end{definition}

\begin{definition}
Let $\PP^2(\EE)$ denote the subcategory of $\smod (\EE)$ consisting of all $F$ for which there exists an exact sequence
\[
0\to \Hom_{\EE}(-,X)\to \Hom_{\EE}(-,Y)\to \Hom_{\EE}(-,Z)\to F\to 0.
\]
\end{definition}

\begin{lemma}
$\PP^2(\EE)$ is closed under extensions, direct summands and kernels of epimorphisms.
\end{lemma}

\begin{proof}
$\PP^2(\EE)$ is the subcategory of objects of projective dimension $\leq 2$ in the category $\smod_{\infty}(\EE)$ consisting of all $F$ which admits an exact sequence
\[
\cdots \to \Hom_{\EE}(-,X_2)\to \Hom_{\EE}(-,X_1)\to F\to 0.
\]
Since $\smod_{\infty}(\EE)$ is closed under extensions, direct summands and kernels of epimorphisms by \cite[Proposition 2.6]{Enomoto17}, the same must hold for $\PP^2(\EE)$.
\end{proof}

\begin{proposition}\label{Proposition:SmallestResolvingSubcategory} 
$\smodad (\EE)$ is the smallest resolving subcategory of $\PP^2(\EE)$ containing $\eff (\EE)$.
\end{proposition}

\begin{proof}
We first show that $\smodad(\EE)$ is a resolving subcategory of $\PP^2(\EE)$. Note that $\smodad(\EE)$ is closed under extensions by \Cref{Proposition:AdmissiblyPresentedLiesExtensionClosed} and is idempotent complete by \Cref{Corollary:(Weak)IdempotentCompleteness}, and therefore closed under direct summands. Also, since $\smodad(\EE)$ contains the representable functors, it is generating in $\PP^2(\EE)$. To show closure under kernels of deflations, it suffices by \cite[Lemma 2.5]{Enomoto17} to show that if $0\to K\to \Hom_{\EE}(-,E)\xrightarrow{p}G\to 0$ is an exact sequence with $G\in \smodad(\EE)$, then $K\in \smodad(\EE)$. To this end, let $0\to F\to G\xrightarrow{q}H\to 0$ be a conflation with $F\in \eff (\EE)$ and $H\in \FF$. Then we get a commutative diagram
\[\xymatrix{
			K'\ar[d]^{r}\ar@{>->}[r] & \Hom_{\EE}(-,E)\ar[d]^{p}\ar@{->>}[r]^{} & H\ar@{=}[d]\\
			F\ar@{>->}[r] & G\ar@{->>}[r]^{q}& H
		}\]
where the rows are conflations. It follows that the left hand square is a pullback square, so $r$ is an epimorphism and $\ker (r)\cong \ker (p)\cong K$. Since $H$ has projective dimension $\leq 1$, it follows that $K'$ is projective, and hence $K'\cong \Hom_{\EE}(-,E')$ for some $E'\in \EE$ since $\EE$ is idempotent complete. Hence $K'\in \smodad (\EE)$, and therefore $r$ is admissible in $\smodad(\EE)$ by \Cref{Proposition:EffAreTorsion}. Hence $K\cong \ker (r)\in \smodad(\EE)$, so $\smodad (\EE)$ is closed under kernels of deflations and therefore a resolving subcategory of $\PP^2(\EE)$. 		
		
Now let $\XX\subset \PP^2(\EE)$ be a resolving subcategory containing $\eff (\EE)$. Then by \Cref{Lemma:TorsionFreeHasWeakInflationToProjective}.\eqref{Item:TorsionFreeHasWeakInflationToProjective3} we have that $\XX$ must contain $\FF$. Since $(\eff (\EE),\FF)$ is a torsion pair in $\smodad(\EE)$ and $\XX$ is closed under extensions, $\XX$ must contain $\smodad(\EE)$. This proves the claim.			
\end{proof}

\begin{remark}\label{remark:WhyIdempotentComplete}
Note that \Cref{Proposition:SmallestResolvingSubcategory} and \Cref{Theorem:ExactStructuresResolving} below are not true without assuming the underlying category of $\EE$ is idempotent complete. Indeed, if $\smodad(\EE)$ is resolving, then $\smodad(\EE)$ is closed under direct summands. Therefore $\smodad(\EE)$ must be idempotent complete. By \Cref{Corollary:(Weak)IdempotentCompleteness} this forces $\EE$ to be idempotent complete.
\end{remark}

Using \Cref{Proposition:SmallestResolvingSubcategory}, we get the following analogue of \Cref{Lemma:ResolutionsOfEffaceables} for admissible morphisms in an idempotent complete exact category.

\begin{proposition}\label{Proposition:ResolutionsOfAdmissibles}
	Let $\Upsilon(X)\xrightarrow{\Upsilon(g)}\Upsilon(Y)\deflation E\to 0$ be a projective presentation of $E\in \smodad(\EE)$. Then $g\colon X\to Y$ is an admissible morphism in $\EE$.
\end{proposition}

\begin{proof}
Since $\smodad(\EE)$ is closed under kernels of deflations by \Cref{Proposition:SmallestResolvingSubcategory}, and $\Upsilon(X),\Upsilon(Y)$ and $E$ are in $\smodad(\EE)$, it follows that $\Upsilon(g)$ is an admissible morphism in $\smodad(\EE)$. Furthermore, since the localization functor $\smodad(\EE)\to \EE$ is exact, it preserves admissible morphisms. Applying it to $\Upsilon(g)$, we get that $g\colon X\to Y$ is an admissible morphism in $\EE$. This proves the claim. 
\end{proof}

\subsection{The Auslander-Bridger transpose}
Fix an idempotent complete additive category $\CC$. We let $\underline{\smod}(\CC)$ denote the projectively stable category of $\smod (\CC)$. Explicitly, $\underline{\smod}(\CC)$ has the same objects as $\smod (\CC)$, and
\[
\Hom_{\underline{\smod}(\CC)}(F,G)=\Hom_{\smod (\CC)}(F,G)/\PP(F,G)
\]
where $\PP(F,G)$ denotes the set of morphisms $F\to G$ factoring through a representable functor.

We now recall the Auslander-Bridger transpose, which was first introduced and studied in \cite{Auslander65} and \cite{AuslanderBridger69}.

\begin{definition}
The \emph{Auslander-Bridger transpose} is a functor 
\[\operatorname{Tr}\colon \underline{\smod}(\CC)\to \underline{\smod}(\CC^{\operatorname{op}})^{\operatorname{op}}
\]
defined as follows:
\begin{enumerate}
\item For each object $F\in \underline{\smod}(\CC)$ choose a projective presentation 
\[
\Hom_{\CC}(-,X)\xrightarrow{\Hom_{\CC}(-,f)}\Hom_{\CC}(-,Y)\to F\to 0
\]
in $\smod (\CC)$ and set 
\[
\operatorname{Tr}(F)=\coker(\Hom_{\CC}(Y,-)\xrightarrow{\Hom_{\CC}(f,-)}\Hom_{\CC}(X,-)).
\]
\item For each morphism morphism $\underline{\phi}\colon F\to G$ in $\underline{\smod}(\CC)$ choose a representative $\phi\colon F\to G$ in $\smod (\CC)$ and a lift
\[\xymatrix@C=4em{
		\Hom_{\CC}(-,X)\ar[r]^{\Hom_{\CC}(-,f)}\ar[d]^{\Hom_{\CC}(-,k)} & \Hom_{\CC}(-,Y)\ar[r]^{}\ar[d]^{\Hom_{\CC}(-,h)}  & F\ar[d]^{\phi} \ar[r]&  0\\
		\Hom_{\CC}(-,X')\ar[r]^{\Hom_{\CC}(-,f')} & \Hom_{\CC}(-,Y')\ar[r]^{}  & G \ar[r] & 0
	}\]
	and define $\operatorname{Tr}(\underline{\phi})\colon \operatorname{Tr}(F)\to \operatorname{Tr}(G)$ to be the image in $\underline{\smod}(\CC)$ of the unique morphism in $\smod(\CC)$ making the diagram 
\[\xymatrix@C=4em{
		\Hom_{\CC}(Y',-)\ar[r]^{\Hom_{\CC}(f',-)}\ar[d]^{\Hom_{\CC}(h,-)} & \Hom_{\CC}(X',-)\ar[r]^{}\ar[d]^{\Hom_{\CC}(k,-)}  & \operatorname{Tr}(G)\ar[d]^{} \ar[r]&  0\\
		\Hom_{\CC}(Y,-)\ar[r]^{\Hom_{\CC}(f,-)} & \Hom_{\CC}(X,-)\ar[r]^{}  & \operatorname{Tr}(F) \ar[r] & 0
	}\]
	commutative.	
\end{enumerate}
\end{definition}
By definition we have that $\operatorname{Tr}^2\cong 1_{\underline{\smod}(\CC)}$, and $\operatorname{Tr}$ therefore induces an equivalence between $\underline{\smod}(\CC)$ and $ \underline{\smod}(\CC^{\operatorname{op}})^{\operatorname{op}}$. Since $\operatorname{Tr}$ is a functor to the projectively stable category, the object $\operatorname{Tr}(F)$ is well-defined in $\smod(\CC)$ up to projective summands. By abuse of notation, we sometimes write $\operatorname{Tr}(F)$ for a choice of an object in $\smod(\CC^{\operatorname{op}})$ which is isomorphic in $\underline{\smod}(\CC^{\operatorname{op}})$ to the image of $\operatorname{Tr}\colon \underline{\smod}(\CC)\to \underline{\smod}(\CC^{\operatorname{op}})^{\operatorname{op}}$ at $F\in \underline{\smod}(\CC)$.  

In the following, for a functor $F\in \Mod(\CC)$ we let $F^*\in \Mod(\CC^{\operatorname{op}})$ denote the functor given by 
\[
F^*(X)=\Hom_{\Mod(\CC)}(F,\Hom_{\CC}(-,X))
\]
and we let $\operatorname{ev}_F\colon F\to F^{**}$ denote the natural transformation sending $x\in F(X)$ to 
\[
\operatorname{ev}_F(x)\colon F^*\to \Hom_{\CC}(X,-) \quad \operatorname{ev}_F(x)(\psi)=\psi(x).
\]
Finally, for $i>0$ we let $\Ext^i_{\CC}(F,\CC)\in \Mod(\CC^{\operatorname{op}})$ denote the functor given by 
\[
\Ext^i_{\CC}(F,\CC)(X)=\Ext^i_{\Mod(\CC)}(F,\Hom_{\CC}(-,X)).
\]

\begin{proposition}[Proposition 6.3 in \cite{Auslander65}]\label{Proposition:AuslanderExactSeq}
For any $F\in \smod(\CC)$ we have an exact sequence 
\[
0\to \Ext^1_{\CC^{\operatorname{op}}}(\operatorname{Tr}(F),\CC^{\operatorname{op}})\to F\xrightarrow{\operatorname{ev}_F}F^{**}\to \Ext^2_{\CC^{\operatorname{op}}}(\operatorname{Tr}(F),\CC^{\operatorname{op}})\to 0.
\]
\end{proposition}

We end this subsection by recalling the definition of the grade, which was first introduced in \cite{Rees57}.

\begin{definition}\label{Definition:grade}
The \emph{grade} of $F\in \operatorname{Mod}(\CC)$, denoted $\operatorname{grade}F$, is the biggest integer $i\geq 0$ such that $\Ext^j_{\Mod(\CC)}(F,\Hom_{\CC}(-,X))=0$ for all $j<i$ and $X\in \CC$. Here we let $\Ext^0_{\Mod(\CC)}(F,\Hom_{\CC}(-,X)):=\Hom_{\Mod(\CC)}(F,\Hom_{\CC}(-,X))$).
\end{definition}

Let $\EE$ be an exact category with underlying additive category $\CC$, and set $\QQ=\operatorname{Proj}(\smodad(\EE))$. Recall that
\[
\FF=\operatorname{cogen}(\QQ)=\{F\in \smodad(\EE)\mid\operatorname{proj.dim}F\leq 1\}
\]
The first equality follows from \Cref{Proposition:TorsionTheoryProperties} \eqref{Proposition:TorsionTheoryPropertiesB}, and the second equality follows by definition since $\FF$ consists of all functors which admit a projective presentation by an inflation. Similarly
\[
\eff(\EE)={}^{\perp}\QQ\subset \{0\}\cup\{F\in \smodad(\EE)\mid\operatorname{proj.dim}F=2\}
\]
where the first equality follows from \Cref{Proposition:TorsionTheoryProperties} \eqref{Proposition:TorsionTheoryPropertiesA}.  To see that the inclusion holds, let $F\in \eff(\EE)$ and let $0\to \Upsilon(X)\xrightarrow{\Upsilon(f)}\Upsilon(Y)\xrightarrow{\Upsilon(g)}\Upsilon(Z)\xrightarrow{} F\to 0$ be a projective resolution where $g$ is a deflation. Then either $g$  is a split epimorphism, in which case $F\cong 0$, or $g$ is not a split epimorphism, in which case $\operatorname{proj.dim}F=2$ since $\Upsilon(f)$ is not a split monomorphism. It follows from \Cref{Proposition:TorsionTheoryProperties} \eqref{Proposition:TorsionTheoryPropertiesC} that $\eff(\EE)$ also has the following description using the notion of grade
\[
\eff(\EE)=\{F\in \smodad(\EE)\mid\operatorname{grade}F=2\}
\]
Now consider the Auslander-Bridger transpose $\operatorname{Tr}\colon  \underline{\smod}(\CC)\to \underline{\smod}(\CC^{\operatorname{op}})^{\operatorname{op}}$ . It restricts to an equivalence
\[
\operatorname{Tr}\colon  \underline{\smodad}(\EE)\to \underline{\smodad}(\EE^{\operatorname{op}})^{\operatorname{op}}
\]
and to equivalences 
\[
\operatorname{Tr}\colon  \underline{\FF}\to \underline{\eff}(\EE^{\operatorname{op}})^{\operatorname{op}} \quad \text{and} \quad \operatorname{Tr}\colon  \underline{\eff}(\EE)\to \underline{\FF'}^{\operatorname{op}}
\]
where $\FF'\subset \smodad(\EE^{\operatorname{op}})$ denotes the subcategory of functors admitting a projective presentation by an inflation. Here $\underline{\smodad}(\EE)$ and $\underline{\FF}$ and $\underline{\eff}(\EE)$ denotes the full subcategories closed under isomorphism in $\underline{\smod}(\CC)$ and containing the objects in $\smodad(\EE)$ and $\FF$ and $\eff(\EE)$, respectively.
\subsection{Main theorem}
Fix an idempotent complete additive category $\CC$. We are now ready to state our main theorem relating subcategories of $\smod(\CC)$ with exact structures on $\CC$. Here, for a subcategory $\XX$ of $\smod(\CC)$ we let $\operatorname{Tr}(\XX)$ denote the subcategory of $\smod(\CC^{\operatorname{op}})$ consisting of all objects $G$ which are isomorphism in $\underline{\smod}(\CC^{\operatorname{op}})$ to an object $\operatorname{Tr}(F)$ with $F\in \XX$. Since we encounter multiple exact structures on $\CC$, we use the letters $\mathfrak{C}$, $\mathfrak{D}$ and $\mathfrak{E}$ for classes of conflations in $\CC$.  Thus, a conflation category or an exact category is written as a pair $(\CC,\mathfrak{C})$, and the associated Auslander exact category is denoted by $\smodad(\mathfrak{C})$.

\begin{theorem}\label{Theorem:ExactStructuresResolving}
Let $\CC$ be an idempotent complete additive category. The association $\mathfrak{C}\mapsto \smodad(\mathfrak{C})$ gives a bijection between:
\begin{enumerate}
\item Classes of conflations $\mathfrak{C}$ in $\CC$ such that $(\CC,\mathfrak{C})$ is an exact category.
\item Subcategories $\XX$ of $\smod(\CC)$ satisfying the following:
\begin{enumerate}
\item\label{Property:1} $\XX$ is a resolving subcategory of $\PP^2(\CC)$ and $\operatorname{Tr}(\XX)$ is a resolving subcategory of $\PP^2(\CC^{\operatorname{op}})$, and
\item\label{Property:2} $\XX$ and $\operatorname{Tr}(\XX)$ have no objects of grade $1$.
\end{enumerate}
\end{enumerate}
\end{theorem}

\begin{remark}
The conflations $\mathfrak{C}$ on $\CC$ for which $(\CC,\mathfrak{C})$ is an exact category form a poset under inclusion. In fact, this is a bounded complete lattice where the meet is given by the intersection, see  \cite[Theorem 5.3]{BrustleHassounLangfordRoy20}. Similarly, the classes of subcategories $\XX$ of $\smod(\CC)$ satisfying condition \eqref{Property:1} and \eqref{Property:2} in \Cref{Theorem:ExactStructuresResolving} form a poset under inclusion, with meet given by the intersection. Since for two exact categories  $(\CC,\mathfrak{C})$ and  $(\CC,\mathfrak{C}')$  we have $\mathfrak{C}\subseteq \mathfrak{C}'$ if and only if $\smodad(\mathfrak{C})\subseteq \smodad(\mathfrak{C}')$, the bijection in \Cref{Theorem:ExactStructuresResolving} is an isomorphism of posets. In particular, the classes of subcategories $\XX$ of $\smod(\CC)$ satisfying condition \eqref{Property:1} and \eqref{Property:2} form a complete bounded lattice under inclusion.
\end{remark}

Our goal in this section is to prove \Cref{Theorem:ExactStructuresResolving}. 
\begin{lemma}\label{Lemma:WhenXandTr(X)AreResolving}
Assume $\XX$ is a subcategory of $\smod(\CC)$ satisfying condition \eqref{Property:1} in \Cref{Theorem:ExactStructuresResolving}, and let $F\in \XX$. The following hold:
\begin{enumerate}
\item\label{Item:DualRepresentable} $F^*\cong \Hom_{\CC}(Z,-)$ for some $Z\in \CC$.
\item\label{Item:EvaluationAdmissible} The morphism $\operatorname{ev}_F\colon F\to F^{**}$ is admissible in $\XX$.
\item\label{Item:ExtInX} $\Ext^1_{\CC}(F,\CC)\in \operatorname{Tr}(\XX)$ and $\Ext^2_{\CC}(F,\CC)\in \operatorname{Tr}(\XX)$.
\end{enumerate}
\end{lemma}

\begin{proof}
\begin{enumerate}
\item Recall that there exists an exact sequence
\[
0\to F^*\to \Hom_{\CC}(X,-)\xrightarrow{\Hom_{\CC}(f,-)}\Hom_{\CC}(Y,-)\to \operatorname{Tr}(F)\to 0.
\]
Since $\operatorname{Tr}(F)\in \operatorname{Tr}(\XX)\in \PP^2(\CC^{\operatorname{op}})$, it follows that $F^*$ is projective in $\smod(\CC^{\operatorname{op}})$ and hence $F^*\cong \Hom_{\CC}(Z,-)$ for some $Z\in \CC$ since $\CC$ is idempotent complete.
\item Consider the exact sequence 
\[
0\to \Ext^1_{\CC^{\operatorname{op}}}(\operatorname{Tr}(F),\CC^{\operatorname{op}})\to F\xrightarrow{\operatorname{ev}_F}F^{**}\to \Ext^2_{\CC^{\operatorname{op}}}(\operatorname{Tr}(F),\CC^{\operatorname{op}})\to 0.
\]
from \Cref{Proposition:AuslanderExactSeq}. Since $F^{**}$ is representable by part \eqref{Item:DualRepresentable}, it is contained in $\XX$. Therefore, since $\XX$ is closed under kernels of epimorphisms, it suffices to show that $\Ext^2_{\CC^{\operatorname{op}}}(\operatorname{Tr}(F),\CC^{\operatorname{op}})$ is in $\XX$. Now by choosing an exact sequence
\[
0\to \Hom_{\CC}(Z,-)\to \Hom_{\CC}(Y,-)\xrightarrow{\Hom_{\CC}(f,-)} \Hom_{\CC}(X,-)\to \operatorname{Tr}(F)\to 0
\]
we immediately see that $\Ext^2_{\CC^{\operatorname{op}}}(\operatorname{Tr}(F),\CC^{\operatorname{op}})\cong \operatorname{Tr}(\im (\Hom_{\CC}(f,-)))$. Since $\operatorname{Tr}(\XX)$ is closed under kernels of epimorphisms, we get that $\im (\Hom_{\CC}(f,-))\in \operatorname{Tr}(\XX)$, and hence $\Ext^2_{\CC^{\operatorname{op}}}(\operatorname{Tr}(F),\CC^{\operatorname{op}})\in\XX$.
\item \Cref{Proposition:AuslanderExactSeq} applied to $\operatorname{Tr}(F)$ gives an exact sequence 
\[
0\to \Ext^1_{\CC}(F,\CC)\to \operatorname{Tr}(F)\xrightarrow{\operatorname{ev}_{\operatorname{Tr}(F)}}\operatorname{Tr}(F)^{**}\to \Ext^2_{\CC}(F,\CC)\to 0.
\]
Hence, the claim follows from part \eqref{Item:EvaluationAdmissible} applied to $\operatorname{Tr}(F)$ and $\operatorname{Tr}(\XX)$.\qedhere
\end{enumerate}
\end{proof}

\begin{lemma}\label{Lemma:EquivalentConditions}
Assume $\XX$ is a subcategory of $\smod(\CC)$ satisfying condition \eqref{Property:1} in \Cref{Theorem:ExactStructuresResolving}. The following are equivalent:
\begin{enumerate}
\item For all $F\in \XX$ of projective dimension $\leq 1$, the map $\operatorname{ev}_F\colon F\to F^{**}$ is an inflation in $\XX$.
\item $\operatorname{Tr}(\XX)$ has no objects of grade $1$.
\end{enumerate}
\end{lemma}

\begin{proof}
For each $F\in \XX$ choose a projective presentation $\Hom_{\CC}(-,Y)\xrightarrow{\Hom_{\CC}(-,f)}\Hom_{\CC}(-,X)\to F\to 0$ such that $\Hom_{\CC}(-,f)$ is a monomorphism if $F$ has projective dimension $\leq 1$. Define $\operatorname{Tr}(F)$ via the exact sequence $\Hom_{\CC}(X,-)\xrightarrow{\Hom_{\CC}(f,-)}\Hom_{\CC}(Y,-)\to \operatorname{Tr}(F)$. Then clearly $F$ has projective dimension $\leq 1$ if and only if $\operatorname{grade}\operatorname{Tr}(F)>0$. Furthermore, $\operatorname{ev}_F$ is an inflation in $\XX$ if and only if $\Ext^1_{\CC^{\operatorname{op}}}(\operatorname{Tr}(F),\CC^{\operatorname{op}})\cong 0$ by \Cref{Proposition:AuslanderExactSeq} and \Cref{Lemma:WhenXandTr(X)AreResolving}.\eqref{Item:EvaluationAdmissible}. Hence, $\operatorname{ev}_F$ is an inflation for any object $F\in \XX$ of projective dimension $\leq 1$ if and only if no object of the form $\operatorname{Tr}(F)$ with $F\in \XX$ has grade $1$. Since any object in $\operatorname{Tr}(\XX)$ is isomorphic in $\underline{\smod} (\CC^{\operatorname{op}})$ to an object of the form $\operatorname{Tr}(F)$ with $F\in \XX$, this is equivalent to $\operatorname{Tr}(\XX)$ having no objects of grade $1$.
\end{proof}

\begin{definition}
Assume $\XX$ is a subcategory of $\smod(\CC)$ satisfying condition \eqref{Property:1} and \eqref{Property:2} in \Cref{Theorem:ExactStructuresResolving}, and let $f\colon X\to Y$ be a morphism in $\CC$.
\begin{enumerate}
\item $f$ is an $\XX$\emph{-inflation} if $\Hom_{\CC}(-,f)\colon \Hom_{\CC}(-,X)\to \Hom_{\CC}(-,Y)$ is an inflation in $\XX$.
\item $f$ is an $\XX$\emph{-deflation} if $\Hom_{\CC}(f,-)\colon \Hom_{\CC}(Y,-)\to \Hom_{\CC}(X,-)$ is an inflation in $\operatorname{Tr}(\XX)$.
\item $f$ is $\XX$\emph{-admissible} if $\Hom_{\CC}(-,f)\colon \Hom_{\CC}(-,X)\to \Hom_{\CC}(-,Y)$ is admissible in $\XX$.
\end{enumerate}
\end{definition}

\begin{remark}\label{Remark:self-dualityXadmissible}
Note that $f\colon X\to Y$ being $\XX$-admissible is also equivalent to 
\[
\Hom_{\CC}(f,-)\colon \Hom_{\CC}(Y,-)\to \Hom_{\CC}(X,-)
\]
being admissible in $\operatorname{Tr}(\XX)$. Indeed, since $\XX$ and $\operatorname{Tr}(\XX)$ are closed under kernels of deflations, it follows that $\Hom_{\CC}(-,f)$ or $\Hom_{\CC}(f,-)$ is admissible in $\XX$ or $\operatorname{Tr}(\XX)$ if and only if $\coker \Hom_{\CC}(-,f)\in \XX$ or $\coker \Hom_{\CC}(f,-)\in \operatorname{Tr}(\XX)$, respectively. Since by definition $\coker \Hom_{\CC}(-,f)\in \XX$ if and only if $\coker \Hom_{\CC}(f,-)\in \operatorname{Tr}(\XX)$, the claim follows.   
\end{remark}

\begin{proposition}\label{Proposition:ResolvingGivesExactStructure}
Assume $\XX$ is a subcategory of $\smod(\CC)$ satisfying condition \eqref{Property:1} and \eqref{Property:2} in \Cref{Theorem:ExactStructuresResolving}. Then there exists a unique conflation category $(\CC,\mathfrak{C})$ with inflations and deflations the $\XX$-inflations and $\XX$-deflations, respectively.
\end{proposition}

\begin{proof}
Let $f$ be an $\XX$-inflation, and consider the exact sequences 
\begin{align*}
0 \to \Hom_{\CC}(-,X)\xrightarrow{\Hom_{\CC}(-,f)}\Hom_{\CC}(-,Y)\to F\to 0 \\
0\to F^*\to \Hom_{\CC}(Y,-)\xrightarrow{\Hom_{\CC}(f,-)}\Hom_{\CC}(X,-)\to \operatorname{Tr}(F)\to 0.
\end{align*}
Since $F^*\cong \Hom_{\CC}(Z,-)$ by \Cref{Lemma:WhenXandTr(X)AreResolving}.\eqref{Item:DualRepresentable}, we get that $f$ admits a cokernel $g\colon Y\to Z$ in $\CC$ which is an $\XX$-deflation. Now since $F$ has projective dimension $\leq 1$, the map $\operatorname{ev}_F\colon F\to F^{**}\cong \Hom_{\CC}(Z,-)$ is an inflation in $\XX$ by \Cref{Lemma:EquivalentConditions}. Hence the sequence $0\to \Hom_{\CC}(-,X)\to \Hom_{\CC}(-,Y)\to \Hom_{\CC}(-,Z)$ is exact, so $f$ is the kernel of $g$. This together with the dual argument shows that the $\XX$-inflations and $\XX$-deflations makes $\CC$ into a conflation category.
\end{proof}

\begin{proposition}\label{Proposition:ResolvingGivesExactStructure:2}
Assume $\XX$ is a subcategory of $\smod(\CC)$ satisfying condition \eqref{Property:1} and \eqref{Property:2} in \Cref{Theorem:ExactStructuresResolving}. Then the conflation category in \Cref{Proposition:ResolvingGivesExactStructure} is an exact category.
\end{proposition}
\begin{proof}
Since the $\XX$-inflations and $\XX$-deflations clearly satisfy axioms \ref{L0}, \ref{L1} and \ref{R0}, \ref{R1} for an exact category, it only remains to show that axioms \ref{L2} and \ref{R2} hold. We only prove \ref{L2} since \ref{R2} is dual. To this end, let $f\colon X\to Y$ be an $\XX$-inflation and let $g\colon X\to Z$ be a morphism in $\CC$. Taking the pushout of $\Hom_{\CC}(-,f)\colon \Hom_{\CC}(-,X)\to \Hom_{\CC}(-,Y)$ along $\Hom_{\CC}(-,g)\colon \Hom_{\CC}(-,X)\to \Hom_{\CC}(-,Z)$, we get a commutative diagram 
\[\xymatrix@C=4em{
			 \Hom_{\CC}(-,X)\ar[d]^{\Hom_{\CC}(-,g)}\ar@{>->}[r]^{\Hom_{\CC}(-,f)} &  \Hom_{\CC}(-,Y)\ar[d]^{}\ar@{->>}[r]^{} & F\ar@{=}[d]\\
			 \Hom_{\CC}(-,Z)\ar@{>->}[r] & G\ar@{->>}[r]^{}& F
		}\]
where the rows are exact sequences. Since the left hand square is a pullback and a pushout square, we have an exact sequence  
\begin{equation}\label{ApplyDualToThis}
0\to \Hom_{\CC}(-,X)\xrightarrow{\begin{psmallmatrix}\Hom_{\CC}(-,f)\\ \Hom_{\CC}(-,g)\end{psmallmatrix}}  \Hom_{\CC}(-,Y)\oplus  \Hom_{\CC}(-,Z)\to G\to 0
\end{equation}
and hence $G$ has projective dimension $\leq 1$. Therefore by \Cref{Lemma:EquivalentConditions} the map $\operatorname{ev}_G\colon G\to G^{**}$ is an inflation in $\XX$. Furthermore, $G^{*}\cong \Hom_{\CC}(W,-)$ for some object $W\in \CC$ by \Cref{Lemma:WhenXandTr(X)AreResolving}.\eqref{Item:DualRepresentable}. Let $h\colon Z\to W$ and $k\colon Y\to W$ be the unique morphisms in $\CC$ so that $\Hom_{\CC}(-,h)$ and $\Hom_{\CC}(-,k)$ are equal to the composites $\Hom_{\CC}(-,Z)\inflation G\xrightarrow{\operatorname{ev}_G}G^{**}\cong \Hom_{\CC}(-,W)$ and $\Hom_{\CC}(-,Y)\to G\xrightarrow{\operatorname{ev}_G}G^{**}\cong \Hom_{\CC}(-,W)$, respectively. Since $\Hom_{\CC}(-,h)$ is a composite of inflations in $\XX$, it follows that $h$ is an $\XX$-inflation. Now consider the commutative square 
\[\xymatrix@C=3em{
		X\ar[r]^{f}\ar[d]^{g} & Y\ar[d]^{k}\\
		Z\ar[r]^{h} & W.
	}\] 
It only remains to show that this is a pushout square. But this is true since applying $(-)^*$ to	\eqref{ApplyDualToThis} gives the left exact sequence
\[
0\to \Hom_{\CC}(W,-)\xrightarrow{\begin{psmallmatrix}\Hom_{\CC}(k,-)\\ \Hom_{\CC}(h,-)\end{psmallmatrix}}  \Hom_{\CC}(Y,-)\oplus  \Hom_{\CC}(Z,-)\xrightarrow{\begin{psmallmatrix}\Hom_{\CC}(f,-)& \Hom_{\CC}(g,-)\end{psmallmatrix}} \Hom_{\CC}(X,-).\qedhere
\]
\end{proof}

\begin{proposition}\label{Proposition:AdmissibleMorphisms}
Assume $\XX$ is a subcategory of $\smod(\CC)$ satisfying condition \eqref{Property:1} and \eqref{Property:2} in \Cref{Theorem:ExactStructuresResolving}, and let $(\CC,\mathfrak{C})$ be the exact category given by the $\XX$-inflations and $\XX$-deflations. The following hold:
\begin{enumerate}
\item\label{Item:X-admissible morphisms} The $\XX$-admissible morphisms in $\CC$ are precisely the admissible morphisms in $(\CC,\mathfrak{C})$.
\item $\smodad(\mathfrak{C})=\XX$.
\end{enumerate}
\end{proposition}

\begin{proof}
\begin{enumerate}
\item Let $f\colon X\to Y$ be a morphism in $\CC$. If $f$ is admissible in $(\CC,\mathfrak{C})$, then it can be written as a composite $f=f_2\circ f_1$ where $f_1$ is an $\XX$-deflation and $f_2$ is an $\XX$-inflation. By definition $\Hom_{\CC}(f_1,-)$ is an inflation in $\operatorname{Tr}(\XX)$, and hence $\Hom_{\CC}(-,f_1)$ must be admissible in $\XX$. Since $\Hom_{\CC}(-,f_2)$ is an inflation in $\XX$, the composite 
\[
\Hom_{\CC}(-,f)=\Hom_{\CC}(-,f_2)\circ \Hom_{\CC}(-,f_1)
\]
is admissible in $\XX$. Hence $f$ is $\XX$-admissible.

Now assume $f$ is $\XX$-admissible. First note that $f$ has a kernel which is an $\XX$-inflation and a cokernel which is an $\XX$-deflation, since $\Hom_{\CC}(-,f)$ and $\Hom_{\CC}(f,-)$ are admissible in $\XX$ and $\operatorname{Tr}(\XX)$, respectively, with kernels being representable functors. Therefore $f$ has a coimage $\operatorname{coim}(f)$ obtained by taking the cokernel of the inclusion $\ker (f)\inflation X$, and an image $\im (f)$ obtained by taking the kernel of the projection $Y\deflation \coker (f)$. Hence we get a canonical map $\operatorname{coim}(f)\to \im(f)$, and it suffices to show that this is an isomorphism. To this end, let $g\colon X\deflation \operatorname{coim}(f)$ denote the given $\XX$-deflation, and let $h\colon \operatorname{coim}(f)\to Y$ denote the composite $\operatorname{coim}(f)\to \im(f)\to Y$. Consider the commutative diagram with exact rows
\[\xymatrix@C=4em{
		& \Hom_{\CC}(Y,-)\ar[r]^{\Hom_{\CC}(f,-)}\ar[d]^{\Hom_{\CC}(h,-)} & \Hom_{\CC}(X,-)\ar[r]^{}\ar@{=}[d]  & F\ar[d]^{p} \ar[r]&  0\\
		0 \ar[r] & \Hom_{\CC}(\operatorname{coim}(f),-)\ar[r]^{\Hom_{\CC}(g,-)} & \Hom_{\CC}(X,-)\ar[r]^{}  & G \ar[r] & 0.
	}\] 
Since $f$ and $g$ are $\XX$-admissible, $F$ and $G$ are in $\operatorname{Tr}(\XX)$. By the snake lemma $p$ is an epimorphism and $\coker (\Hom_{\CC}(h,-))\cong \ker (p)$. Since $\operatorname{Tr}(\XX)$ is closed under kernels of epimorphisms, $\ker (p)$ is contained in $\operatorname{Tr}(\XX)$. Hence $\Hom_{\CC}(-,h)$ is admissible in $\operatorname{Tr}(\XX)$ since $\operatorname{Tr}(\XX)$ is closed under kernels of deflations. This implies that $h$ is $\XX$-admissible, see \Cref{Remark:self-dualityXadmissible}. Let $i\colon \ker (h)\to \operatorname{coim}(f)$ denote the inclusion of the kernel, and let $j\colon E\to X$ denote the pullback of $i$ along $g$. Then we get a commutative diagram
\[\xymatrix{
			\ker(f)\ar@{=}[d]\ar@{>->}[r] & E\ar@{>->}[d]^j\ar@{->>}[r]^k & \ker(h)\ar@{>->}[d]^i\\
			\ker(f)\ar@{=}[d]\ar@{>->}[r] & X\ar@{=}[d] \ar@{->>}[r]^g& \operatorname{coim}(f)\ar[d]^h\\
			\ker(f)\ar@{>->}[r] & X\ar[r]^f &Y
		}\] 
where the two upper rows are conflations in $(\CC,\mathfrak{C})$. Hence $f\circ j=0$, so $j$ factors through $\ker (f)$. It follows that $i$ must be $0$, so $\ker (h)\cong 0$ since $i$ is an inflation in $(\CC,\mathfrak{C})$. Hence $h$ is an $\XX$-inflation. Since $h$ can be written as a composite $\operatorname{coim}(f)\to \im (f)\to Y$ and $\CC$ is weakly idempotent complete, the map $\operatorname{coim}(f)\to \im (f)$ must be an $\XX$-inflation by \cite[Proposition 7.6]{Buhler10}. Dually, one shows that the map $\operatorname{coim}(f)\to \im (f)$ is an $\XX$-deflation. Since a map in an exact category which is both an inflation and a deflation must be an isomorphism, the claim follows.
\item This follows immediately from part \eqref{Item:X-admissible morphisms}. \qedhere
\end{enumerate} 
\end{proof}

\begin{proof}[Proof of \Cref{Theorem:ExactStructuresResolving}]
If $(\CC,\mathfrak{C})$ is an exact category, then $\smodad(\mathfrak{C})$ and $\operatorname{Tr}(\smodad(\mathfrak{C}))=\smodad(\mathfrak{C}^{\operatorname{op}})$ are resolving by \Cref{Proposition:SmallestResolvingSubcategory}. Also any $F\in \smodad(\mathfrak{C})$ and $G\in \smodad(\mathfrak{C}^{\operatorname{op}})$ of grade $>0$ must satisfy 
\[
\Ext^1_{\Mod(\CC)}(F,\Hom_{\CC}(-,X))=0 \quad \text{and} \quad \Ext^1_{\Mod(\CC^{\operatorname{op}})}(G,\Hom_{\CC}(X,-))=0
\] 
for all $X\in \CC$ by \Cref{Proposition:TorsionTheoryProperties}.\eqref{Proposition:TorsionTheoryPropertiesC}. Therefore $F$ and $G$ must have grade $>1$. This shows that $\smodad(\mathfrak{C})$ satisfies \eqref{Property:1} and \eqref{Property:2} in \Cref{Theorem:ExactStructuresResolving}. The fact that the association $\mathfrak{C}\mapsto \smodad(\mathfrak{C})$ is a bijection follows from \Cref{Proposition:ResolvingGivesExactStructure}, \Cref{Proposition:ResolvingGivesExactStructure:2} and \Cref{Proposition:AdmissibleMorphisms}.
\end{proof}

We end this section by considering exact structures on an abelian category $\AA$. 

\begin{corollary}\label{Corollary:ExactStructuresResolving}
Let $\AA$ be an abelian category. Then there exists a bijection between the following:
\begin{enumerate}
\item Exact structures on $\AA$.
\item Resolving subcategory $\XX$ of $\smod (\AA)$ for which $\operatorname{Tr}(\XX)$ is a resolving subcategory of $\smod (\AA^{\operatorname{op}})$. 
\end{enumerate}
\end{corollary}

\begin{proof}
Since $\AA$ is abelian, $\smod (\AA)=\PP^2(\AA)$ and $\smod (\AA^{\operatorname{op}})=\PP^2(\AA^{\operatorname{op}})$, and $\smod (\AA)$ and $\smod (\AA^{\operatorname{op}})$ have no objects of grade $1$. Hence, the claim follows from \Cref{Theorem:ExactStructuresResolving}.
\end{proof}

\begin{example}\label{example:MainTheorem1}
Let $Q$ be the $A_2$ quiver $1 \to 2$ and let $\AA = \rep_k(Q)$ be the category of $k$-linear representations.  We use this example to illustrate \Cref{Corollary:ExactStructuresResolving}.  The Auslander-Reiten quiver of $\AA$ is
\[\begin{tikzcd}[row sep=small, column sep=small]
& P_1 \arrow[rd] \\
P_2 \arrow[ru] && I_1 
\end{tikzcd}\]
while the Auslander-Reiten quivers of $\mod(\AA)$ and $\mod(\AA^{\operatorname{op}})$ are given in \Cref{figure:FirstARQuivers}.  We have $\Tr(X) \cong X'$ and $\Tr(Y) \cong Y'$.  It is now easy to verify that $\mod(\AA)$ has two resolving subcategories $\XX$ which satisfy the conditions of \Cref{Corollary:ExactStructuresResolving}.
\begin{enumerate}
	\item The first choice is where $\XX$ is the category of projective objects of $\mod(\AA)$.  The category of effaceable objects is zero, and we find $\AA \simeq \XX$ as exact categories, that is to say, the exact sequences in $\AA$ are precisely the split exact sequences.
	\item The second choice is where $\XX = \mod(\AA)$.  Here, we find $\eff(\AA) = \add Y$.  It follows from $\AA \simeq \XX / \eff(\AA)$ that the exact structure we recover on $\AA$ is the one induced by the abelian structure, that is, all kernel-cokernel pairs are conflations.
\end{enumerate}
\end{example}

\begin{figure}
\[\begin{tikzcd}[row sep=small, column sep=small]
& \Yoneda(P_1) \arrow[rd] && \Yoneda(I_1)\arrow[rd] & && & \Yoneda(P_1)\arrow[rd] && \Yoneda(P_2)\arrow[rd] & \\
\Yoneda(P_2) \arrow[ru] && X \arrow[ru]&& Y && \Yoneda(I_1)\arrow[ru] && {Y'}\arrow[ru] && {X'} 
\end{tikzcd}\]
	\caption{The Auslander-Reiten quivers of $\mod(\rep_k(Q))$ (left) and $\mod(\rep_k(Q)^{\operatorname{op}})$ (right) from \Cref{example:MainTheorem1}.}
	\label{figure:FirstARQuivers}
\end{figure}

\begin{example}\label{example:MainTheorem2}
For the second example, we consider the bound quiver $Q$, given as $1 \xrightarrow{\alpha} 2 \xrightarrow{\beta} 3 \xrightarrow{\gamma} 4$ with relation $\beta \alpha = 0$.  Let $P_i$ be the projective indecomposable corresponding with the vertex $i$ of $Q$. Let $\EE = \add (P_1 \oplus P_2 \oplus P_3 \oplus P_4)$ be the category of projective objects in $\rep_k(Q)$.  The Auslander-Reiten quivers of $\mod(\EE)$ and $\mod(\EE^{\operatorname{op}})$ are given in \Cref{figure:SecondARQuivers}.

Here, we have $\Tr(U) = U'$, and similar notation for $V,W$, and $X$.  Note that $U'$ and $V'$ have grade one.  Again, we find two possible exact structures on $\EE$, corresponding to two resolving subcategories $\XX$ of $\mod(\EE)$ satisfying the conditions of \Cref{Theorem:ExactStructuresResolving}.
\begin{enumerate}
	\item The first case is where $\XX$ is the category of projective objects of $\mod(\EE)$.  The category of effaceable objects is zero, and we find $\EE \simeq \XX$ as exact categories, meaning that the exact sequences in $\AA$ are precisely the split exact sequences.
		\item The second choice is where $\XX = \add(P_1 \oplus P_2 \oplus P_3 \oplus P_4 \oplus W \oplus X)$.  We have $\eff(\EE) = \add X$.  In this exact structure, the only conflations in $\EE$ are direct sums of $P_1 \inflation P_2 \deflation P_3$ and split conflations.
\end{enumerate}
\end{example}

\begin{figure}
	\centering
	\[\begin{tikzcd}[row sep=tiny, column sep=tiny]
&& \Yoneda(P_2)\arrow[rd] &&& \\
& \Yoneda(P_3) \arrow[rd]\arrow[ru] && V \arrow[rd]&& \Yoneda(P_1) \arrow[rd]&  \\
\Yoneda(P_4) \arrow[ru] && U \arrow[ru] && W \arrow[ru] && X \\
&&&& \Yoneda(P_4)\arrow[rd] && \\
& \Yoneda(P_2)\arrow[rd] && \Yoneda(P_3) \arrow[rd] \arrow[ru]&& V' \arrow[rd]\\
\Yoneda(P_1)  \arrow[ru]&& X'  \arrow[ru]&& W'  \arrow[ru]&& U'
\end{tikzcd}\]
	\caption{The Auslander-Reiten quivers of $\mod(\EE)$ and $\mod(\EE^{\operatorname{op}})$ from \Cref{example:MainTheorem2}.}
	\label{figure:SecondARQuivers}
\end{figure}

Let $\ic{\EE}$ be the idempotent completion of an additive category $\EE$.  If the category $\EE$ has an exact structure $\mathfrak{C}$, then $\ic\EE$ naturally has the structure of an exact category as well: the conflations in $\ic{\mathfrak{C}}$ are direct summands of conflations in $\EE$ (see \cite[Proposition~6.13]{Buhler10}).

It might be tempting to apply \Cref{Theorem:ExactStructuresResolving} to a category $\EE$ which is not idempotent complete by first passing to the idempotent completion.  The following proposition describes what can be expected. 

\begin{proposition}
Let $\EE$ be an additive category.  The correspondence $(\EE, \mathfrak{C}) \mapsto (\ic\EE, \ic{\mathfrak{C}})$ is an injection.
\end{proposition}

\begin{proof}
Let $\mathfrak{C}$ and $\mathfrak{C}'$ be exact structures on $\EE$ satisfying $\ic{\mathfrak{C}} = \ic{\mathfrak{C}}'$.  Let $\fe\colon A \inflation B \deflation C$ be a conflation in $\mathfrak{C}$.  As $\ic{\mathfrak{C}} = \ic{\mathfrak{C}}'$, we know that there is a conflation $\fe'\colon A' \inflation B' \deflation C'$ in $(\ic\EE, \ic{\mathfrak{C}}) = (\ic\EE, \ic{\mathfrak{C}}')$ such that $\fe \oplus \fe' \in \mathfrak{C}'$.  It now follows from \cite[Proposition~2.6]{HenrardVanRoosmalen19b} that $\fe \in \mathfrak{C}'$.  This shows that $\mathfrak{C} \subseteq \mathfrak{C}'$.  The other inclusion can be shown in a similar fashion.
%
%
\end{proof}

In general, the completion $\ic\EE$ can admit strictly more exact structures than $\EE$ (see also \cite[Proposition~3.3]{Crivei12}), even when $\EE$ is weakly idempotent complete.

\begin{example}
Let $Q$ be the $A_2$ quiver $1 \to 2$ and let $\AA = \rep_k(Q)$ be the category of $k$-linear representations.  Let $\EE$ be the full subcategory of $\AA$ consisting of objects not isomorphic to the projective-injective indecomposable $P_1$ (see \Cref{example:MainTheorem1}).  As $\ic\EE \simeq \AA$, we know that there are two exact structures on $\ic\EE$: we have $\mathfrak{C}_\text{all}$ coming from the abelian structure on $\AA$, and we have $\mathfrak{C}_\text{split}$ consisting only of the split kernel-cokernel pairs.

The embedding $\EE \to \ic\EE$ of an additive category into its idempotent completion preserves and reflects pullbacks and pushouts (and hence also kernels and cokernels, see \cite[Lemma~2.2]{Crivei12}).  Hence, we know that every kernel-cokernel pair in $\EE$ is of the form: direct sums of split kernel-cokernel pairs and copies of $\fe\colon P_2 \inflation P_1 \deflation I_1$.  Let $\ff\colon X \inflation Y \deflation Z$ be any non-split kernel-cokernel pair in $\EE$.  Necessarily, we can find a retract of the form $\fe$.  This gives a coretraction $p\colon I_1 \to Z$ and a retraction $i\colon X \to P_2$.  Taking the pullback of $\fe$ along $p$ in $\ic\EE$, we find a kernel-cokernel pair $\ff p\colon X \inflation Y' \deflation I_1$; subsequently taking the pushout along $i$ in $\ic\EE$, we recover the kernel-cokernel pair $i(\ff p) = \fe\colon P_2 \inflation P_1 \deflation I_1$.  As this last kernel-cokernel pair does not lie in $\EE$, we find that $\ff$ cannot be a conflation in any exact structure on $\EE$.  This shows that the maximal exact structure on $\EE$ is the split exact structure.  Hence, the correspondence $(\EE, \mathfrak{C}) \mapsto (\ic\EE, \ic{\mathfrak{C}})$ is not surjective.
\end{example}

\begin{example}
Let now $\AA = \rep_k(Q) \oplus (\mod k)$, where $Q$ is as above.  The abelian category $\AA$ thus has, up to isomorphism, four indecomposable objects: $P_1, P_2, I_1$, and $S$, where $S$ is a simple object from $\mod k.$  Let $\EE$ be the additive category generated by $P_2, I_1$, and $S \oplus P_1$.  We remark that $\EE$ is weakly idempotent complete and that $\AA \simeq \ic\EE$.

Let $\mathfrak{C}_\AA$ be the exact structure coming from the abelian structure of $\ic\EE \simeq \AA$.  We verify that $\mathfrak{C}_\AA$ is not of the form $\ic{\mathfrak{C}}$ for any exact structure $\mathfrak{C}$ on $\EE$.  It follows from \cite[Lemma~A.14]{HenrardVanRoosmalen19b} that $\EE$ lies extension-closed in $(\ic\EE, \ic{\mathfrak{C}})$, for any exact structure $\mathfrak{C}$ on $\EE$.  However, the conflation $P_2 \inflation P_1 \deflation I_1$ in $\mathfrak{C}_\AA$ shows that $\EE$ does not lie extension-closed in $(\ic\EE, \mathfrak{C}_\AA)$, hence $\mathfrak{C}_\AA \neq \ic{\mathfrak{C}}$, for any exact structure $\mathfrak{C}$ on $\EE$.
\end{example}


\bibliographystyle{amsplain}

\providecommand{\bysame}{\leavevmode\hbox to3em{\hrulefill}\thinspace}
\providecommand{\MR}{\relax\ifhmode\unskip\space\fi MR }
\providecommand{\MRhref}[2]{%
  \href{http://www.ams.org/mathscinet-getitem?mr=#1}{#2}
}
\providecommand{\href}[2]{#2}

\end{document}